\documentclass{article}
\usepackage[margin=1in]{geometry}

\usepackage{graphicx}
\usepackage{amsmath}
\usepackage{amsfonts}
\usepackage{amssymb}
\usepackage{amsthm}
\usepackage{subfigure}
\usepackage{cite}
\usepackage{hyperref}

\newtheorem{theorem}{Theorem}[section]
\newtheorem{proposition}[theorem]{Proposition}
\newtheorem{lemma}[theorem]{Lemma}
\newtheorem{corollary}[theorem]{Corollary}
\newtheorem{conjecture}{Conjecture}

\theoremstyle{definition}
\newtheorem{definition}[theorem]{Definition}

\theoremstyle{remark}

\newcommand{\norm}[1]{\left\lVert{#1}\right\rVert}

\newcommand{\TheTitle}{Consensus and Information Cascades in Game-Theoretic Imitation Dynamics with Static and Dynamic Network Topologies}
\newcommand{\TheAuthors}{C. Griffin, S. Rajtmajer, A. Squicciarini and A. Belmonte}

\ifpdf
\hypersetup{
  pdftitle={\TheTitle},
  pdfauthor={\TheAuthors}
}
\fi

\title{{\TheTitle}
\thanks{Portions of Dr. Griffin's, Dr. Squicciarini's and Dr. Rajtmajer's work were supported by the Army Research Office under grant W911NF-13-1-0271. Portions of Dr. Griffin's and Dr. Belmonte's work were supported by the National Science Foundation under grant number CMMI-1463482.
}}

\author{
Christopher Griffin\footnote{C. Griffin is with the Applied Research Laboratory, The Pennsylvania State University, University Park, PA, 16802,  \texttt{griffinch@ieee.org}}, \and
Sarah Rajtmajer\footnote{S. Rajtmajer and A. Squicciarini are with the College of Information Science and Technology,
The Pennsylvania State University, University Park, PA, 16802, \texttt{srajtmajer@ist.psu.edu}, \texttt{asquicciarini@ist.psu.edu}} \and
Anna Squicciarini\footnotemark[3] \and
Andrew Belmonte\footnote{A. Belmonte is with the Department of Mathematics, The Pennsylvania State University, University Park, PA, 16802, \texttt{alb18@psu.edu}}
}

\begin{document}

\maketitle

\begin{abstract}
We construct a model of strategic imitation in an arbitrary network of players who interact through an additive game. Assuming a discrete time update, we show a condition under which the resulting difference equations converge to consensus. Two conjectures on general convergence are also discussed. We then consider the case where players not only may choose their strategies, but also affect their local topology. We show that for prisoner's dilemma, the graph structure converges to a set of disconnected cliques and strategic consensus occurs in each clique. Several examples from various matrix games are provided. A variation of the model is then used to create a simple model for the spreading of trends, or information cascades in (e.g., social) networks. We provide theoretical and empirical results on the trend-spreading model. 

\vspace*{1em}

\textbf{AMS Subj. Class.} 37N40, 91A43, 39A30
\end{abstract}

%

\section{Introduction}

There are many social, biological, and physical systems in which a number of discrete individuals adjust an internal variable based on mutual interactions, leading the group 
to converge towards some sort of \textit{consensus} (see \cite{MT14} and its references). Examples such as social convention or social consensus \cite{DeGroot74,Krause00,Centola15}, flocking \cite{TT98,CS07}, swarming \cite{EK01,L08,LX10, DM11}, and other collective motion \cite{MT14,VZ12,OFM07} have been extensively studied recently, with some modelling approaches treating the dynamics of opinions \cite{DeGroot74,Krause00,HK02,BN05,WDA05,Toscani06,Weisb06,Lorenz07,BHT09,CFL09,KR11,DMPW12,CFT12,JM14} or flocks \cite{CS07,HT08,HL09,CFRT10,MOA10,Hask13,EHS16} as examples of self-organized behavior \cite{VZ12,MT14}. These models depend sensitively on the connectedness of the individuals and many of these models have been considered in the context of networks, for instance in social systems  \cite{DY00,MSC01,OM04,J08,H10,DdGL10,BHT13} and natural phenomena (see e.g., \cite{BCCC08,HH08,CCGP10}). Krause's original consensus model \cite{Krause00} has been studied in networks, in particular Olfati-Saber Murray \cite{OM04}, Blondel et al. \cite{BHOT05,BHT09}, and by Canuto et al. \cite{CFT12}, who investigate the network effect. In a similar spirit, consensus with communications constraints is investigated in \cite{CFSZ08}. Algorithms for consensus are considered in \cite{CS11,DdG09,FZ08,RSGK16}. Recent work by Proskurnikov et al. \cite{PMC16} studies the opinion consensus problem with hostile camps. Distributed consensus in a stochastic setting is studied in \cite{WTCZ16} and in second-order multi-agent systems in \cite{MRC16}. Most recently, Cucker and Dong \cite{CD19} study flocking under switching topologies with directed interaction, which is somewhat similar to the topological model in this paper.

Among the models for interactions in these systems, evolutionary game theory offers a conveniently adjustable and straightforward model for well-characterized strategic interactions, which include aspects such as sub-optimal stable equilibria and multiple equilibria \cite{Hofbauer98}. Work in theoretical biology has begun to use evolutionary games on graphs in similar ways to understand network topologies for which evolutionary stability can be expected \cite{ON08,ON06} and develop variations on the replicator dynamic (see e.g., \cite{Wei95,EGB16}). Hussein \cite{H09} investigated a similar problem for generic network social behaviors, while Pantoja and Quijano \cite{PQ12} investigate a distributed optimization problem on a network with the replicator. We note that recent work by Madeo and Mocenni \cite{MM15} has developed a general replicator dynamic on graph structures, extending previous results \cite{ON08,ON06}. A result most closely related to this paper is found in \cite{GTBS16}, which studies convergence of best-response strategies on graphs.

In addition to the work on (uncontrolled) consensus and flocking, there has been substantial work on control and stability in models like these. Jadbabaie et al. studied coordination of mobile agents using flocking rules \cite{JLM03}, and Blondel et al. study consensus and flocking from a control theoretic point of view  \cite{BHOT05,BHT09}. Justh and Krishnaprasad \cite{JK10} studied extremal collective behaviors from the point of view of geometric optimal control. Li \cite{L08} considered the stability of swarms, while Olfati-Saber et al. studied the problem of controlled consensus in a topologically dynamic network \cite{OFM07}. 
Motsch and Tadmor \cite{MT14} propose a unifying model for consensus systems as defined elsewhere. There model follows the dynamic:
\begin{equation}
\frac{d\mathbf{p}^i}{dt} = \sum_{j \neq i} \frac{1}{\sigma_i}\phi\left(\left\lvert \mathbf{x}^j - \mathbf{x}^i\right\rvert\right)\left(\mathbf{p}^j - \mathbf{p}^i\right)
\label{eqn:Unify}
\end{equation}
where $\mathbf{x}^i$ is the state of Agent $i$ and $\mathbf{p}$ is either the state or its first time derivative (as is the case for the Cucker-Smale model \cite{CS07}). For this model, the evolution of $\mathbf{x}^i$ is controlled by the evolution of $\mathbf{p}^i$. This model subsumes the discrete-time work on opinion formation (e.g., \cite{DeGroot74,Krause00,BHT09,JM14}) by assuming a continuum limit on the discrete time-step. Here $|\cdot|$ is an appropriate metric and $\phi$ is the \textit{influence function}. Notice the (general) symmetry of the dynamics, which simplifies some of the analysis. In particular, \cite{MT14} provides a proof of consensus in non-symmetric opinion dynamics, where $\phi$ is replaced by $\phi_{ij}$. This work is continued in \cite{JM14}, which also studies non-symmetric opinion dynamics.

In this paper, we define a discrete time proportional imitation dynamic on a graph structure. Players reside on the vertices of the graph and are given an initial mixed strategy. Payoffs are additive and the result of interaction with neighbors. Players proportionally mimic neighbors who out-perform them. The imitation rate is a parameter of the model. The major results of the paper are:
\begin{enumerate}
\item Using an auxiliary structure called an imitation graph, we show a sufficient condition under which these dynamics converge to consensus.
\item For the case of the generalized Prisoner's Dilemma game (a game with a strictly dominant strategy), we show that when topological changes are allowed in the graph structure, the graph structure converges to a set of disconnected cliques and is pairwise stable.
\item Finally, we show how these dynamics can be adapted to model trends in (social) networks and study the qualitative properties of trends theoretically and empirically.
\end{enumerate}

This is a significant extension of work \cite{RGMS15} in which we present the basic game-theoretic model for weighted imitation and consider its evolution on static graphs, without showing a condition on strategic consensus or studying either topological changes or trends. Consensus in a changing network topology is considered in \cite{OFM07} and is related to this work in so far as we also analyze the case of an agent-controlled changing network topology. A changing topology is also considered in \cite{BGM13}, but the authors consider a specific opinion formation game consistent with the DeGroot model rather than a general matrix game.

The model considered in this paper is distinct from the unifying model in \cite{MT14} because:
\begin{enumerate}
\item the coupling strength in the evolution equation is defined by a \textit{game payoff}, rather than by an increasing (or non-negative) influence function $\phi$, and
\item the agents' state is the strategy played, rather than a physical location, opinion or velocity.
\end{enumerate}
The coupling function we investigate may not be monotonic as a function of any metric on the strategy space over which our dynamics operate. This difference leads to a distinct set of consensus criteria than those given in \cite{MT14}. As a result, our model cannot be subsumed by the unifying model in \cite{MT14}.

In addition to this difference, we use a discrete time model (following e.g., \cite{Krause00,BHT09}), both to simplify our proofs and to more accurately model systems with non-continuous observations. Our dynamics (defined in Equations \ref{eqn:Pi}-\ref{eqn:Imitation}) are discontinuous, but Lebesgue integrable. This discontinuity makes general analysis in the continuous time case difficult, while it is more tractable in the discrete time case. We discuss the nature of the discontinuity in the body of the text. We also note that Reference \cite{MT14} and its main model sources (e.g., \cite{Krause00,CS07}) do not contain discontinuous dynamics on their right-hand-sides.

Finally, as an extension of the model studied in this paper, we also investigate the formation of trends or fads \cite{BHW92}, specifically as a result of imitation (swarming/flocking) of a seed player (leader). In developing our proposed trend model, we seek to extend the notion of network influence beyond the classic epidemic models in which individuals are exposed to a contagion through interaction with a neighbor in the network and are influenced or altered by this exposure with some probability \cite{H00}. While these models have have been used successfully to explain the dynamics of many diverse processes including emotional contagion on Facebook \cite{KGH14,CSKM14} and health-related behaviors such as smoking, alcohol consumption or obesity \cite{CF13,CF07,CF08}, they typically do not capture the personal agency of member nodes in determining which peers to imitate and to what extent, or whether to change local network topology to increase payoff.

\subsection{Paper Organization}
The remainder of this paper is organized as follows: In Section \ref{sec:Prelim} we discuss preliminary mathematical structures used throughout. We also layout the core elements of the model as well as the discussion of topological evolution. In Section \ref{sec:convergence} we study the problems of consensus and convergence of the dynamics in both static and dynamic graph topologies. Sufficient conditions for strategy consensus are provided in the static topology case. Several numerical examples are provided. In Section \ref{sec:Trend}, we extend the dynamics to model trends in networks and show how two parameters govern trend-spread characteristics, just as the contact and recovery rates govern epidemic dynamics in classical epidemic models \cite{H00}. Empirical analysis on several scale-free networks is used in an example case to justify portions of the evaluation. Finally we present conclusions and future directions in Section \ref{sec:Conclusion}.

\section{Preliminaries}\label{sec:Prelim}

We develop a model of behavioral evolution for players on a connected graph $G = (V,E)$ on $n$ vertices (nodes). When unclear, $V = V(G)$ denotes the vertex set of graph $G$ and $E = E(G)$ denotes the edge set.

Let $\mathbf{A} \in \mathbb{R}^{m \times m}$ be the row-player payoff matrix in a symmetric two-player game. The payoff matrix is \textit{arbitrary} unless otherwise stated.

The unit simplex in $\mathbb{R}^m$ is:
\begin{displaymath}
\Delta_m = \left\{\mathbf{x}\in\mathbb{R}^m : \sum_i x_i = 1, x_i \geq 0\right\}
\end{displaymath}
If $i \in V$, let $\mathbf{x}^i(t) \in \Delta_m$ be the mixed strategy vector associated to Player $i$ at time $t$.

We denote the standard inner (dot) product by $\langle{\cdot,\cdot}\rangle$. If $\mathbf{x}^1,\mathbf{x}^2 \in \Delta_m$, then the expected payoff to Player 1 is $\langle{\mathbf{x}^1,\mathbf{A}\mathbf{x}^2}\rangle$.

In a network setting, the total payoff to Player $i$ summed over neighbors is:
\begin{equation}
P^i(\mathbf{x}) = \sum_{j \in N(i)} \left\langle{\mathbf{x}^i,\mathbf{A}\mathbf{x}^j}\right\rangle
\label{eqn:Pi}
\end{equation}
where $N(i)$ is the neighborhood of $i$ in $G$.
Define:
\begin{equation}
S^i(\mathbf{x}) = \sum_{k \in N(i)}w_{ik}\left\lfloor{P^k(\mathbf{x}) - P^i(\mathbf{x})}\right\rfloor_0
\label{eqn:S}
\end{equation}
and
\begin{equation}
\kappa_{ij}(\mathbf{x}) =
\begin{cases}
\frac{w_{ij}\left\lfloor{P^j(\mathbf{x}) - P^i(\mathbf{x})}\right\rfloor_0}{S^i(\mathbf{x})} & \text{if $S^i(\mathbf{x}) > 0$} \\
0 & \text{otherwise}
\end{cases}
\label{eqn:kappa}
\end{equation}
Here $\lfloor{\cdot}\rfloor_0$ denotes $\max\{0,\cdot,\}$. The terms $w_{ik} \geq 0$ denote a preference weighting given to Player $k$ by Player $i$. If all neighbors are considered equally, then $w_{ik} = 1$ for all $k$. In the sequel, we require $w_{ik} = w_{ki}$; i.e., that weighting is symmetric. Since the payoff function will asymmetrically modify the effects of the weights, this is not a substantial loss of generality.

Note that $\sum_{j}\kappa_{ij}(\mathbf{x}) = 1$, just in case $S^i(\mathbf{x}) > 0$. Let:
\begin{equation}
f^i(\mathbf{x}) = \sum_{j \in N(i)} \kappa_{ij}(\mathbf{x})(\mathbf{x}^j - \mathbf{x}^i).
\label{eqn:fi}
\end{equation}
This simplifies to
\begin{equation}
f^i(\mathbf{x}) = \begin{cases}
\sum_{j \in N(i)} \kappa_{ij}(\mathbf{x})\mathbf{x}^j - \mathbf{x}^i & \text{if $S_i(\mathbf{x}) > 0$}\\
0 & \text{otherwise}
\end{cases}
\label{eqn:master}
\end{equation}

The strategy update rule is given by
\begin{equation}
\mathbf{x}^i(t+1) = \mathbf{x}^i(t) + \alpha f^i(\mathbf{x})
\label{eqn:Imitation}
\end{equation}
Here, $\alpha \in (0,1)$ is an adaptation (learning) rate. This rule is a \textit{proportional success mimicking rule} in which players will drift toward (imitate) successful behaviors, where $\alpha$ is essentially a rate of imitation. Letting $\alpha \rightarrow 0$ on the right-hand-side while replacing $t+1$ by $t+\alpha$ on the left-hand-side, one can obtain the continuous time dynamics,
which we do not consider in this paper, but discuss in future directions. Note the right-hand-side of the dynamics are discontinuous (but Lebesgue integrable) as a result of the definition of $\kappa$. As noted, papers like \cite{MT14,CS07} consider continuous dynamics for opinion and flocking by taking this limit with the unifying model given in Equation \ref{eqn:Unify}. Neither the derived continuous dynamics nor the discrete dynamics we study fit the model in \cite{MT14} because $\kappa$ need not be symmetric and is in general is not a function of the metric $\norm{\mathbf{x}^j - \mathbf{x}^i}$. 

It is worth noting that the results in the paper are not affected if we replace  $\lfloor{\cdot}\rfloor_0$ with $\lfloor{\cdot}\rfloor_\delta$ for some $\delta > 0$. This allows us to model cases where a player is only interested in imitating another player if she outperforms him by a sufficient amount. In particular, $\delta$ could vary for each player, vary with time, etc.

Finally, we note that the dynamics in Equation \ref{eqn:Imitation} will remain in $\Delta_m$ if initialized in $\Delta_m$. In fact, as we show in the sequel, the dynamics are weakly contracting on the convex hull of the set of strategies in $\Delta^n_m = \Delta_m \times \Delta_m \times \cdots \times \Delta_m$, the strategy space of all players (see Corollary \ref{prop:hull}).

\subsection{Topological Evolution}\label{sec:TopEvRules}
If players may update their local topology, we assume there is a second time-scale at work. After a fixed number of discrete strategy updates $\tau \gg 1$, each player may unilaterally delete an edge with another player \textit{or} add an edge with another player if the other player agrees. This changes the structure of the graph underlying the dynamics. In the case when a topological update and strategy update coincide in time, we assume the strategy update occurs first. The impact of the order of operations is most significant when $\alpha$ in Equation \ref{eqn:Imitation} is large. For small $\alpha$, the order is less important. In this case, we assume that topological updates are infrequent. In future research, one could study the problem of both a fast cycle of strategy updates combined with a fast cycle of topological structure updates. We focus on infrequent topology updates to better study the effect of multiple time scales.

Let $G' = (V',E')$ be an alternative graph and denote $P'_i(\tau)$ as Player $i$'s payoff at time $\tau$ (the epoch at which structural change may occur). We assume that $G'$ is preferred to $G$ by Player $i$, written $G' \succ_i G$, if and only if $P'_i(\tau) > P_i(\tau)$. In this way, players may make locally optimal decisions about edge additions and deletions. Let $e = \{i,j\}$:
\begin{enumerate}
\item If $e \in E$ and $G' = G - e$ (edge $e$ is removed from $G$) and either $G' \succ_i G$ or $G' \succ_j G$, then $G \rightarrow G'$.
\item If $e \not \in E$ and $G = G + e$ (edge $e$ is added to $G$) and both $G' \succ_i G$ or $G' \succ_j G$, then $G \rightarrow G'$.
\end{enumerate}
Specifically, any player can unilaterally remove an adjacent edge if it improves his/her payoff but an edge can only be added if both players agree it will improve both their payoffs. To remove any ambiguity about zero improvement cases, we assume that if removing edge $e$ neither increases nor decreases a player's payoff, then the player will \textit{remove} this zero-value edge. Similarly, an edge with zero value will never be added. This condition is important in the proof of Proposition \ref{prop:zerosum}.

\begin{definition}[Pairwise Stability] A graph is \textit{pairwise stable} \cite{J08,JV05,J03,JW96} at $k\tau$ (for $k \in \mathbb{Z}_+$) if no topological update occurs at $k\tau$ and simply \textit{pairwise stable after $K$ steps} if for all $k \geq K$ the graph is \textit{pairwise stable} at $k\tau$.
\end{definition}

\section{Consensus in Static and Dynamic Topologies}\label{sec:convergence}
As in the models for flocking \cite{CS07} or opinion dynamics \cite{DeGroot74,Krause00}, convergence of strategies can occur either to a single consensus strategy or to clusters of strategies. In this section we provide preliminary results and then provide conditions for strategy consensus in static and dynamic topologies. A conjecture on general convergence is provided and left as an open question. 

\subsection{Preliminary Results on Imitation Dynamics}
\begin{definition}[Imitation Graph] Let $\mathbf{x} = (\mathbf{x}^1,\dots,\mathbf{x}^n)$ be a point in $\Delta_m^n$. The \textit{imitation graph} $I_G$ is the graph with vertex set $V$ and a directed edge $(v_i,v_j)$ just in case the undirected edge $\{v_i,v_j\} \in E$ and $P_i(\mathbf{x}) < P_j(\mathbf{x})$.
\end{definition}

Given any connected $G$ and payoff matrix $\mathbf{A}$, the imitation graph $I_G$ is a directed acyclic graph. 
The following vertex ordering lemma is Proposition 2.1.3 of \cite{BJG09}.
\begin{lemma} For any imitation graph $I_G$, there is an ordering of the vertices in $G$ so that $(v_i,v_j) \in I_G(t)$ if and only if $i < j$. \hfill\qed
\label{lem:Ordering}
\end{lemma}

\begin{corollary} Let $\mathbf{K}(\mathbf{x})$ be the matrix whose $(i,j)$ entry is $\kappa_{ij}(\mathbf{x})$. Then $\mathbf{K}(\mathbf{x})$ is upper-triangular, assuming we are using an ordering implied by Lemma \ref{lem:Ordering}. Furthermore, if $S^i(\mathbf{x}) = 0$, then row $i$ of $\mathbf{K}(\mathbf{x})$ is the zero row-vector. Otherwise, the row sum is $1$.
\hfill\qed
\end{corollary}
Throughout the remainder of this paper, we refer to the vertex in $I_G$ corresponding to Player $i$ as $v_i$ and the vertex in $G$ corresponding to the Player $i$ simply as $i$. We also note that $\mathbf{K}(\mathbf{x})$ is also necessarily a time varying function of $t$ whenever $\mathbf{x}(t)$ is changing as a result of Equation \ref{eqn:Imitation}.

Along any trajectory of Equation \ref{eqn:Imitation}, the imitation graph will change as $\kappa_{ij}(\mathbf{x})$ changes and thus $\mathbf{K}(\mathbf{x})$ changes (see Expression \ref{eqn:kappa}). As we see in the sequel, the dynamics of the discrete dynamical system $I_G(t)$ are coupled to the dynamics of Equation \ref{eqn:master}. Without loss of generality, we assume the vertex ordering of Lemma \ref{lem:Ordering} changes accordingly in time.

We'll say that the imitation graph converges to $I_G^*$ if there is some $T \geq 0$ so that for all \textit{finite} $t > T$, $I_G(t) = I_G^*$. This does not rule out the possibility that $\lim_{t \rightarrow \infty} I_G(t) \neq I_G^*$, which could occur if we have vertices $v_i$ and $v_j$ so that $\lim_{t\rightarrow \infty} P_i(t) = \lim_{t\rightarrow \infty} P_j(t)$ but (e.g.) $P_i(t) < P_j(t)$ for all finite $t > T$. In this case, $i > j$ in the vertex ordering of $I^*_G$ and $(v_i,v_j) \in E(I^*_G)$, but $(v_i,v_j) \not\in E(\lim_{t\rightarrow\infty}I_G(t))$.

Let $\mathcal{I}$ be the set of all possible imitation graphs. There is an equivalence relation $\sim_{\mathcal{I}}$ on $\Delta_m^n$ (the strategy space of all players) so that $\mathbf{z}\sim_{\mathcal{I}}\mathbf{z}'$ with $\mathbf{z},\mathbf{z}' \in \Delta_m^n$ if and only if the imitation graphs generated by $\mathbf{z}$ and $\mathbf{z}'$ are identical (not just isomorphic). Thus, $\Delta_m^n$ is partitioned by this equivalence relation.

\subsection{Main Lemma}
We provide a preliminary result that we use in our proof of the main consensus theorem (Theorem \ref{thm:consensus}). 

\begin{lemma} Let $\mathbf{Q}(t)$ be a sequence of $n \times n$ upper-triangular stochastic matrices with the following form:
\begin{equation}
  \mathbf{Q}(t) =
  \begin{bmatrix}
    (1-\alpha) & \alpha \kappa_{11}(t) & \alpha \kappa_{12}(t) & \cdots & \alpha \kappa_{1n}(t)\\
    0 & (1 - \alpha) & \alpha\kappa_{23}(t) & \cdots & \alpha\kappa_{2n}(t)\\
    \vdots & \vdots & \vdots &\ddots & \vdots\\
    0 & 0 & 0 & \cdots & 1
  \end{bmatrix}
  \label{eqn:QMatrix}
\end{equation}
Here $\alpha \in (0,1)$ and for all $t$:
\begin{displaymath}
  \sum_j \kappa_{ij}(t) = 1
\end{displaymath}
with $\kappa_{ij}(t) \geq 0$ for all $(i,j)$. Moreover, we assume that  $\kappa_{ij}(1) = 0$ if and only if $\kappa_{ij}(t) = 0$ for all $t$. Thus zero-pattern of the matrices remains constant and the value at position $(n,n)$ in $\mathbf{Q}(t)$ is always $1$ for all $t$. Let $\mathbf{x} \in [0,1]^n$ and let:
\begin{equation}
  Q_t(\mathbf{x}) = \mathbf{Q}(t) \cdot\mathbf{x}
\end{equation}
If $\mathbf{x}_0 = \left\langle{x_{1}^{0},\dots,x_n^{0}}\right\rangle$, then
\begin{displaymath}
  \lim_{t \to \infty} (Q_t\circ Q_{t-1} \circ \cdots \circ Q_1)(\mathbf{x}_0) = \left\langle{x_{n}^{0},\dots,x_n^{0}}\right\rangle
\end{displaymath}
and this is a consensus point. This means that all strategies are eventually pulled to the strategy of Player $n$, which $x_n^0$. 
\label{lem:ConsensusLemma}
\end{lemma}
\begin{proof} The proof will show convergence of the matrix product:
\begin{displaymath}
\lim_{m \to \infty}\left(\prod_{t=1}^m\mathbf{Q}(t)\right)\mathbf{x}_0
\end{displaymath}
Let $\mathbf{x} = \langle{x_1,\dots,x_n}\rangle$ be an arbitrary (strategy) vector. Decompose $\mathbf{x}$ so that:
\begin{displaymath}
\mathbf{x} = \begin{bmatrix} x_1 - x_n \\ \vdots\\ x_{n-1} - x_n \\ 0\end{bmatrix} + \begin{bmatrix} x_n \\ \vdots\\ x_n \\ x_n\end{bmatrix} = \mathbf{y} + x_n\mathbf{1},
\end{displaymath}
where $\mathbf{1}$ is an $n$-dimensional vector of ones. Denote $\mathbf{y}_{-} = \langle{x_1 - x_n ,\dots, x_{n-1} - x_n}\rangle \in \mathbb{R}^{n-1}$ so that:
\begin{displaymath}
\mathbf{y} = \begin{bmatrix} \mathbf{y}_{-} \\ 0\end{bmatrix}
\end{displaymath} 
From Equation \ref{eqn:QMatrix}, the matrix $\mathbf{Q}(t)$ can be written as:
\begin{displaymath}
\mathbf{Q}(t) = (1-\alpha)\mathbf{I}_n + \alpha\mathbf{K}(t)
\end{displaymath}
and it follows that:
\begin{displaymath}
\mathbf{Q}(t)\mathbf{x} = \mathbf{Q}(t)\left(x_n\mathbf{1} + \mathbf{y}\right) = x_n\mathbf{Q}(t)\mathbf{1} + \mathbf{Q}(t)\mathbf{y} = x_n\mathbf{1} + \mathbf{Q}(t)\mathbf{y}
\end{displaymath}
because $\mathbf{Q}(t)$ is row stochastic and so $\mathbf{Q}(t)\mathbf{1} = \mathbf{1}$. The product $\mathbf{Q}(t)\mathbf{y}$ can be decomposed so that:
\begin{displaymath}
\mathbf{Q}(t)\mathbf{y} = \begin{bmatrix} \mathbf{Q}_{-}(t) & \mathbf{c}\\ \mathbf{0}^T & s
\end{bmatrix}\begin{bmatrix}\mathbf{y}_{-}\\0\end{bmatrix},
\end{displaymath}
where $\mathbf{c}$ is a column whose value is irrelevant and $\mathbf{Q}_(t)$ is an $(n-1) \times (n-1)$ matrix. Then:
\begin{displaymath}
\mathbf{Q}(t)\mathbf{x} = \begin{bmatrix}\mathbf{Q}_{-}(t)\mathbf{y}_{-}\\0\end{bmatrix} + x_n\mathbf{1}
\end{displaymath}
It now suffices to show that:
\begin{equation}
\lim_{m \to \infty} \prod_{t=1}^m\mathbf{Q}_{-}(t) = \mathbf{0}
\label{eqn:MatrixSuffCond}
\end{equation}
Note that:
\begin{displaymath}
\mathbf{Q}_{-}(t) = (1-\alpha)\mathbf{I}_{n-1} + \alpha\mathbf{K}_{-}(t),
\end{displaymath}
where $\mathbf{K}_{-}(t),$ is a submatrix of $\mathbf{K}(t)$ composed of the first $n-1$ columns and rows. Note that the induced $\ell_\infty$ norm of $\mathbf{K}_{-}(t)$ is bounded above by $1$ (because $\mathbf{K}_{-}$ is substochastic); i.e., $\norm{\mathbf{K}_{-}(t)}_\infty \leq 1$. Further note that for any $t$, the last row and diagonal of $\mathbf{K}_{-}(t)$ are zero and therefore $\mathbf{K}_{-}(t)$ is nilpotent. We will use this fact in the remainder of the proof. For nilpotent matrices of the form of $\mathbf{K}_{-}(t)$, it can be shown that the product of any $n$ of them must be the zero matrix. That is, if $\mathbb{Z}_+$ is the possible time values and $S \subset \mathbb{Z}_+$ with $|S| = n$, then:
\begin{displaymath}
\prod_{t \in S} \mathbf{K}_{-}(t) \equiv \mathbf{0}
\end{displaymath}

Let $T \in \mathbb{Z}_+$ be an arbitrarily time with $T > n$. Let $\mathcal{S}_{T,n}$ be $\binom{T}{n}$ subset of $\{1,\dots,T\}$ with cardinality $n$. That is, if $S \in \mathcal{S}_{T,n}$, then $S \subset \{1,\dots,T\}$ and $|S| = n$. Expanding we have:
\begin{displaymath}
\prod_{t=1}^T \mathbf{Q}_{-}(t) = \prod_{t=1}^T\left[(1-\alpha) \mathbf{I} + \alpha\mathbf{K}_{-}(t) \right] = 
\sum_{k=0}^T\sum_{S \in \mathcal{S}_{T,k}} (1-\alpha)^{T-k}\alpha^k\prod_{t\in S}\mathbf{K}_{-}(t)
\end{displaymath}
Applying the facts that $\norm{\mathbf{K}_{-}(t)}_\infty \leq 1$ and that products of (at most) $n$ nilpotent matrices like $\mathbf{K}_{-}(t)$ vanish yields:
\begin{displaymath}
\norm{\prod_{t=1}^T \mathbf{Q}_{-}(t)}_\infty \leq \sum_{k=0}^n\binom{T}{k}\alpha^k(1-\alpha)^{T-k}
\end{displaymath}
For fixed $\alpha$ and $n$, the right hand side goes to zero as $m \to \infty$. To see this explicitly, we can compute a bound on the right hand side. Algebraic manipulation yields:
\begin{displaymath}
\sum_{k=0}^n\binom{T}{k}\alpha^k(1-\alpha)^{T-k} = (1-\alpha)^{T-n}
\sum_{k=0}^n\frac{T!}{(T-k)!}\frac{(n-k)!}{n!}\binom{n}{k}\alpha^k(1-\alpha)^{n-k}
\end{displaymath}
Observe that for all $k$:
\begin{displaymath}
\frac{(n-k)!}{n!} \leq 1
\end{displaymath}
and:
\begin{displaymath}
\frac{T!}{(T-k)!} \leq  \left(\frac{T}{T-n}\right)^T = \frac{1}{\left(1-\tfrac{n}{T}\right)^T}
\end{displaymath}
Substituting these bounds in yields:
\begin{equation}
\norm{\prod_{t=1}^T \mathbf{Q}_{-}(t)}_\infty \leq \frac{(1-\alpha)^{T-n}}{\left(1-\tfrac{n}{T}\right)^T},
\label{eqn:LastBound}
\end{equation}
since
\begin{displaymath}
\sum_{k=0}^n\binom{n}{k}\alpha^k(1-\alpha)^{n-k} = 1
\end{displaymath}
For $\alpha > n/T$, the right hand side of Equation \ref{eqn:LastBound} decays exponentially fast and therefore the condition given in  Equation\ref{eqn:MatrixSuffCond} holds. This completes the proof.
\end{proof}
\subsection{Static Topologies}
In the following analysis, it is convenient to consider $\mathbf{x}_k(t) \in [0,1]^n$. This is the vector of $k^\text{th}$ strategy probabilities and is an element of the unit hypercube, while $\mathbf{x}^i(t) \in \Delta_m$ is the strategy vector for Player $i$, contained in the $(m-1)$-dimensional unit simplex.

To study the problem of consensus convergence, we recast the dynamics as the product of an infinite sequence of stochastic matrices. Let us momentarily fix the structure of $I_G(t)$, the imitation graph. Assume that $\mathbf{x}_k(t)$ is ordered using the induced ordering that defines $I_G(t)$ and that the first $r$ rows of $\mathbf{K}(\mathbf{x})$ are non-zero. Define:
\begin{displaymath}
  \mathbf{Q}(\mathbf{x}) =
  \begin{bmatrix}
    (1-\alpha)\mathbf{I}_r & \mathbf{0}\\
    \mathbf{0} & \mathbf{I}_{n-r}
  \end{bmatrix} + \alpha\mathbf{K}(\mathbf{x})
\end{displaymath}
This is an upper-triangular stochastic matrix (as is $\mathbf{K}$). Moreover Equation \ref{eqn:Imitation} can be written:
\begin{displaymath}
  \mathbf{x}(t+1) = \mathbf{Q}(t)\mathbf{x}(t)
\end{displaymath}
Thus:
\begin{equation}
  \mathbf{x}(t) = \prod_{s=0}^{t}\mathbf{Q}(t-s)\mathbf{x}_0
  \label{eqn:StocMatMaster}
\end{equation}
It is worth noting that in such an equation, the order of the indices cannot be permuted continually and therefore $\mathbf{Q}(t-s)$ (for $s\in\{0,\dots,t\}$) need not be  triangular at all times, though the product formulation of the dynamics will hold. That is, $\mathbf{Q}(t)$ is a stochastic matrix that is not always upper-triangular when $I_G(t)$ varies.  

\begin{proposition} Let $\mathcal{H}(\mathbf{x})$ denote the convex hull of the set of strategies $\mathbf{x}$ for all players. Then $\mathcal{H}(\mathbf{x}(t+1)) \subseteq \mathcal{H}(\mathbf{x}(t))$.
\label{prop:hull}
\end{proposition}
\begin{proof} Each update can be computed using the stochastic matrix $\mathbf{Q}(t)$. As a result, for each coordinate $k \in \{1,\dots,m\}$, the function:
  \begin{displaymath}
    F(\mathbf{x}(t)) = \mathbf{Q}(t)\mathbf{x}(t)
  \end{displaymath}
is non-expansive in the induced infinity norm. That is:
\begin{displaymath}
  \norm{\mathbf{x}(t+1)}_\infty \leq \norm{\mathbf{x}(t)}_\infty.
\end{displaymath}
Because $F(\mathbf{x}_k)$, is taking a weighted average to compute each new component of $\mathbf{x}_k$, we also have:
\begin{displaymath}
  \min_{i\in\{1,\dots,n\}}\mathbf{x}^i_k(t+1) \geq \min_{i\in\{1,\dots,n\}}\mathbf{x}^i_k(t)
\end{displaymath}
Thus $\mathcal{H}(\mathbf{x}(t+1)) \subseteq \mathcal{H}(\mathbf{x}(t))$.
\end{proof}

Proposition \ref{prop:hull} is proved for continuous flocking and opinion dynamics in \cite{MT14} (see Proposition 2.1) using a similar approach. Since the dynamics studied in this paper do not adhere to the unifying model in \cite{MT14}, we have provided a similar proof for our dynamics. It is worth noting that this proof is simpler than the one in \cite{MT14} only because our time step is discrete.

Let
\begin{displaymath}
  F_t(\mathbf{x}_k(t)) = \mathbf{Q}(t)\mathbf{x}_k(t)
\end{displaymath}
The function $\mathbf{F}:\Delta_m^n \to \Delta_m^n$ defined as:
\begin{displaymath}
  \mathbf{F}_t(\mathbf{x}) = (F_t(\mathbf{x}_1(t)),\dots,F_t(\mathbf{x}_m(t)))
\end{displaymath}
is a non-expansive map in the infinity norm and is computed by the multiplication of the stochastic matrix $\mathbf{Q}(t)$. Figure \ref{fig:Cube} illustrates the action on the unit cube of the sequence of non-expansive maps that arise from a convergent example of the imitation dynamics played with the rock-paper-scissors payoff matrix using three players. (See Section \ref{sec:Bistable} for further examples.) This illustrates the affect of the imitation dynamics one one dimension (strategy) of the strategy space for the three players.
\begin{figure}[htbp]
  \centering
  \subfigure[$t = 0$]{\includegraphics[scale=0.35]{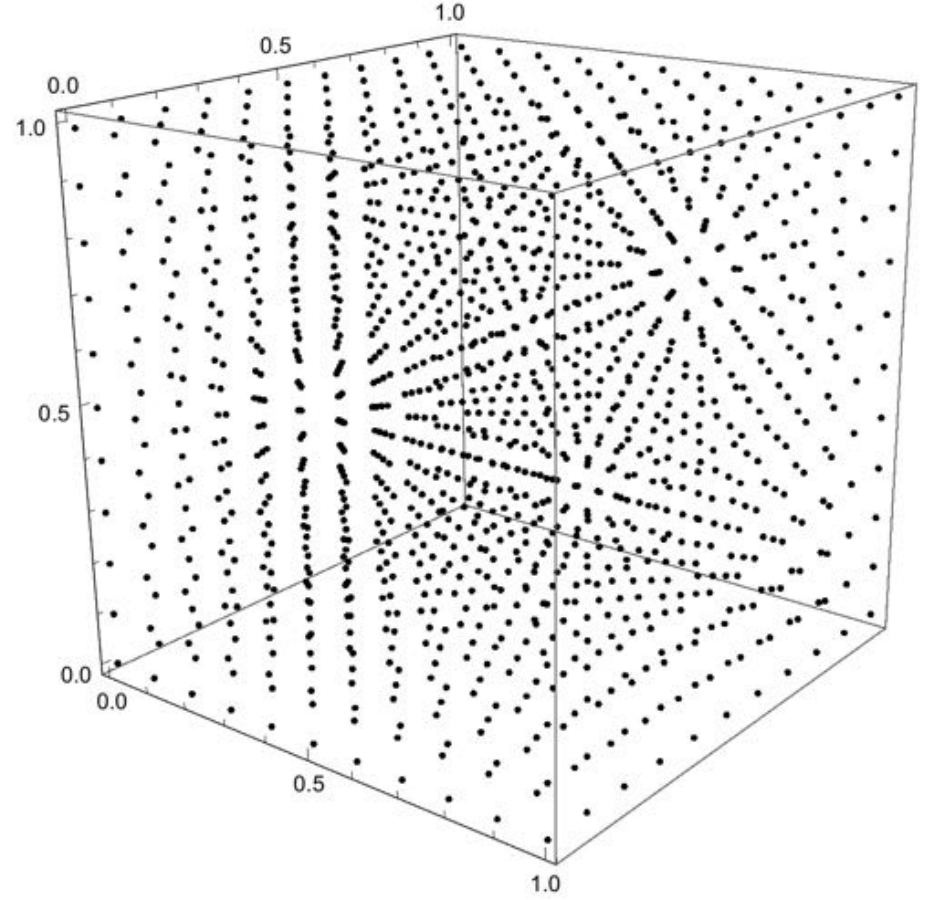}}
  \hspace*{4em}
  \subfigure[$t = 200$]{\includegraphics[scale=0.35]{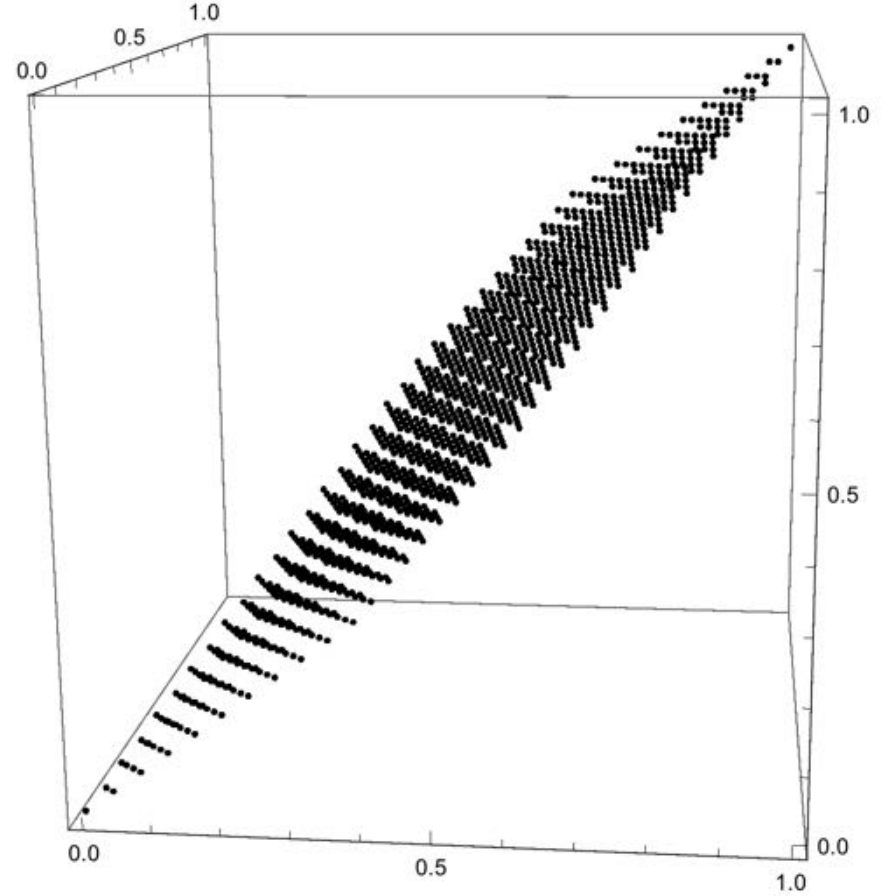}}
  \hspace*{4em}
  \subfigure[$t = 500$]{\includegraphics[scale=0.35]{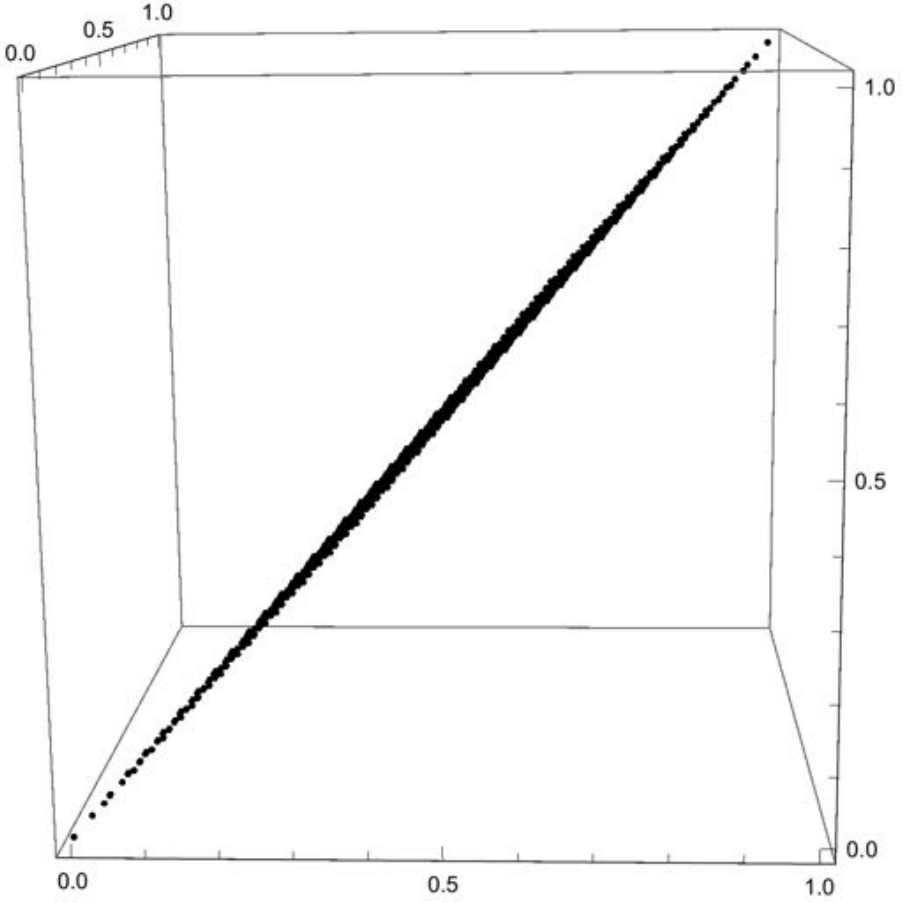}}
  \caption{The non-expansive maps arising from the imitation dynamics contract the unit cube into the one dimensional set corresponding to consensus.}
  \label{fig:Cube}
\end{figure}

Convergence of compositions of non-expansive maps has been extensively studied, starting with the work of Banach \cite{B22}. The more difficult subject of (infinite) compositions of distinct contractions or non-expansive maps continuous to be an active area of research. Williams \cite{W71} studies finite \textit{words} of contractions, inspired by an observation by Smale in \cite{S67}. Nadler provides a set of fascinating results on contractions and their fixed points in general metric spaces in \cite{N68,FN69}. Browder \cite{B65,B67} studies the convergence of non-linear compositions and in particular of non-expansive maps \cite{B65} in Hilbert space, showing that these maps have a fixed point property.  Halpern \cite{H67} worked explicitly on fixed points of non-expanding maps and investigated methods for identifying fixed points. More recent work on non-expansive maps has focused on extending Halpern's original method (see e.g., \cite{B96,J05} and their references). In addition to work in general Hilbert and Banach spaces, compositions of analytic contractions are known to have strong convergence properties \cite{Gill91,Lorentzen90}.
 We cannot prove general convergence (see Section \ref{sec:GenConv}), but we can show a sufficient condition for strategy consensus to emerge using Lemma \ref{lem:ConsensusLemma}.

\begin{theorem} Suppose $I_G(t)$ converges to the connected imitation graph $I_G^*$ by a finite time $t_0$. If there is a unique maximal vertex $v^*$ in $I_G^*$, then Equation \ref{eqn:Imitation} converges to consensus.
\label{thm:consensus}
\end{theorem}
\begin{proof} The fact that $I_G(t)$ converges to $I_G^*$ with a unique maximal vertex $v^*$ implies that for all $t > t_0$, $\mathbf{x}(t)$ and $\mathbf{Q}(t)$ can be ordered so $\mathbf{Q}(t)$ satisfies the necessary conditions of Lemma \ref{lem:ConsensusLemma} for all $t > t_0$. Specifically, the last row is the only identity row. Apply Lemma \ref{lem:ConsensusLemma} to each dimension independently. Convergence to consensus follows immediately.
\end{proof}

\begin{corollary} Suppose the maximal set of $I_G(t)$ is constant and a singleton for all time $t > t_0$. Then Equation \ref{eqn:Imitation} converges to consensus. \hfill\qed
\label{cor:ConstantSingleton}
\end{corollary}

The following corollary can be obtained by a slight modification to the proof of Lemma \ref{lem:ConsensusLemma} to take into account a changing matrix structure, but a fixed maximal strategy set.

\begin{corollary} Let $\mathcal{X}^*$ be the set of strategies of the maximal set of vertices. If $\mathcal{X}^*$ is singleton for all time $t > t_0$, then Equation \ref{eqn:Imitation} converges to consensus. \hfill\qed
\label{cor:Consenus}
\end{corollary}

Infinite matrix products, as in the proof of Theorem \ref{thm:consensus}, are used extensively in the study of consensus and distributed control, see for example \cite{LBF04,QWH05,TJ10,CXL16}.
Theorem 2.3 of \cite{MT14} uses a similar style of argument. However, the problem considered there is much more difficult because the row-stochastic matrix under consideration does not have an absorbing state.

The generic problem of the convergence of an infinite product of (stochastic) matrices is an active area of research, with the general method of analysis being to find an appropriate coefficient of ergodicity \cite{S81} and thus prove that the sequence has the \textit{scrambling property} \cite{S81}. Early work in this area (in particular for consensus) can be found in \cite{CS77} and \cite{H74}. Related results for substochastic matrices are given by Pullman in \cite{P66}. More recent work on this topic can be found in \cite{L92} and this work is nicely unified with work on non-expansive maps in \cite{W08}. Alternate work on infinite products using the joint spectral radius of a matrix family is studied in \cite{DL92,DL01}.

Unfortunately, unlike Theorem 4.3 of \cite{MT14}, which provides a simple consensus criteria for local non-symmetric opinion dynamics, no such simple criteria exists \textit{ab initio} for all games. We can, however, produce a simple sufficient condition for consensus using Proposition \ref{prop:hull}.

\begin{proposition} Suppose at time $t\geq 0$:
\begin{enumerate}
\item The set $\mathcal{X}^*$ is a singleton (i.e., all maximal vertices share a common strategy);

\item There is a directed path from each vertex $v_i$ to a maximal vertex.

\item If $V^*$ is the set of maximal vertices and $d_i$ is the degree of vertex $i$ in the original graph $G$, then for each $i \in V^*$
\begin{equation}
\min_{\substack{i\in V^*\\\mathbf{y}\in\mathcal{H}(\mathbf{x}(t))}}d_i\cdot\left\langle{\mathbf{x}^i,\mathbf{A}\mathbf{y}}\right\rangle >
\max_{\substack{j \in V\setminus V^*\\\mathbf{y},\mathbf{z}\in\mathcal{H}(\mathbf{x}(t))}} d_j\cdot\left\langle{\mathbf{y},\mathbf{A}\mathbf{z}}\right\rangle,
\label{eqn:SufficientCriteria}
\end{equation}
\end{enumerate}
then the maximal vertex set is fixed and $\mathbf{x}(t)$ converges to consensus.
\label{prop:MaxMin}
\end{proposition}
\begin{proof} From Proposition \ref{prop:hull}, we know strategies will evolve inside the convex hull of $\mathcal{H}(\mathbf{x})$. Further:
\begin{displaymath}
P_i(\mathbf{x}) = \sum_{j \in N(i)}\left\langle{\mathbf{x}^i,\mathbf{A}\mathbf{x}^j}\right\rangle \leq d_i\max_{j}\left\langle{\mathbf{x}^i,\mathbf{A}\mathbf{x}^j}\right\rangle.
\end{displaymath}
The left-hand-side of Equation \ref{eqn:SufficientCriteria} is a lower-bound on the payoff that can be obtained by a maximal vertex for all future times. It can be computed as the solution to a linear program. The right-hand-side of Equation \ref{eqn:SufficientCriteria} is an upper-bound on the payoff any other vertex can obtain for all future times. It can be computed by solving a quadratic programming problem. The sufficiency for consensus is immediate from Corollary \ref{cor:Consenus}.
\end{proof}
We provide an example of the use of this result in Section \ref{sec:Bistable} when we discuss bistable games.

\subsection{General Convergence}\label{sec:GenConv}
The problem of general convergence of Equation \ref{eqn:Imitation} is far more complex, in essence because the underlying dynamics are not strict contractions. We provide two convergence conjectures and then suggest methods by which they could be proved, but leave both as open problems. Several examples of convergence are shown in Section \ref{sec:Ex}.

\begin{conjecture}[Weak Convergence] Given a game matrix $\mathbf{A}$ and a learning rate $\alpha$, suppose the maximal set of $I_G(t)$ is constant for $t \geq t_0$. If $\alpha < \alpha^*$, (a rate to be determined), then there is an attracting fixed point of Equation \ref{eqn:Imitation} contained in $\mathcal{H}(\mathcal{X}^*)$, the convex hull of $\mathcal{X}^*$ and the dynamics converge to it.
  \label{con:1}
\end{conjecture}

\begin{conjecture}[Strong Convergence] For all game matrices $\mathbf{A}$, there is some learning rate $\alpha^*$ so that Equation \ref{eqn:Imitation} always converges when $\alpha \leq \alpha^*$.
  \label{con:2}
\end{conjecture}

In both cases, these conjectures assert than oscillation within the dynamical system is a transient function of the dynamic structure of $I_G(t)$, rather than the dynamics themselves, since it is relatively clear that if $I_G(t)$ is constant for $t \geq t_0$, then $\mathcal{H}(\mathcal{X}^*)$. The dynamics appear to contract toward $\mathcal{H}(\mathcal{X}^*)$, but whether they attract to a fixed point is the essence of Conjecture \ref{con:1}. Conjecture \ref{con:2} is much stronger.

While we do not have a formal proof, one approach to proving Conjecture \ref{con:1} is to modify the proof of Lemma \ref{lem:ConsensusLemma} and construct a result like Theorem \ref{thm:consensus}. This proof would rest on the assertion that $|\kappa_{ij}(t) - \kappa_{ij}(t+1)|$ is not too large. In turn, this is a function of the payoff functions, which are bilinear in the strategies (and hence continuous) not changing too much from one time step to the next. However, that argument is entirely controlled by the rate of change of $\mathbf{x}$ itself, which is a function of $\alpha$,
the learning rate.

To prove the more general case, we observe that Equation \ref{eqn:Imitation} resembles a gradient descent \cite{Bert99}. Specifically, consider the joint energy function:
\begin{displaymath}
  \mathcal{E}(\mathbf{x}) = \frac{1}{2}\sum_{i} \sum_{j \in N(i)}\kappa_{ij}(\mathbf{x})\norm{\mathbf{x}^{j} - \mathbf{x}^{i}}^2
\end{displaymath}
If $\mathbf{f}(\mathbf{x}) = \left\langle{f^1(\mathbf{x}),\dots,f^n(\mathbf{x})}\right\rangle$ is the update step, then Equation \ref{eqn:Imitation} is just:
\begin{displaymath}
  \mathbf{x}(t+1) = \mathbf{x}(t) + \alpha \mathbf{f}(\mathbf{x})
\end{displaymath}
If one can show there is a $t_0$ so that:
\begin{displaymath}
  \left\langle{\nabla\mathcal{E},\mathbf{f}}\right\rangle < 0,
\end{displaymath}
for all time $t \geq t_0$, then $\mathbf{f}$ is a descent step and consequently, there is a $\alpha^*$ (given by the Armijo rule) so that if $\alpha < \alpha^*$, then Equation \ref{eqn:Imitation} will converge to a stationary point \cite{Bert99}. The challenge with this approach is that while $I_G(t)$ is changing, the gradient $\nabla \mathcal{E}$ may not exist because $\kappa_{ij}$ is discontinuous. Thus, a more sophisticated analysis method may be required (e.g. a Gateaux Derivative or subgradient method). Proving or finding counter-examples for both of these conjectures is part of our planned future work.

\subsection{Results on Special Games with Dynamic Topologies}
\label{sec:Topology}
We now consider the case where players can alter their local topology according to the dynamics laid out in Section \ref{sec:TopEvRules}. We begin with a simple observation:
\begin{proposition} Let $G(t)$ be the time-varying graph structure. If $\mathbf{A}$ is a zero-sum game matrix, then $G(t)$ converges to a graph with no edges and it does so at time $t = \tau$.
\label{prop:zerosum}
\end{proposition}
\begin{proof} A zero-sum game necessarily implies that if $P_{ij}(\tau) > 0$, then $P_{ji}(\tau) < 0$. Thus every edge is removed at the first graph update step.
\end{proof}
Just as zero-sum games are inessential in cooperative game theory because there is no reason for players to form coalitions (\cite{Grif11}, Chapter 11), so too in our dynamic framework the network dissolves completely into isolated individuals.

By the same token, games in which $\mathbf{A} > \mathbf{0}$ also yield especially simple graph evolution dynamics.
\begin{proposition} Suppose $\mathbf{A} > \mathbf{0}$. At $t = \tau$ (the first network update epoch), if $e = \{i,j\} \not\in E(G)$, then $G + e \succ_i G$ and $G + e \succ_j G$; i.e., every player interaction, regardless of strategy improves total payoff. Consequently, the complete graph with $n$ vertices is pairwise stable.
\label{prop:Complete}
\end{proposition}
\begin{proof} This is a consequence of the fact that $\mathbf{A} > \mathbf{0}$ and thus every interaction increases the total payoff to any player irrespective of their strategies.
\end{proof}

\subsubsection{Classical Prisoner's Dilemma}
For the remainder of this section, we consider the special case where $\mathbf{A}$ is a prisoner's dilemma matrix with structure:
\begin{equation}
\mathbf{A} =
\begin{bmatrix}
R & S\\
T & P
\end{bmatrix}.
\label{eqn:PDMatrix}
\end{equation}
Here $T > R > P > S$. As noted in \cite{Wei95}, the same dominance structure can be preserved by subtracting $T$ from the first column and $S$ from the second column to obtain a diagonal matrix with diagonal elements $R - T < 0$ and $P - S > 0$, which simplifies calculations.
The following lemma is clear:
\begin{lemma} Assume $\mathbf{A} \in \mathbb{R}^{2\times 2}$ is a prisoner's dilemma matrix. If $\mathbf{x}^1,\mathbf{x}^2 \in \Delta_2$ and $\mathbf{x}_2^1 > \mathbf{x}_2^2$, then $\langle{\mathbf{x}^1,\mathbf{A}\mathbf{x}^2}\rangle > \langle{\mathbf{x}^2,\mathbf{A}\mathbf{x}^1}\rangle$. \hfill\qed
\label{lem:PD}
\end{lemma}
The next assertion follows immediately from Corollary \ref{cor:ConstantSingleton}.
\begin{corollary} Let $\mathbf{A} \in \mathbb{R}^{2\times 2}$ be a prisoner's dilemma matrix and suppose $G$ is the complete graph. Assume an initial condition $\mathbf{x}(0) = \mathbf{x}_0$ for Equation \ref{eqn:master}. Without loss of generality, assume $P_n(\mathbf{x}_0) \geq P_i(\mathbf{x}_0)$, then $\lim_{t\rightarrow\infty} \mathbf{x}^i = \mathbf{x}^n_0$ for all $i$. 
\label{cor:PD}
\end{corollary}
\begin{proof} Prisoner's dilemma has a strictly dominant strategy and symmetry in the graph shows that Player $n$ must play this dominant strategy with highest probability and therefore $P_n(t)$ is always larger than all other payoffs. Thus $\mathcal{X} = \mathbf{x}^n$ for all time. Convergence follows from Corollary \ref{cor:ConstantSingleton}.
\end{proof}

We now focus our attention on community evolution when players are engaged in the Prisoner's dilemma game with payoff matrix given by Equation \ref{eqn:PDMatrix}.

Since we are dealing with two-strategy Prisoner's dilemma, we may assume that $\mathbf{x}^{i} = (x_i,1-x_i) \in \Delta_2$. Here, $x_i \in [0,1]$ is a scalar with $x_i=1$ corresponding to the pure strategy $\mathbf{e}_1$ (cooperate) and $x_i = 0$ corresponding to the pure strategy $\mathbf{e}_2$ (defect). Thus, we may associate each vector strategy $\mathbf{x}^{i}$ to a scalar $x_i$ and work with this as needed. The following lemma follows from algebra:
\begin{lemma} Player $i$ will form a link with Player $j$ during graph reconfiguration if and only if:
\begin{equation}
x_j > \frac{(P-S)x_i-P}{(P+R-S-T)x_i + (T-P)}
\label{eqn:LinkPD}
\end{equation}
\hfill\qed
\end{lemma}
We extend Proposition \ref{prop:Complete} to the two-strategy Prisoner's dilemma game.
\begin{theorem} If $\mathbf{A} \in \mathbb{R}^{2\times 2}$ is a Prisoner's Dilemma matrix, and $G(t)$ converges to $G^*$, then $G^*$ is a graph composed of pairwise stable isolated complete subgraphs (cliques).
\end{theorem}
\begin{proof} If $\mathbf{A} > \mathbf{0}$, the result follows from Proposition \ref{prop:Complete}. Assume $\mathbf{A} \not> \mathbf{0}$. We have proved that there is a fixed point $\mathbf{x}^*$ to which the system converges and we may assume from the statement of the theorem that $G(t) \rightarrow G^*$. The graph $G^*$ must contain components $\mathcal{C} = \{C_1,\dots,C_M\}$. Without loss of generality, consider any component $C \in \mathcal{C}$. Let $\mathbf{x}^{C}$ be the fixed point strategies of the players in component $C$.

It suffices to show that $\mathbf{x}_1^{C} = \mathbf{x}_2^{C} = \cdots =  \mathbf{x}_{|C|}^{C}$, where $|C|$ is the number of vertices in $C$. Clearly if $|C| = 1$, the result is trivial. Suppose that $|C| > 1$. We now make use of the fact that we can write $\mathbf{x}^{i} = (x_i,1-x_i)$.

Assume the strategies in $C$ are organized so that $x^C_1 \leq x^C_2 \leq \cdots \leq x_{|C|}^C$. That is, $x^C_1$ is the lowest probability of cooperation, while $x_{|C|}^C$ is the highest probability of cooperation. There is at least one $j < |C|$ so that $x^C_j$ satisfies Inequality \ref{eqn:LinkPD} when $i = |C|$. Thus, in particular, $j = |C| - 1$ satisfies this inequality. Note that $N(j) \supseteq N(|C|)$; that is, every neighbor of $|C|$ is also a neighbor of $j$. By our assumption, $x_j^C \neq x_{|C|}^C$. Therefore, at equilibrium, either $P_{|C|} = P_j$ or $\kappa_{|C|j}(\mathbf{x}^C) = 0$. The fact that $N(j) \supseteq N(|C|)$ implies at once that $P_{|C|} < P_j$ as a result of the Prisoner's dilemma payoff function and therefore, by extension, $\kappa_{|C|j}(\mathbf{x}^C) \neq 0$ for any $j \in N(|C|)$. It follows immediately that $\mathbf{x}^*$ is not a fixed point unless $\mathbf{x}_1^{C} = \mathbf{x}_2^{C} = \cdots =  \mathbf{x}_{|C|}^{C}$ for every $C \in \mathcal{C}$.

The fact that every player in a component reaches consensus (within the component) implies that $G^*$ must be composed of cliques because each player can always improve her score by joining with every other player as a result of the structure of the player payoff function (Equation \ref{eqn:Pi}). Therefore, $G^*$ is composed of pairwise stable isolated cliques. This completes the proof.
\end{proof}
We remark that this is consistent with the underlying thesis of \cite{MSC01} that individuals with similar traits will form small networks with each other (homophilly). Moreover, consistent with other social theories, individuals who are alike enough will ultimately converge to a common behavior.

\subsection{Examples with Static Topology}\label{sec:Ex}
In this section and the next, we illustrate convergence and consensus using various classical two and three strategy games. For the sake of simplicity, we will fix a graph. We use the classic Karate Club graph from Zachary's original study \cite{Z77} (see Figure \ref{fig:KarateClub}).
\begin{figure}[htbp]
\centering
\includegraphics[scale=0.5]{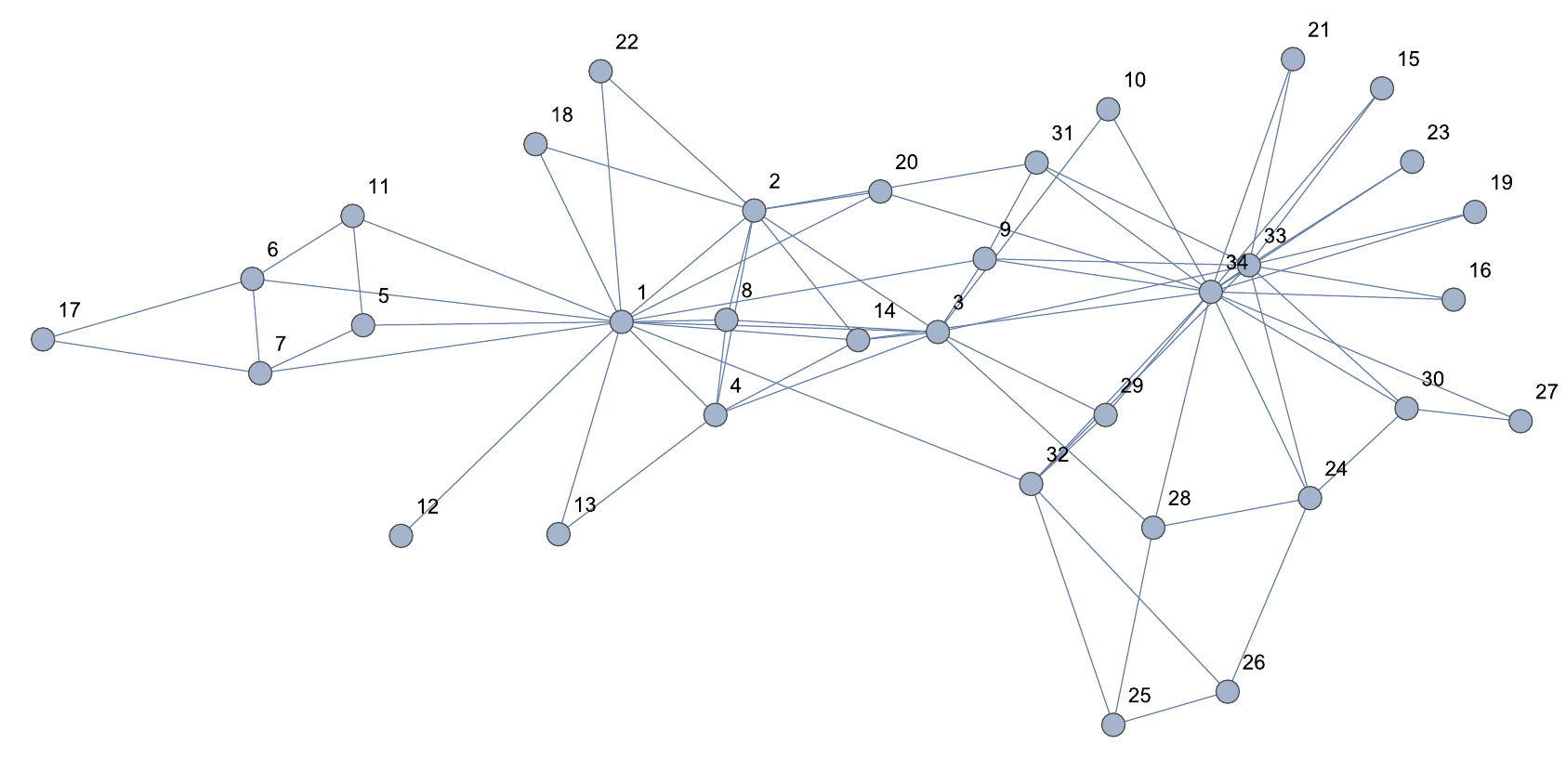}
\caption{Zachary's Karate Club graph with its two ``hub'' vertices, 1 and 34.}
\label{fig:KarateClub}
\end{figure}
This small graph has two ``hub'' vertices, who ultimately become leaders of their own individual groups \cite{Z77}. These two hub vertices are 1 and 34. In the sequel, we refer to these as the leader players. We will use this graph to numerically explore the basins of attraction of various fixed points.

\subsubsection{Bistable Games}\label{sec:Bistable}
Consider the classical bistable stag-hunt game with payoff matrix:
\begin{displaymath}
\mathbf{A} = \begin{bmatrix}2 & -1 \\ -1 & 2\end{bmatrix}
\end{displaymath}
Figure \ref{fig:StagHunt} shows three examples with varying convergence and strategy consensus modes.
\begin{figure}[htbp]
\centering
\subfigure[]{
\includegraphics[scale=0.5]{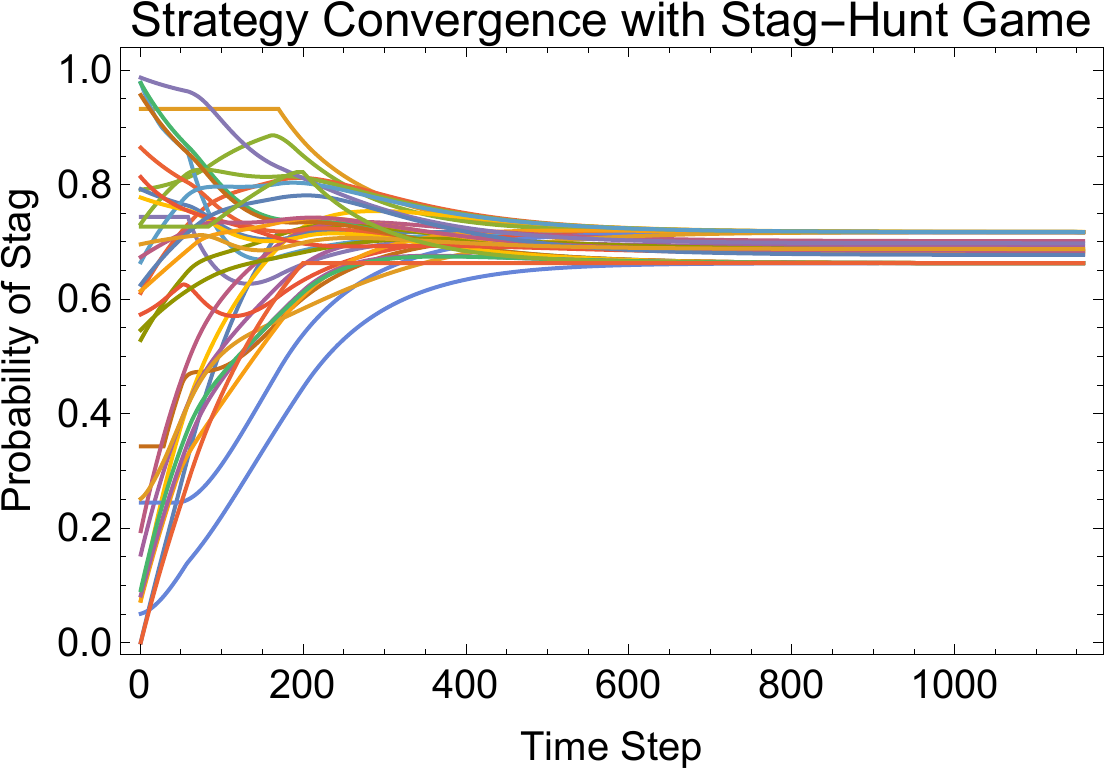}}
\subfigure[]{\includegraphics[scale=0.5]{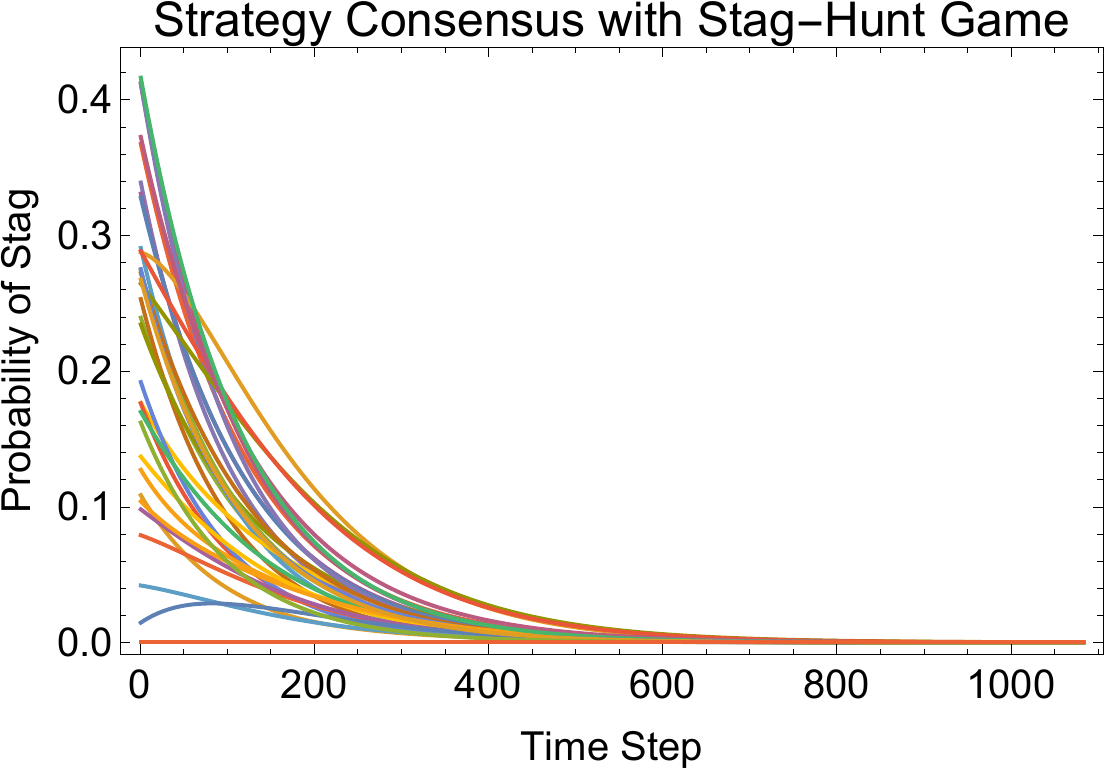}}
\subfigure[]{\includegraphics[scale=0.5]{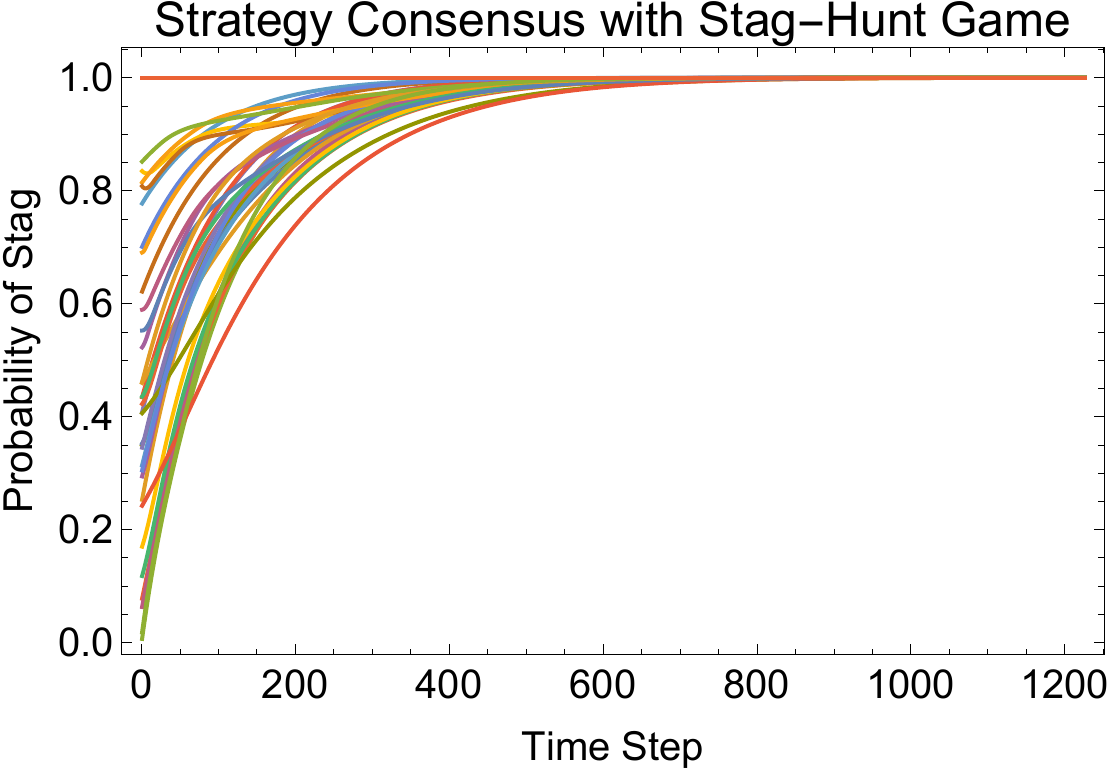}}
\caption{Consensus and convergence in Stag-Hunt is partially determined by the graph structure and partially determined by the starting conditions. (a) Strategy Convergence. (b) Strategy consensus pre-determined by starting condition. (c) Strategy consensus with random starting condition.}
\label{fig:StagHunt}
\end{figure}
The left and middle figures were generated by initializing Vertices 1 and 34 with strategy $\mathbf{e}_2 = \langle{0,1}\rangle$ and all other vertices were initialized with a random strategy. The right figure was generated by initializing Players 1 and 34 with strategy $\mathbf{e}_1 = \langle{1,0}\rangle$ while all other players were initialized with a random strategy. Notice in two of the three cases, the random strategy is in the basin of attraction of the consensus strategy $\mathbf{e}_2$ (center) or $\mathbf{e}_1$ (right).

The relative simplicity of this game allows us to use Proposition \ref{prop:MaxMin} to determine sufficient criteria on the starting condition to ensure that it will converge to one of the pure strategy fixed points. We start with the case when $\mathbf{x}^1 = \mathbf{x}^{34} = \mathbf{e}_1$. From the graph, we can determine that $d_1 = 16$ and $d_{34} = 17$. If an arbitrary strategy $\mathbf{y} = \langle{y,1-y}\rangle$, then $\mathbf{y} \in \mathcal{H}(\mathbf{x})$ can be written simply as $a \leq y \leq b$ for appropriately chosen $a, b \in [0,1]$. Then the left-hand-side of Expression \ref{eqn:SufficientCriteria} can be written:
\begin{displaymath}
\min_{\mathbf{y}\in\mathcal{H}(\mathbf{x}(t))} 16\cdot\left\langle{\mathbf{e}_1,\mathbf{A}\mathbf{y}}\right\rangle = 16 \min_{y\in[a,b]} 16(3y - 1) = 16(3a - 1)
\end{displaymath}
The degree of the next largest non-maximal vertex is 12. Therefore, the right-hand-side of Expression \ref{eqn:SufficientCriteria} can be written as:
\begin{displaymath}
\max_{\substack{j \in V\setminus V^*\\\mathbf{y},\mathbf{z}\in\mathcal{H}(\mathbf{x}(t))}} d_j\cdot\left\langle{\mathbf{y},\mathbf{A}\mathbf{z}}\right\rangle = \max_{x,y \in [a,b]} 12\left(1 - 2 x - 2 y + 5 x y\right) = 12\left(1-4b+5b^2\right)
\end{displaymath}
The fact that $\mathbf{x}^1 = \mathbf{x}^{34} = \mathbf{e}_1$ implies that $b = 1$ and the right-hand-side is maximized when $x = y = 1$. This is the strategy currently used by Vertices 1 and 34, which we discuss momentarily. To ensure that the maximal set is fixed, we require:
\begin{displaymath}
16(3a - 1) > 24 \implies a > \frac{5}{6}
\end{displaymath}
This result is sufficient to ensure that $\mathbf{x}^i \rightarrow \mathbf{e}_1$ for all $i \in \{1,\dots,34\}$. However, in constructing this condition, we used the observation that another vertex is using strategy $\mathbf{e}_1$. If we assume that the player to whom we compare the worst possible payoff to a maximal vertex cannot play the same strategy as the maximal vertex, then we set $b = 1-\epsilon$ for some small $\epsilon > 0$ (distinct from the imitation rate) and conclude:
\begin{displaymath}
a>\frac{1}{12} \left(15 \epsilon ^2-18 \epsilon +10\right)
\end{displaymath}
It is clear from empirical analysis that the basin of attraction of the fixed point where all players play $\mathbf{e}_1$ is larger than this strict bound suggests.

By way of comparison, suppose $\mathbf{x}^1 = \mathbf{x}^{34} = \mathbf{e}_2$. The left-hand-side of Expression \ref{eqn:SufficientCriteria} is:
\begin{displaymath}
\min_{\mathbf{y}\in\mathcal{H}(\mathbf{x}(t))} 16\cdot\left\langle{\mathbf{e}_2,\mathbf{A}\mathbf{y}}\right\rangle =\min_{y\in[a,b]} 16(1-2x) = 16(1-2b)
\end{displaymath}
The right-hand-side of Expression \ref{eqn:SufficientCriteria} is still $12\left(1-4b+5b^2\right)$. Thus, a sufficient criterion for consensus to $\mathbf{e}_2$ is:
\begin{displaymath}
16(1-2b) > 12\left(1-4b+5b^2\right) \implies b < \frac{1}{15} \left(2+\sqrt{19}\right)
\end{displaymath}
That is, the set of all strategies of the form $\mathbf{x}^{i} = \langle{x,1-x}\rangle$ with $x < \frac{1}{15} \left(2+\sqrt{19}\right)
$ and $\mathbf{x}^1 = \mathbf{x}^{34} = \mathbf{e}_2$ is in the basin of attraction of the consensus point $\mathbf{x}^i = \mathbf{e}_2$ for $i = 1,\dots,34$. This is illustrated in Figure \ref{fig:StagHunt}(b), where a starting strategy was specifically picked using this bound.

\paragraph{The Chicken Game} We contrast these results with the Chicken game using payoff matrix:
\begin{displaymath}
\mathbf{A} = \begin{bmatrix}0 & -1\\1 & -10\end{bmatrix}
\end{displaymath}
In the figures, we describe the pure strategy $\mathbf{e}_1$ as \textit{swerve} and pure strategy $\mathbf{e}_2$ as \textit{don't swerve} in keeping with the classical story behind the game. Two examples are shown in Figure \ref{fig:Chicken}.
\begin{figure}[htbp]
\centering
\subfigure[]{\includegraphics[scale=0.5]{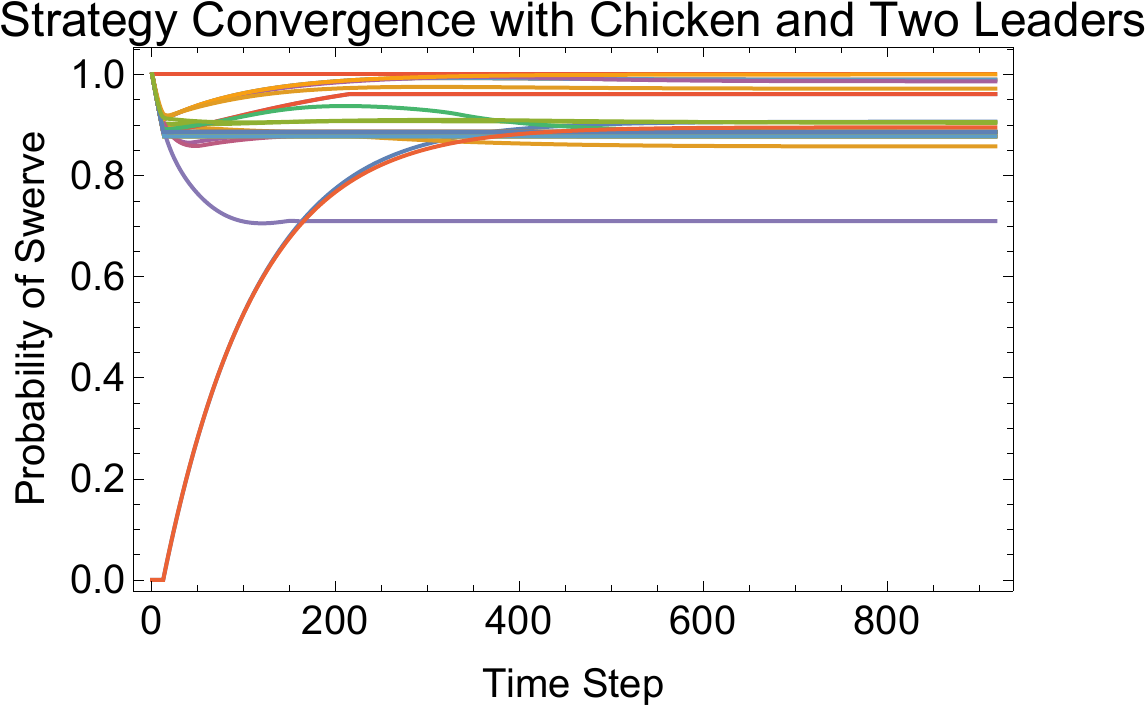}}
\subfigure[]{\includegraphics[scale=0.5]{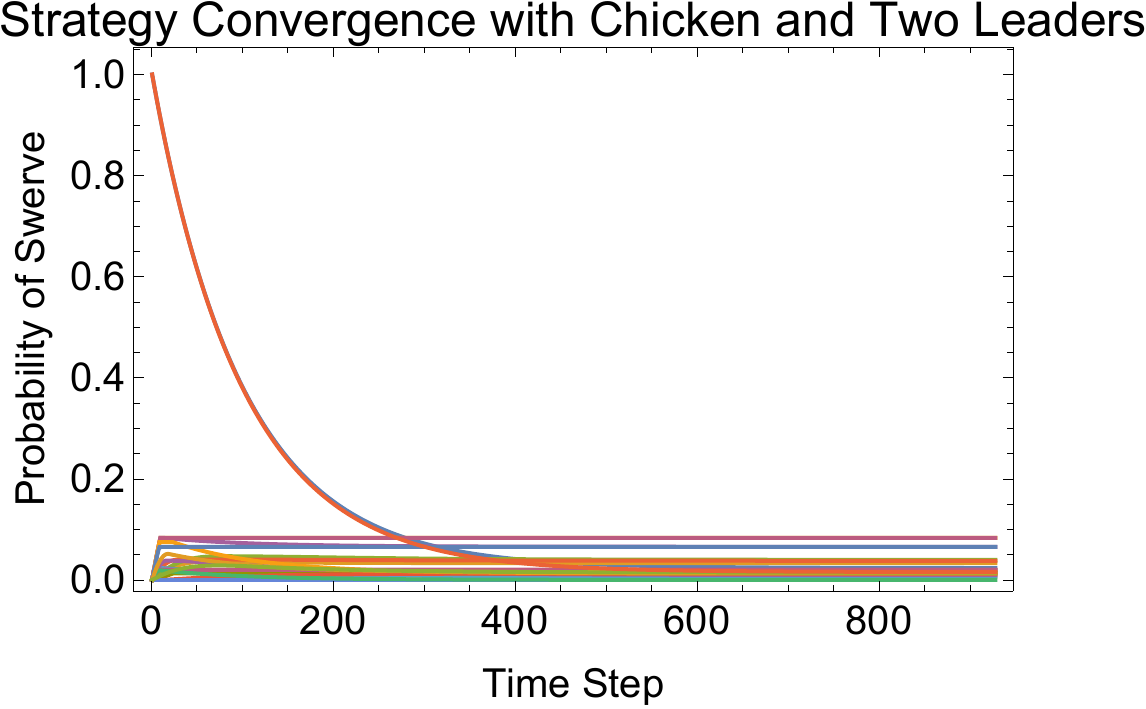}}
\subfigure[]{\includegraphics[scale=0.5]{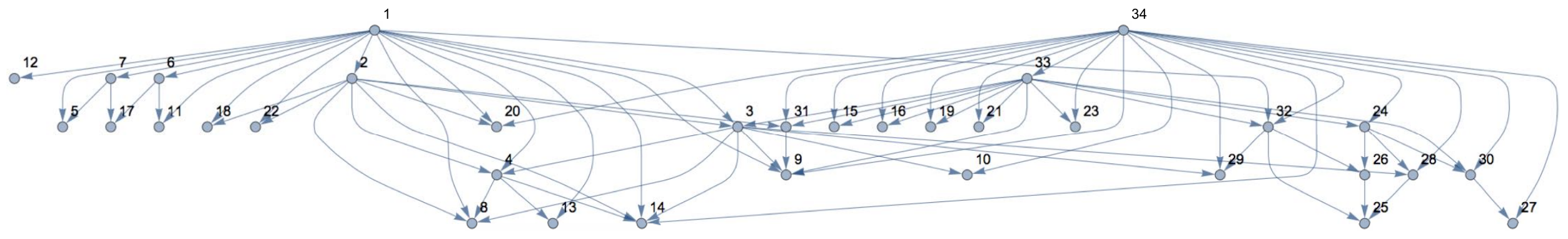}}
\caption{Consensus in chicken seems difficult to obtain empirically. (a) The chicken game with the leaders initialized with $\mathbf{e}_2$ and all other vertices initialized with $\mathbf{e}_1$. (b) The chicken game with the leaders initialized with $\mathbf{e}_1$ and all other players initialized with $\mathbf{e}_2$. (c) The imitation graph when the leaders initialized with $\mathbf{e}_2$ and all other players initialized with $\mathbf{e}_1$.}
\label{fig:Chicken}
\end{figure}
In Figure \ref{fig:Chicken}(a), we initialized the two leader players with $\mathbf{e}_2$ and observe as the remainder of the group drives them toward $\mathbf{e}_1$. The final imitation graph is shown in Figure \ref{fig:Chicken}(c), where the leader vertices ($v_1$ and $v_{34}$) are minimal in the induced vertex ordering. This was also the case in the second example in which the leader vertices were initialized with pure strategy $\mathbf{e}_1$ and were pulled toward the $\mathbf{e}_2$. To see that the pure strategies are unstable under these dynamics, consider (e.g.) the fixed point $\mathbf{x}^* = (\mathbf{e}_2,\mathbf{e}_2,\dots,\mathbf{e}_2)$. That is, where all players play $\mathbf{e}_2$ (\textit{don't swerve}). If any player changes strategy, his payoff immediately increases and all other players will begin to imitate that player. A similar argument can be made when all players begin at $\mathbf{e}_1$. Interestingly for the Chicken game, the imitation graph suggests that those vertices with lower degree have higher order in the induced vertex ordering, probably because their reduced interactions lead to less loss overall.


\subsubsection{Cyclic Games}
Consider the three strategy generalization of rock-paper-scissors (RPS) with game matrix:
\begin{displaymath}
\mathbf{A} = \begin{bmatrix} 0 & -1 & 1+a\\
1+a & 0 & -1\\
-1 & 1+a & 0
\end{bmatrix}
\end{displaymath}
When $a = 0$, this game has the property that for any $\mathbf{x} \in \Delta_3$, $\langle{\mathbf{x},\mathbf{A}\mathbf{x}}\rangle = 0$. Thus as player imitate each other, their payoffs from interaction will approach zero. Furthermore, when $a = 0$, it is clear there are graphs and initial conditions for which consensus cannot occur. We illustrate such a graph in Figure \ref{fig:RPSNoConsensus}.
\begin{figure}[htbp]
\centering
\includegraphics[scale=0.75]{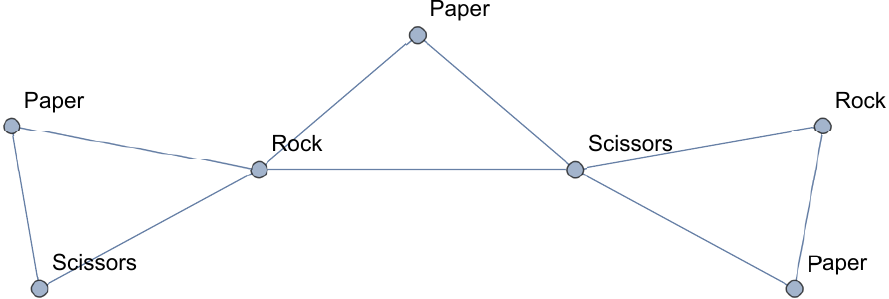}
\caption{A graph and initial condition that is already a non-consensus equilibrium point in Rock-Paper-Scissors.}
\label{fig:RPSNoConsensus}
\end{figure}
In fact any 3-colorable graph that admits a coloring in which each vertex is adjacent to equal numbers of vertices colored the opposite two colors has an initial condition from which consensus cannot be achieved. Empirical evidence, however, suggests that for a random starting configuration on an arbitrary graph, RPS with $a = 0$ comes to consensus. Changing $a$ to be positive or negative destroys this consensus. This is illustrated in Figure \ref{fig:RPS}.
\begin{figure}[htbp]
\centering
\subfigure[$a=0$]{\includegraphics[scale=0.4]{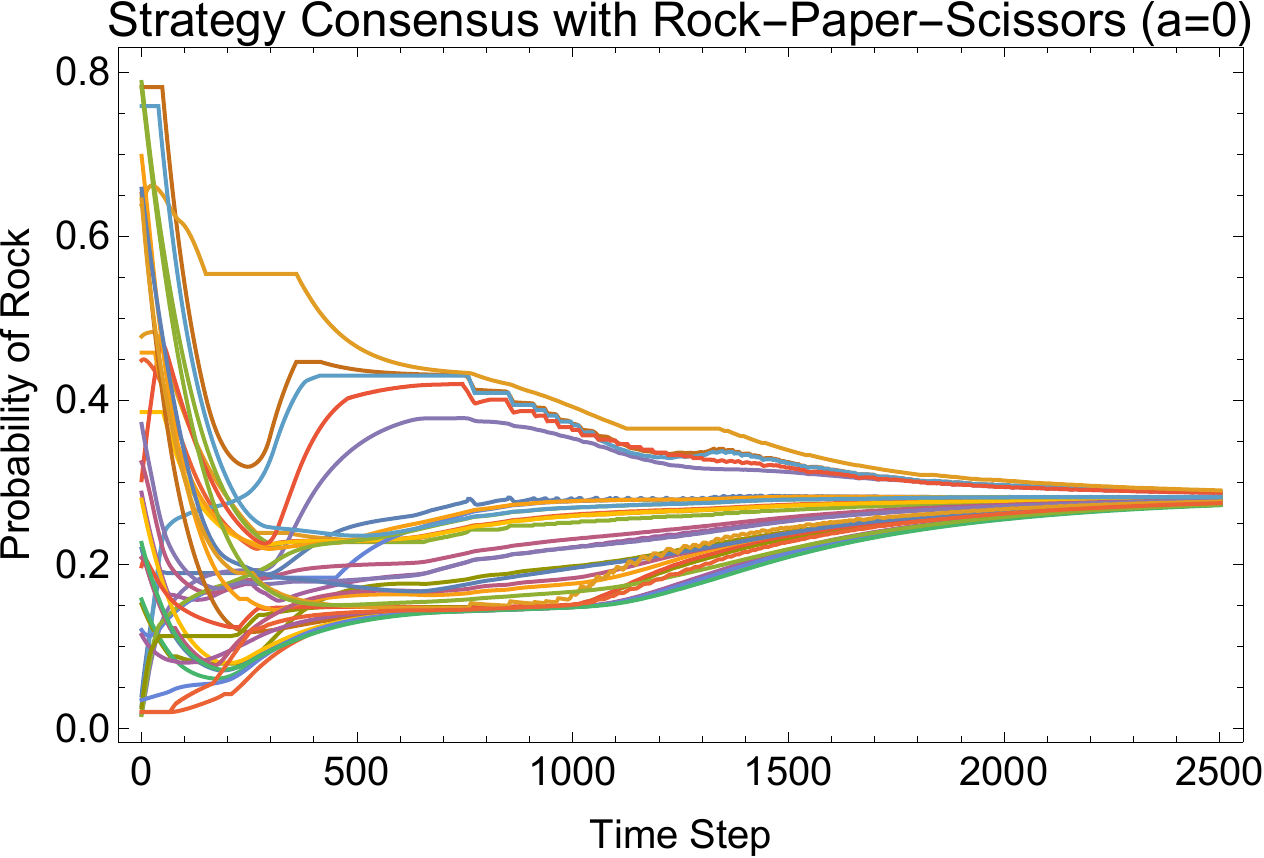}}
\subfigure[$a=0.05$]{\includegraphics[scale=0.4]{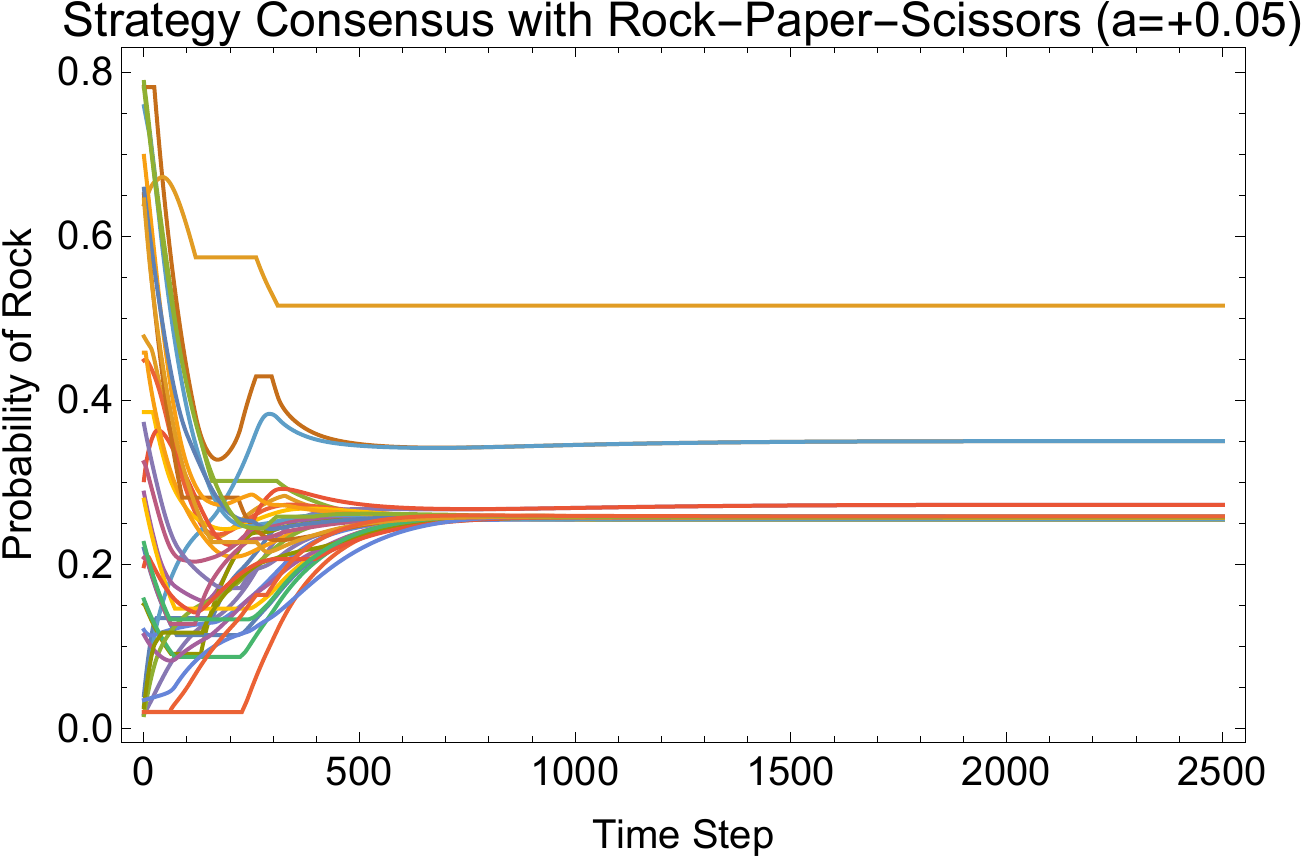}}
\subfigure[$a=-0.05$]{
\includegraphics[scale=0.4]{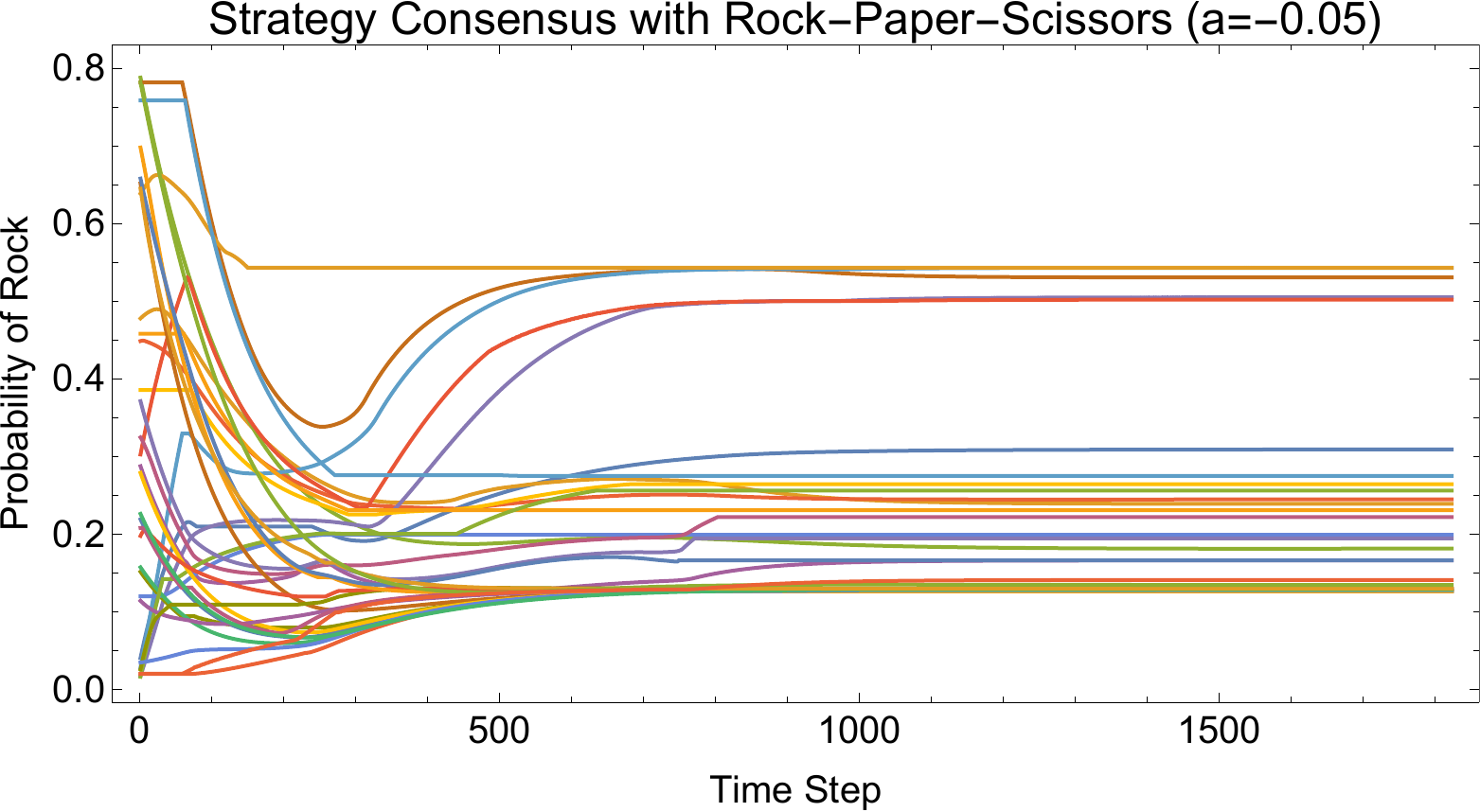}
}
\caption{RPS tends to come to consensus from random starting points when $a = 0$, but consensus is broken when $a \neq 0$.
}
\label{fig:RPS}
\end{figure}
In every experiment run with $a = 0$, RPS converged to consensus assuming random starting strategies were used; i.e., the dynamics shown in Figure \ref{fig:RPS}(a) are representative. Thus, an open question to be explored in future work is to prove or disprove the conjecture that for arbitrary unbiased cyclic games ($a=0$) with an odd number of strategies, almost every initial strategy converges to consensus on an arbitrary connected graph.

\subsection{Example with Topological Evolution}
We illustrate the results presented on dynamic topologies, especially with Prisoner's dilemma, using a random initial strategy and playing Prisoner's dilemma on the Karate Club graph with (e.g.) $R=3$, $S=-1$, $T=5$ and $P=2$. Figure \ref{fig:Strategies} shows the resulting probability of cooperation in an example simulation. Notice the system converges, but not to consensus even though prisoner's dilemma has a strictly dominating strategy.
\begin{figure}[htbp]
\centering
\includegraphics[scale=0.6]{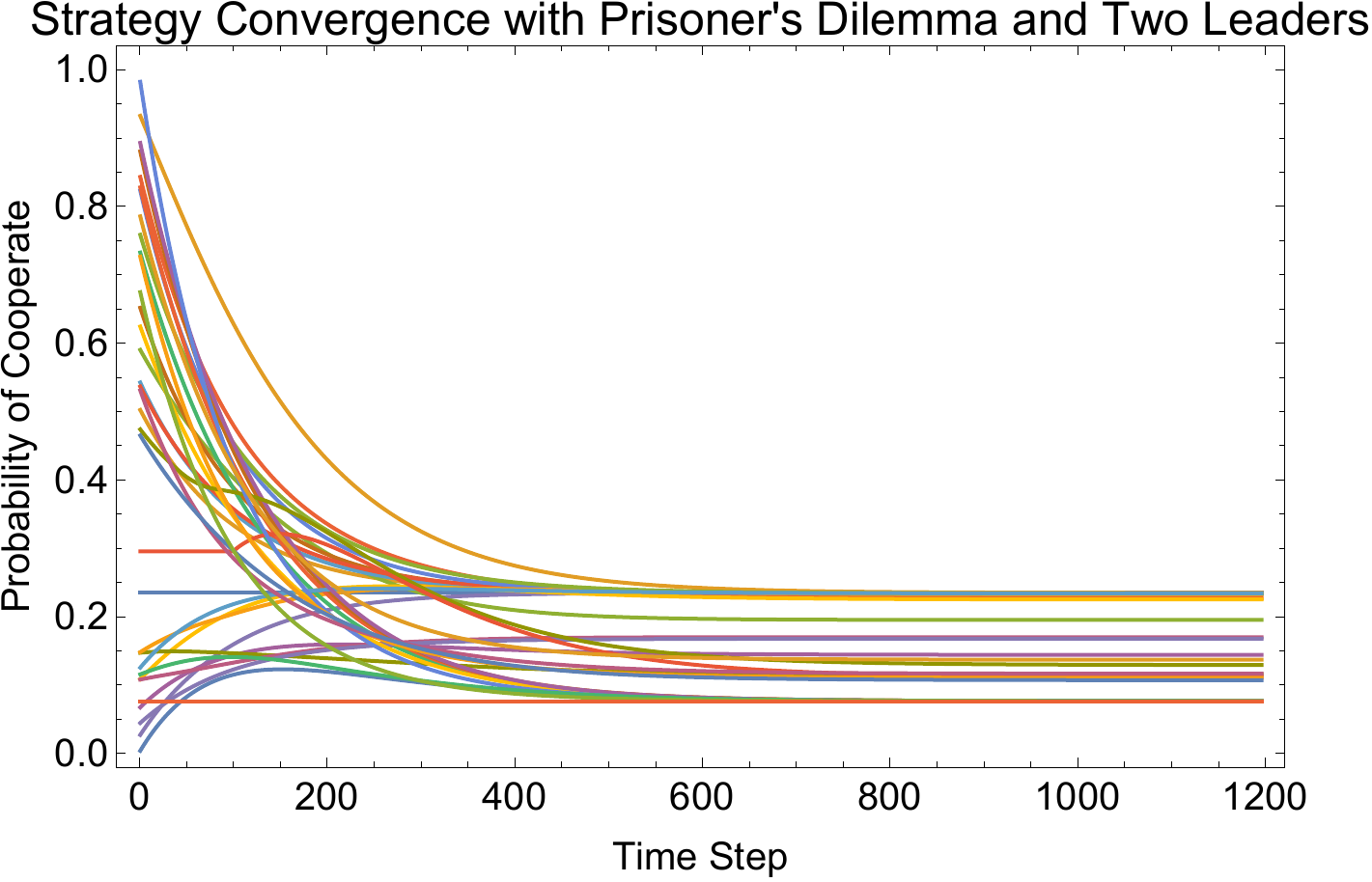}
\caption{Players' strategies converge but strategy consensus is not achieved because there are two clear leaders.}
\label{fig:Strategies}
\end{figure}
The fact that strategy consensus is not achieved is explained by the imitation graph, shown in Figure \ref{fig:IGraph}(a) as well as Theorem \ref{thm:consensus}. We can see that Vertex 1 and Vertex 34 both have zero out-degree, but different strategies. 
\begin{figure*}[htbp]
\centering
\subfigure[Imitation Graph]{\includegraphics[scale=0.45]{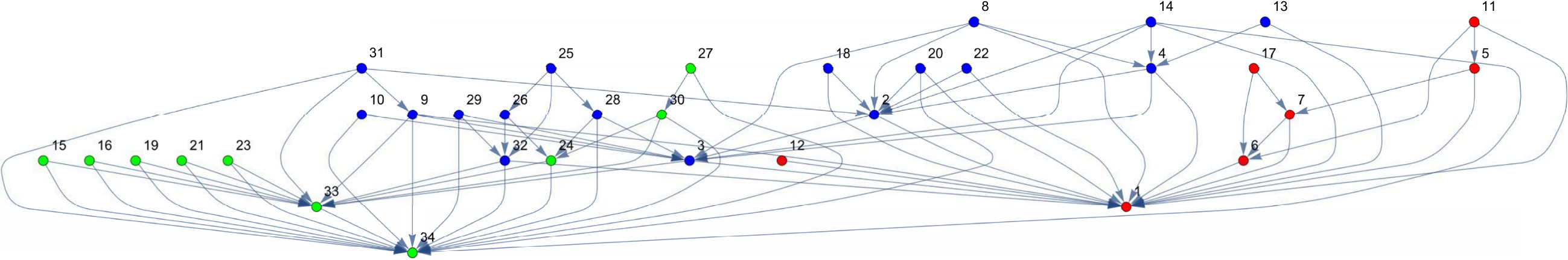}}
\subfigure[Highlighted Karate Club]{\includegraphics[scale=0.4]{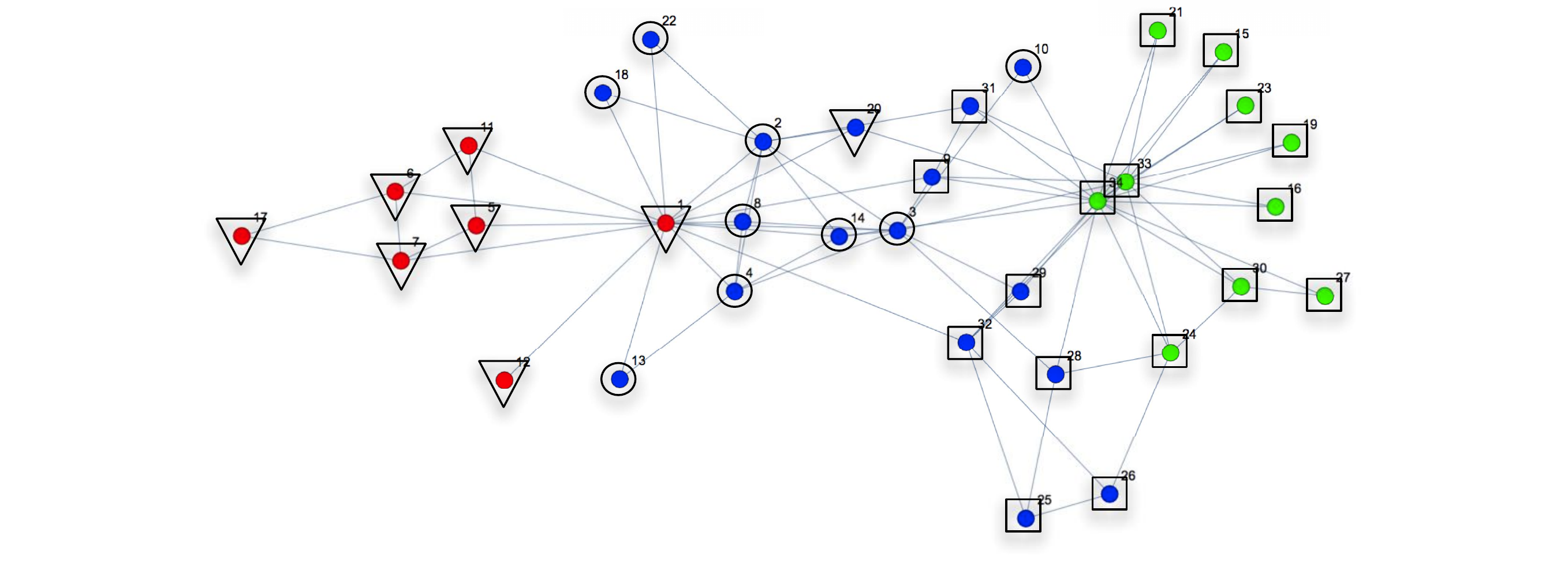}}
\caption{The imitation graph at convergence is shown with two clear leaders. Vertices are colored based on which leader vertex or vertices they imitate. The result is qualitatively similar to maximum modularity clustering. Maximum modularity clusters are shown by shape around the vertex.}
\label{fig:IGraph}
\end{figure*}
In Figure \ref{fig:IGraph}, we have highlighted vertices that uniquely imitate Vertex 1, uniquely imitate Vertex 34 and imitate both Vertex 1 and 34. We can see this tripartitions the graph. Interestingly, this is qualitatively similar to the community structures identified by maximum modularity clustering \cite{BDGG08}. Maximum modularity clusters are shown in Figure \ref{fig:IGraph}(b) by shape for comparison. This suggests the potential for a novel community detection algorithm, which we discuss in future work.

We also illustrate an example of topological change with a slight variation in the payoff matrix. We use $R = 8$, $S = -4$, $T = 10$ and $P = 2$. We set $\tau = 25$ so that every twenty-fifth time step, the topology of the graph was updated. The graph evolution is shown in Figure \ref{fig:Topology}.
\begin{figure}[htbp]
\centering
\includegraphics[scale=0.6]{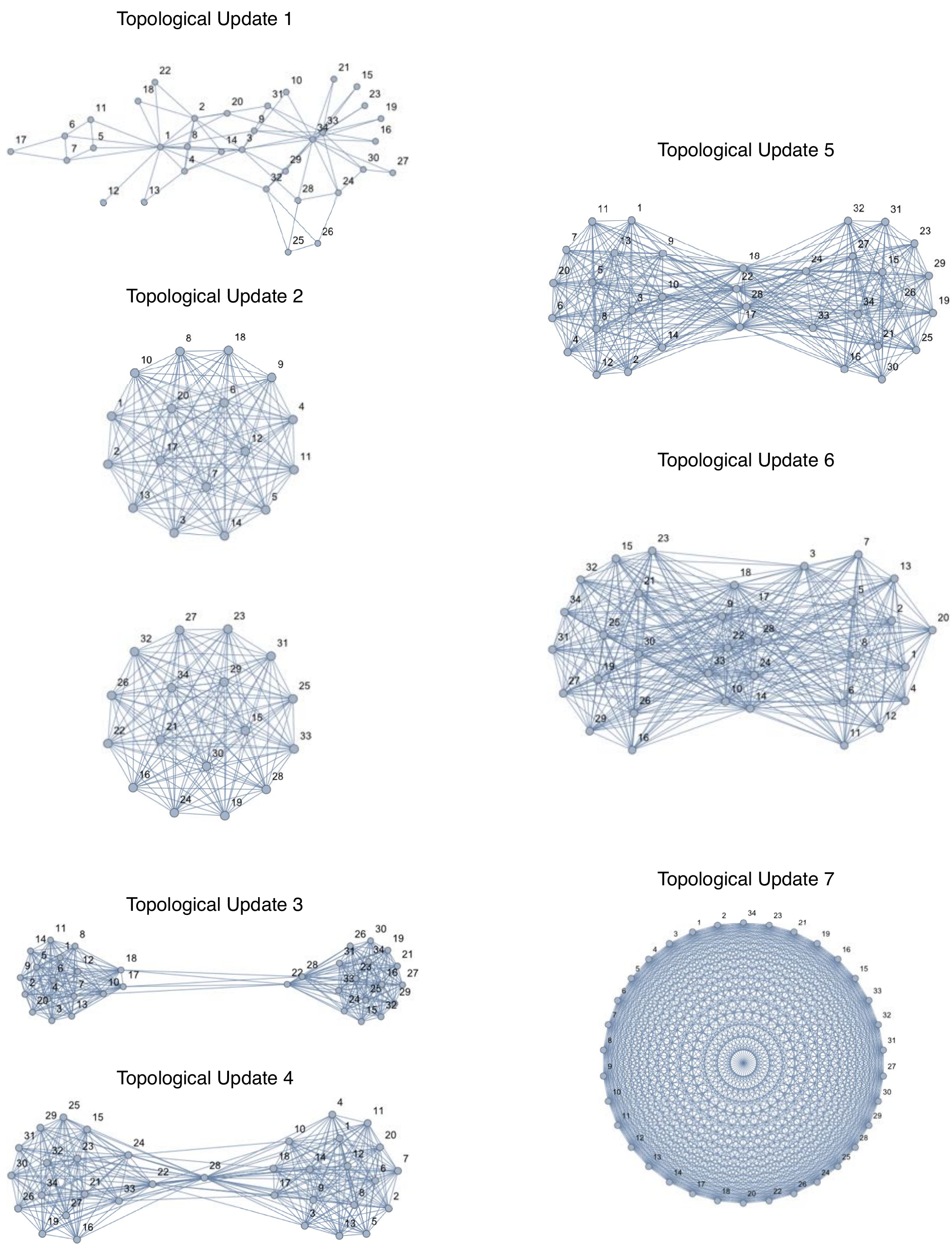}
\caption{Topological evolution of the network playing Prisoner's Dilemma shows the network first breaks into two components and then reforms into a single complete graph.}
\label{fig:Topology}
\end{figure}
In this example, the network shows the network first breaks into two components and then reforms into a single complete graph. We also show the path of the trajectory of strategic consensus,  which is ensured by Corollary \ref{cor:PD} in Figure \ref{fig:Consensus}. Notice that points of non-differentiability in the trajectories can correspond to change points in both the topology as well as change points in the imitation graph.
\begin{figure}[htbp]
\centering
\includegraphics[scale=0.5]{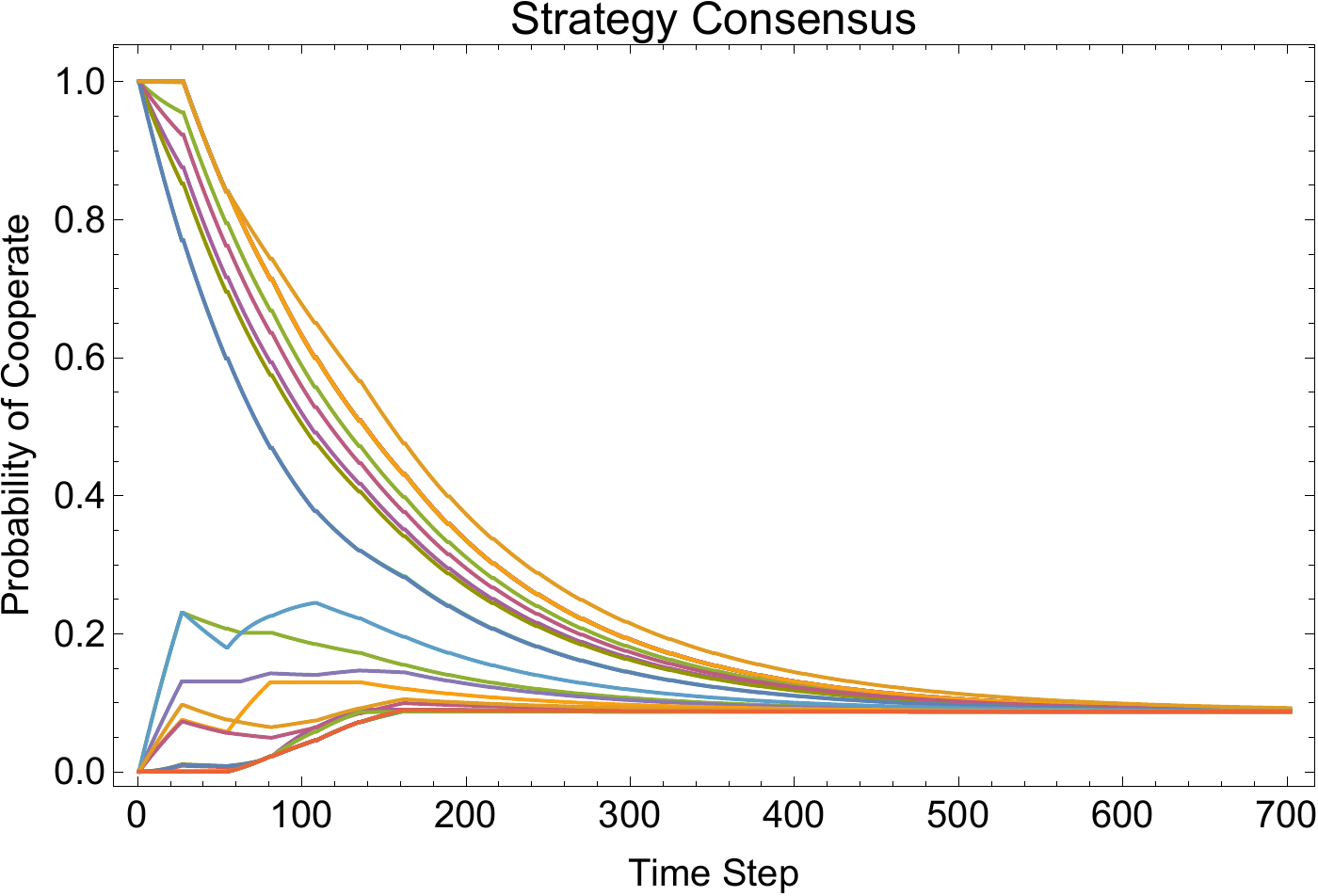}
\caption{In a complete graph, the strategies must converge in Prisoner's dilemma, see Corollary \ref{cor:PD}.
}
\label{fig:Consensus}
\end{figure}

\section{Trend Models}\label{sec:Trend}
In this section, we consider a variation on the model in which a new strategy emerges naturally (the trend) and users are influenced by an initial set of trend players. We assume that prior to trend emergence, the system is at (or near) a consensus point. For simplicity, we combine all non-trend strategies into a single non-trend strategy and analyze as we did in Section \ref{sec:Topology}. Notation will be presented for the general case, which is substantially harder to analyze. Most result are empirical because of the difficulties in obtaining closed form solutions in this case.

In our simplified model, the initial payoff matrix consists of a single element $\mathbf{A} = [R]$, while the post-trend payoff has structure:
\begin{equation}
\mathbf{A}^i_t = \begin{bmatrix}
R & S_0-(S_0-S)\beta^{-(t-\tau_i)}\\
T_0-(T_0-T)\beta^{-(t-\tau_i)} & P_0-(P_0-P)\beta^{-(t-\tau_i)}
\end{bmatrix}
\label{eqn:TrendMatrix}
\end{equation}
where $T > R$ and $P > S$ and $\tau_i$ is the time Player $i$ first has a non-zero probability of using the trend strategy. That is:
\begin{equation}
\arg\min_{t \geq 0} \langle{\mathbf{x}^i,\mathbf{e}_2}\rangle > 0
\end{equation}
The second column strategy corresponds to being ``in-trend,'' thus $\mathbf{e}_2 = \langle{0,1}\rangle$. Unlike prisoner's dilemma, we may assume that $P > R$, to model the social benefit from being ``in-trend.'' The value $\beta \in (0,1]$ is a decay parameter modelling the staying power of the trend once it's been played by a player; i.e., some trends are beneficial in the long-term (the use of e-mail over postal mail for correspondence, e.g.), other trends are short-lived (e.g., the ice-bucket challenge \cite{W16}). The fact that we have a two-strategy game, allows us to use the same notational convenience as in Section \ref{sec:Topology}; i.e., that $\mathbf{x}^{i} = (x_i,1-x_i)$. Prior to the introduction of the trend, we assume the system is at an equilibrium position $\mathbf{x}^*$. In the simplified case, all players simply play the non-trend strategy; i.e., $x_1^* = x_2^* = \cdots = x_n^* = 1$.

Let $\mathcal{S} \subseteq \{1,\dots,n\}$ be a set of \textit{seed players} so that $x_i(0) = 0$ for all $i \in \mathcal{S}$ and $x_i(0) = 1$ for all $i \not\in \mathcal{S}$. That is, the seed players initialize the trend.

Before proceeding, it is worth noting that all convergence results hold in the case of a time varying payoff matrix $\mathbf{A}_t$. Thus, we do not need to update any results assuming a trend of this type. However, the imitation model cannot accurately describe disassociation from the trend. That is, assuming all players eventually converge to the consensus $\mathbf{x} = \mathbf{0}$, where $\mathbf{x} = (x_1,\dots,x_n)$, it is impossible for the system to ever deviate from the trend strategy. To compensate for this, we modify Equation \ref{eqn:S} as:
\begin{equation}
S_i(\mathbf{x}) = w_i\lfloor{P_i({\mathbf{x}^i}^*) - P_i(\mathbf{x})}\rfloor_0 + \sum_{k \in N(i)}w_{ik}\left\lfloor{P_k(\mathbf{x}) - P_i(\mathbf{x})}\right\rfloor_0
\label{eqn:Sitrend}
\end{equation}
Here, $P_i({\mathbf{x}^i}^*)$ is the payoff Player $i$ receives if he suddenly stopped using his current strategy and reverted to his pre-trend strategy. Then we have:
\begin{equation}
\kappa_{0}(\mathbf{x}) =
\begin{cases}
\frac{w_{i}\left\lfloor{P_i({\mathbf{x}^i}^*) - P_i(\mathbf{x})}\right\rfloor_0}{S_i(\mathbf{x})} & \text{if $S_i(\mathbf{x}) > 0$} \\
0 & \text{otherwise}
\end{cases}
\label{eqn:kappatrend}
\end{equation}
and consequently:
\begin{equation}
f_i(\mathbf{x}) = \kappa_0(\mathbf{x})\left({\mathbf{x}^i}^* - \mathbf{x}^i\right) + \sum_{j \in N(i)} \kappa_{ij}(\mathbf{x})(\mathbf{x}^j - \mathbf{x}^i)
\label{eqn:fitrend}
\end{equation}
In the simplified model ${\mathbf{x}^i}^* = (1,0)$ for all $i$. We assume  seed players use this strategy before becoming early adopters of the trend. Equation \ref{eqn:fitrend} is then used in Equation the dynamics to model the trend.

\begin{definition}[Spread] A trend \textit{spreads} if there is some $j \not\in \mathcal{S}$ and some time $t \geq 0$ so that $x_j(t) < 1$. That is, a trend spreads only if some non-seed player has a non-zero probability of using the trend strategy. Equivalently $\tau_j < \infty$.
\end{definition}

\begin{definition}[Saturation] The trend \textit{saturates} the network if for each $i \in \{1,\dots,n\}$, we have $\tau_i < \infty.$ Alternatively, the saturation proportion at time $t$, $\pi_\mathcal{S}(t)$, is simply the proportion of players with non-zero probability of playing the trend strategy.
\end{definition}
When referring to $\max_{t\geq 0}\pi_\mathcal{S}(t)$, we will simply say the \textit{saturation proportion} without reference to time.

For simplicity, let $N^\mathcal{S}(i) = N(i) \cap \mathcal{S}$ and let $\bar{N}^\mathcal{S}(i)$ be its complement. Let $\bar{\mathcal{S}}$ be the set theoretic complement of $\mathcal{S}$ in $\{1,\dots,n\}$. The following proposition is obvious from the construction:
\begin{proposition} Using the dynamics defined by Equation \ref{eqn:fitrend} with payoff matrix defined in Equation \ref{eqn:TrendMatrix}, a trend spreads if and only if there is some $i \in \bar{\mathcal{S}}$ and some $j \in N^\mathcal{S}(i)$ so that:
\begin{equation}
|N^\mathcal{S}(i)|S + |\bar{N}^\mathcal{S}(i)|R < |N^\mathcal{S}(j)|P + |\bar{N}^\mathcal{S}(j)|T
\end{equation}
\end{proposition}
\begin{corollary} If $G$ is a complete graph and $|\mathcal{S}| = k$, then using the dynamics defined by Equation \ref{eqn:fitrend} with payoff matrix defined in Equation \ref{eqn:TrendMatrix}, a trend spreads and saturates if and only if:
\begin{equation}
(n-2)R + S < (n-k)P + (k-1)T
\end{equation}
\end{corollary}
In the complete graph, the duration of time before the trend collapses is governed by $\epsilon$, which essentially defines the uptake rate of the trend, as well as $\beta$, which defines the rate of collapse of the trend once encountered. By uptake rate, we mean the rate at which individuals move from having zero probability of playing the trend strategy to non-zero probability of playing the trend strategy. By trend collapse, we mean a point at which the probability any user is playing the trend is less than some $\epsilon > 0$ with $\epsilon \ll 1$.

Trend behavior on a complete graph is illustrated in Figure \ref{fig:CompleteGraphTrend}. Here $P_0 = S_0 = T_0 = 0$.
\begin{figure}[htbp]
\centering
\subfigure[$\beta=0.95$, $\epsilon=0.1$]{\includegraphics[scale=0.6]{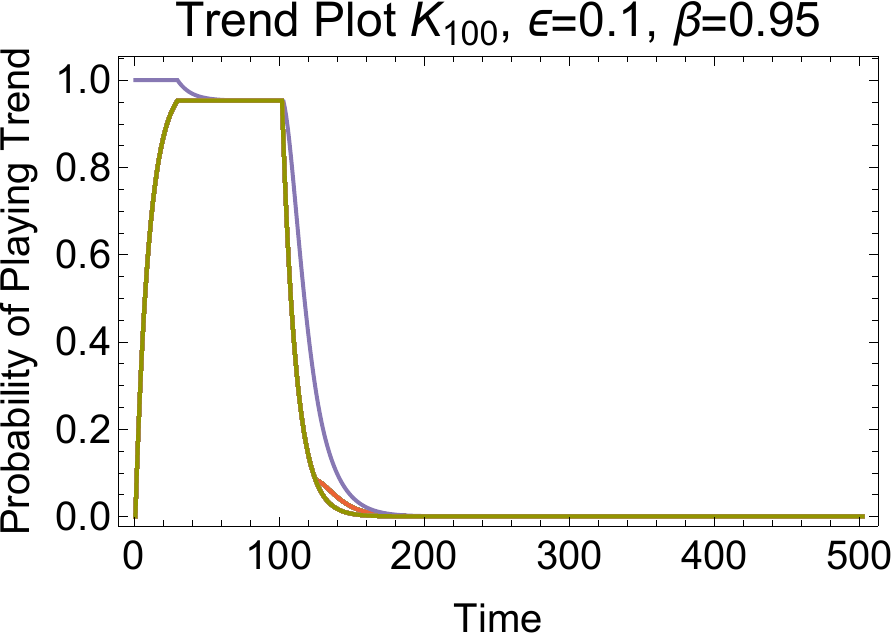}}
\hspace{2mm}
\subfigure[$\beta=0.95$, $\epsilon=0.25$]{\includegraphics[scale=0.6]{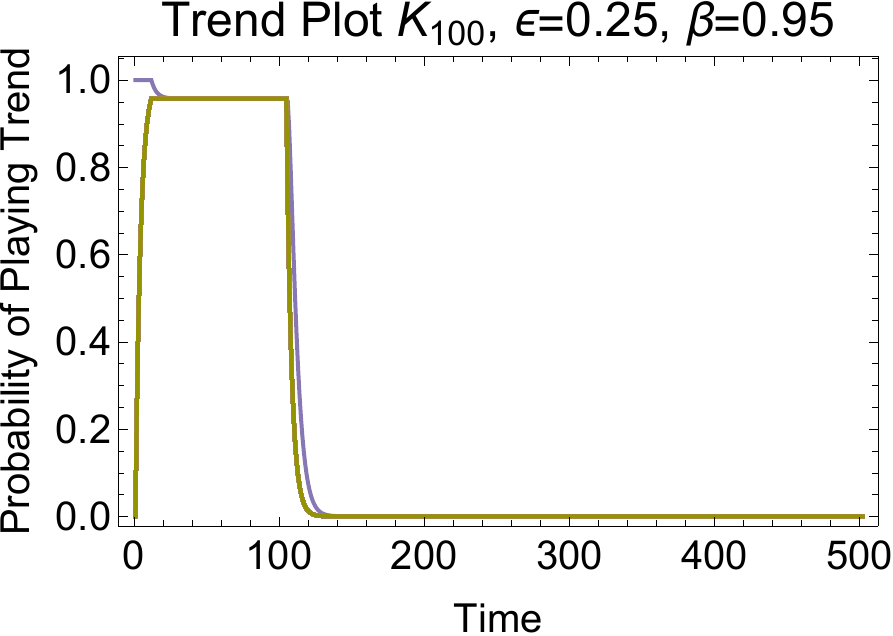}}
\subfigure[$\beta=0.85$, $\epsilon=0.1$]{\includegraphics[scale=0.6]{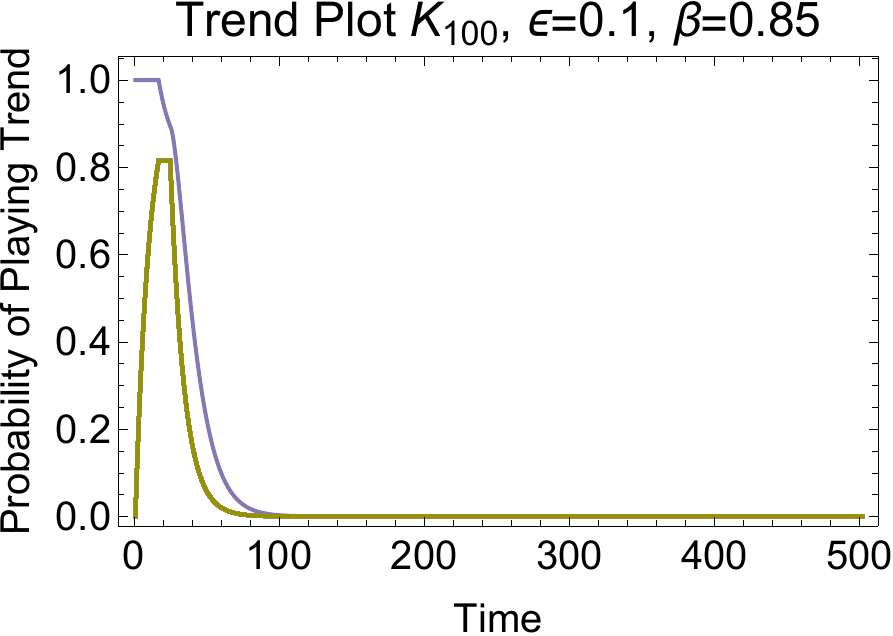}}
\hspace{2mm}
\subfigure[$\beta=0.85$, $\epsilon=0.25$]{\includegraphics[scale=0.6]{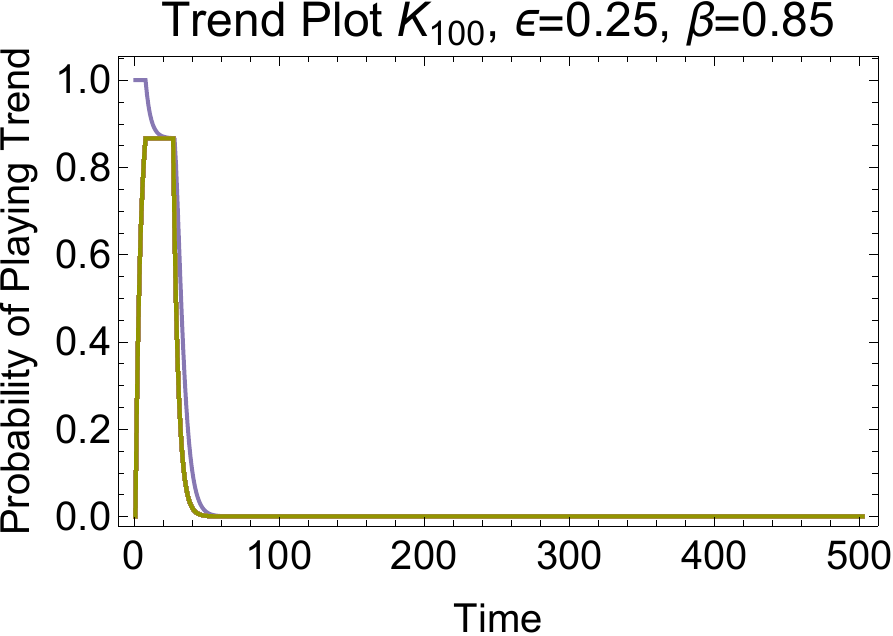}}
\caption{An illustration of the impact of $\epsilon$ and $\beta$ on trend duration in a complete graph with 100 vertices. Larger $\beta$ leads to longer trend uptake. Larger $\epsilon$ increases the rate of trend uptake as well as the rate of trend collapse. Here $P_0 = S_0 = T_0 = 0$.}
\label{fig:CompleteGraphTrend}
\end{figure}
Assume $n > 2$\footnote{The case when $n=2$ is simpler, but not obvious from the $n > 2$ case.} and  $P_0 = S_0 = T_0 = 0$. Furthermore, assume a single seed player. Computing the break-even time for the payoff to the seed player and the (identical) payoffs to the non-seed players show analysis shows that non-seed players mimic the seed player until the time exceeds:
\begin{displaymath}
t_1^* = 1 - \frac{\log\left(\frac{(n-2)PR}{((n-1)S-T)T+P(T-S)}\right)}{\log{\beta}}
\end{displaymath}
At this time, the seed player begins to imitate the non-seed players, who have fixed strategy until their original strategy ($\mathbf{e}_1$) out-performs their current strategy. Setting the computed payoffs equal yields a time:
\begin{displaymath}
t_2^* = \frac{\log \left(\frac{\beta  (1-\epsilon )^{-{t_1^*}} \left((1-\epsilon )^{t_1^*}
   (S+T-P)+P-S\right)}{R}\right)}{\log \beta}
\end{displaymath}
At all times greater than $t_2^*$, the trend collapses. This is the behavior illustrated in Figure \ref{fig:CompleteGraphTrend}.

More complex graphs can lead to distinct behavior, though like other spreading behavior, diffusion phenomena can be observed for instance in cycle graphs. We illustrate this using a cycle graph and a density plot that shows the trend diffusion pattern. Here $P_0 = S_0 = T_0 = 0$. (See Figure \ref{fig:CycleGraphTrend}.)
\begin{figure}[htbp]
\centering
\subfigure[$\beta=0.95$, $\epsilon=0.1$]{\includegraphics[scale=0.5]{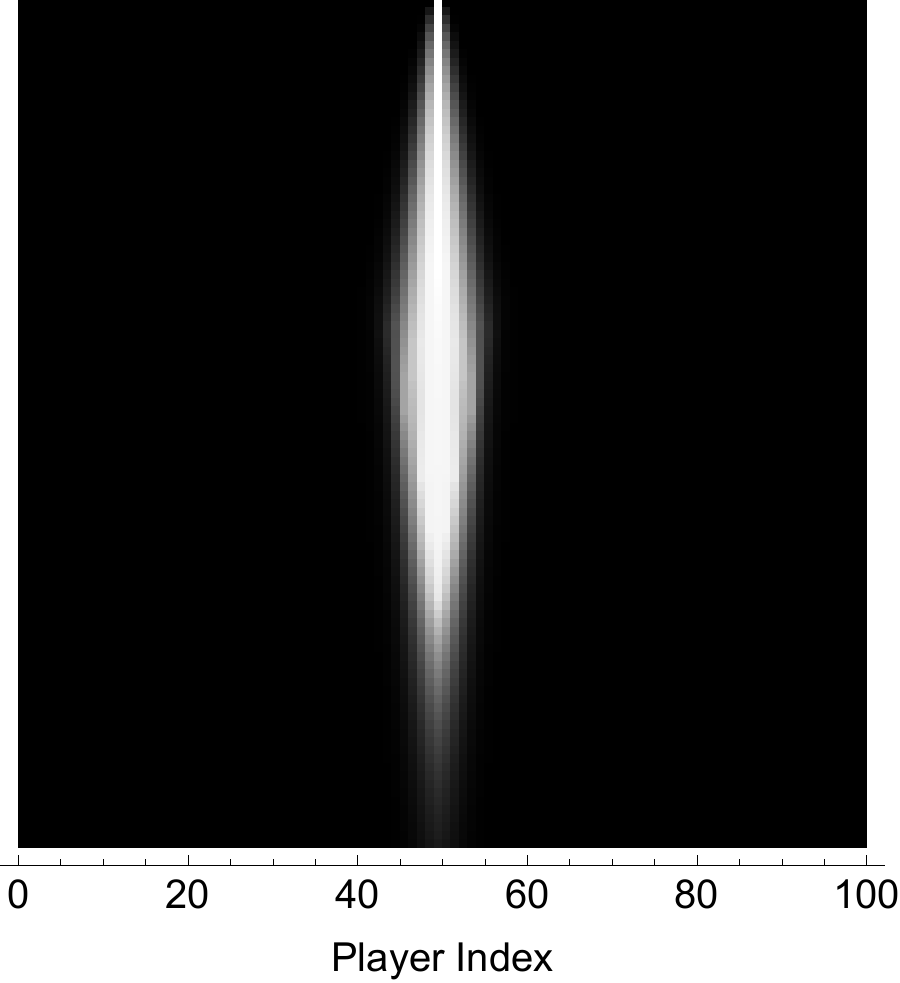}}
\subfigure[$\beta=0.95$, $\epsilon=0.25$]{\includegraphics[scale=0.5]{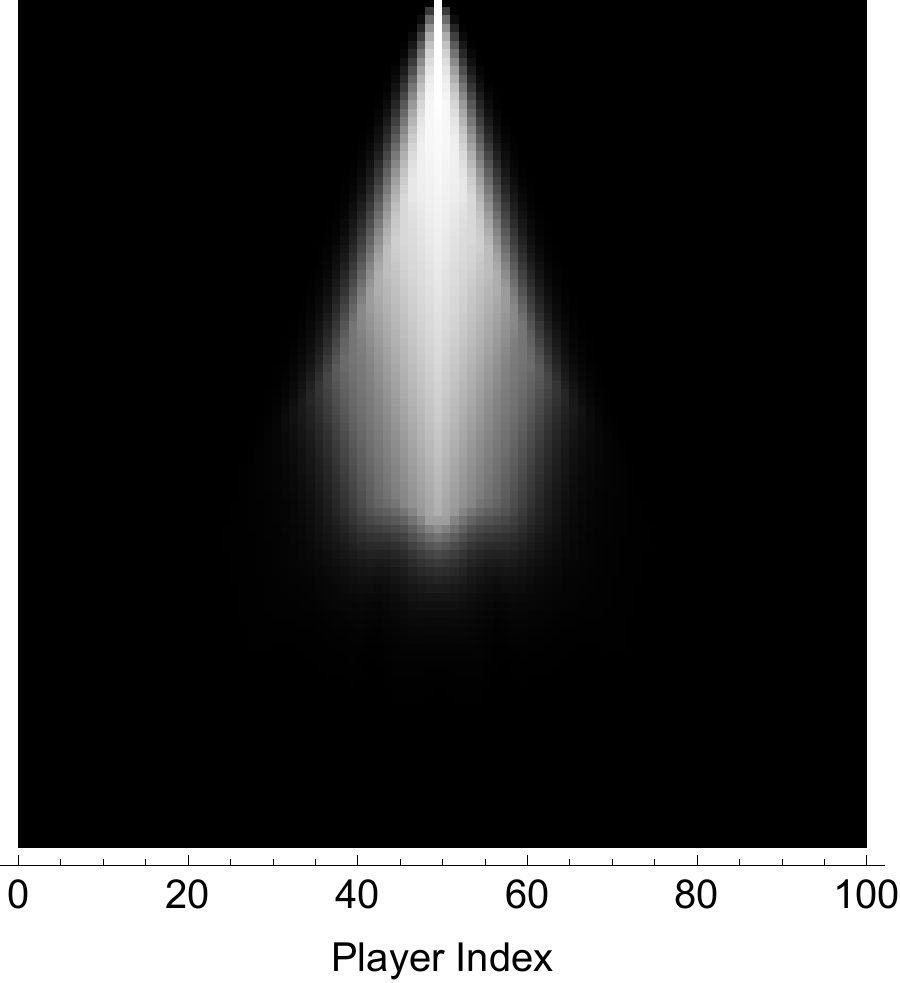}}
\subfigure[$\beta=0.85$, $\epsilon=0.1$]{\includegraphics[scale=0.5]{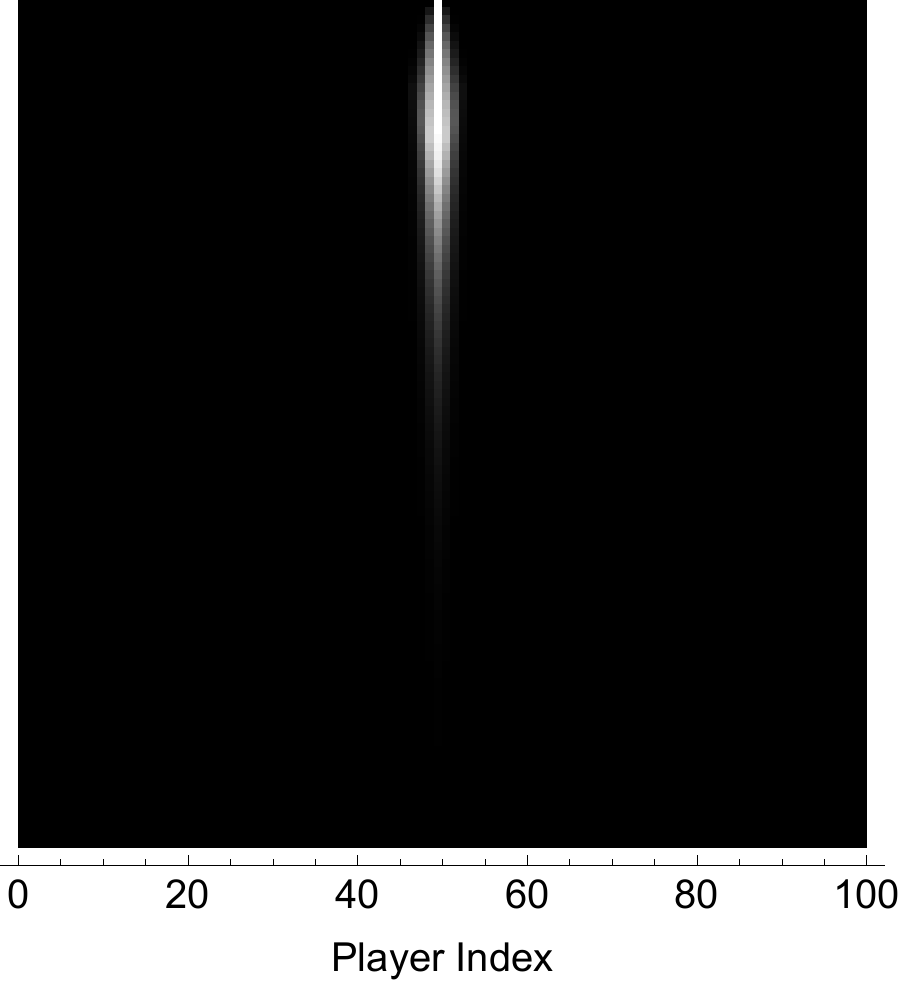}}
\subfigure[$\beta=0.85$, $\epsilon=0.25$]{\includegraphics[scale=0.5]{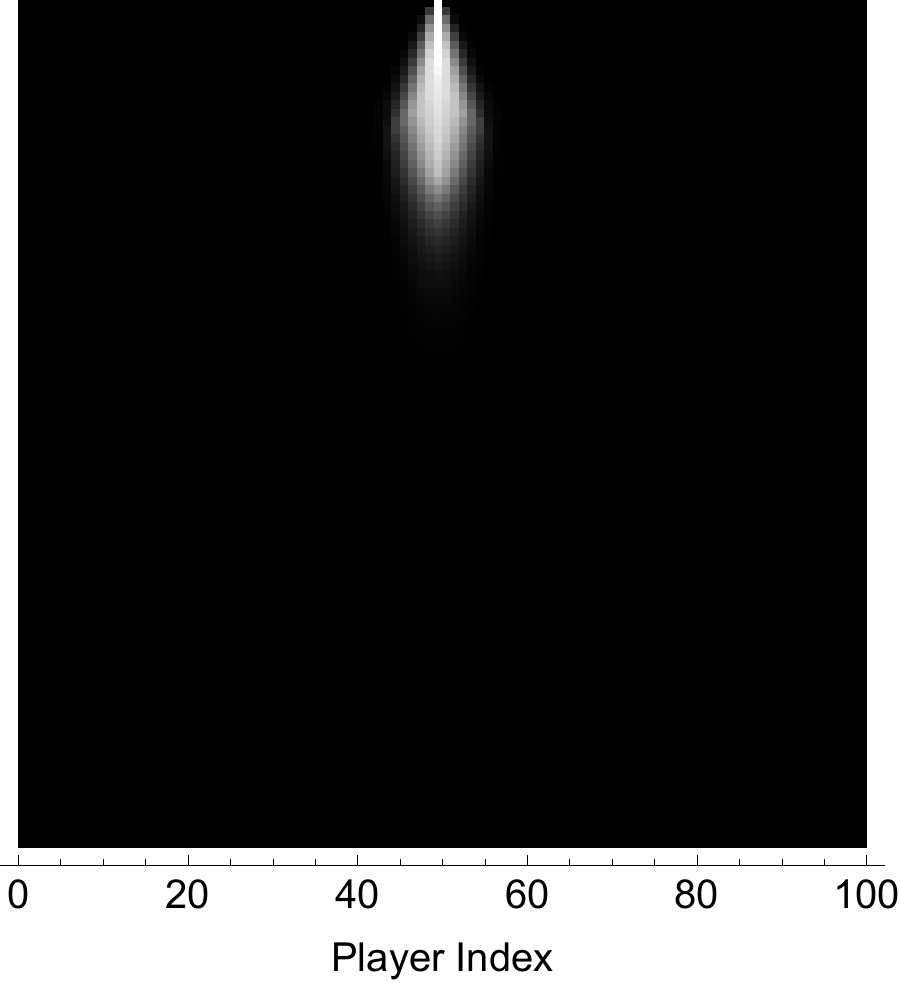}}
\caption{An illustration of the impact of $\epsilon$ and $\beta$ on trend duration and saturation in a cycle graph. Time increases down the plot. Here $P_0 = S_0 = T_0 = 0$.}
\label{fig:CycleGraphTrend}
\end{figure}
It is interesting to note that when $\beta = 0.95$ and $\epsilon = 0.25$  the trend diffuses more extensively, but also decays much more quickly when compared to the case when $\beta = 0.95$ and $\epsilon = 0.1$. This is consistent with patterns seen in fads, where uptake may be extensive and fast, but persistence and total graph saturation may still not be complete.

We complete our empirical analysis of the trend model by analyzing the spread of trends in scale-free graphs. The differences in saturation characteristics can be quite substantial because of the power-law distribution followed by vertex degrees, which correlates strongly to the payoff received by a player (see Proposition \ref{prop:MaxMin}). Figure \ref{fig:SF1} shows the same random scale-free network with two distinct seed vertices, one with high degree (27) and the other with low degree (3). Note, spatial adjacency is not captured in Figure \ref{fig:SF1}, but these figures are convenient for anecdotally comparing trend duration and saturation.
\begin{figure}[htbp]
\centering
\subfigure[$\beta=0.95$, $\epsilon=0.1$]{\includegraphics[scale=0.4]{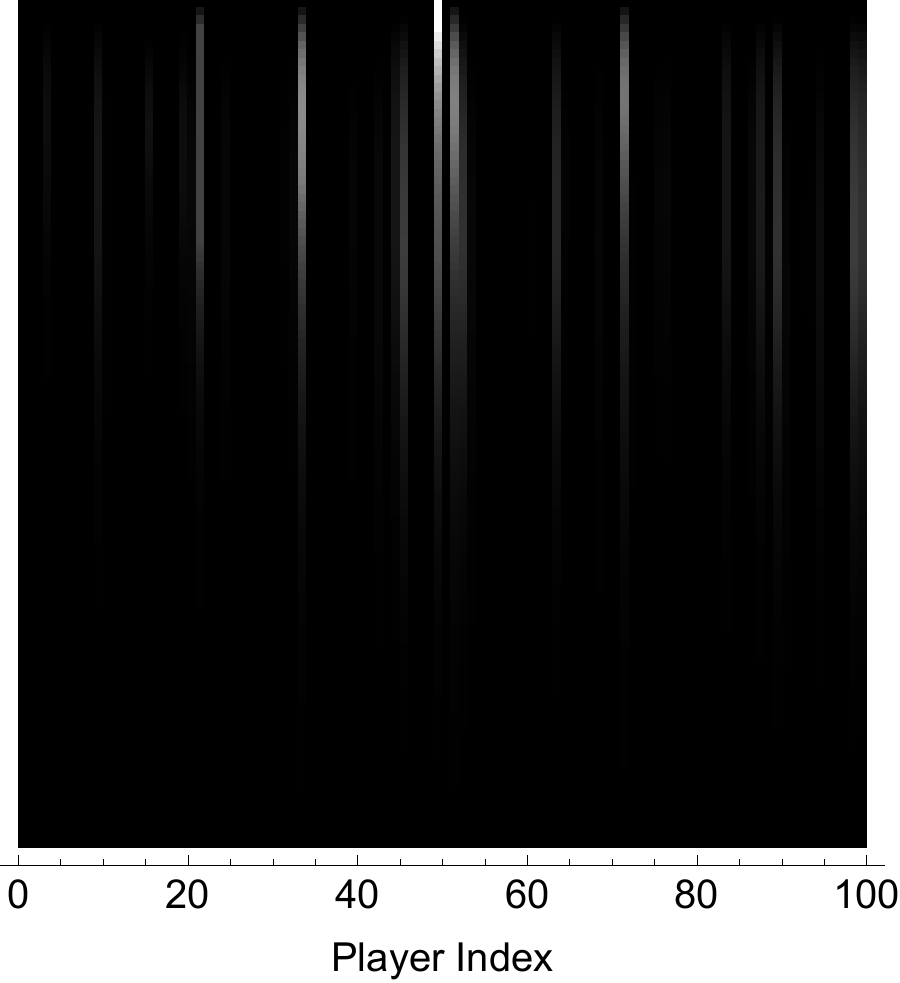}}
\subfigure[$\beta=0.95$, $\epsilon=0.25$]{\includegraphics[scale=0.4]{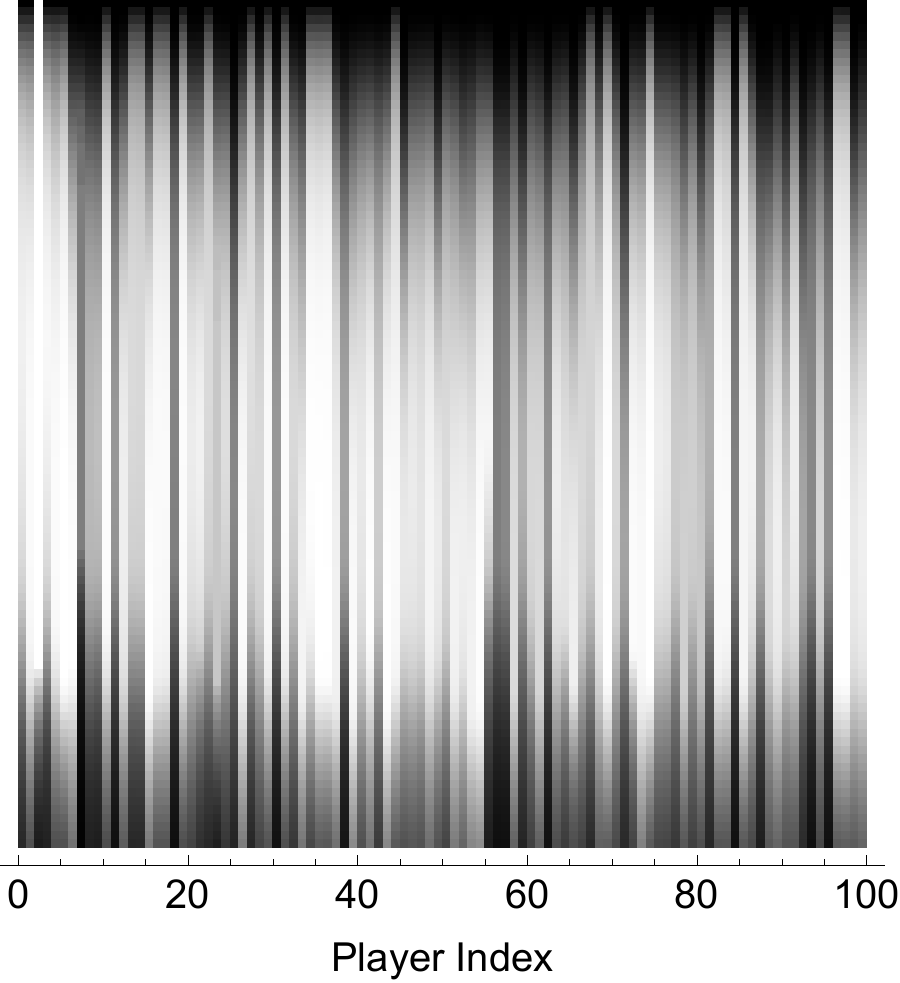}}
\caption{An illustration of the impact of player degree on trend duration and saturation in a scale-free graph. Time increases down the plot. Here $P_0 = S_0 = T_0 = 0$.}
\label{fig:SF1}
\end{figure}
Notice in Figure \ref{fig:SF1}(b), a kind of self-reinforcing trend diffusion is taking place at multiple vertices in the graph because the initial vertex was a hub.

To understand the saturation process, we analyzed trend spread in 20 randomly generated scale free graphs (using the Barabasi-Albert distribution \cite{BA99a})  each with 100 vertices for varying values of $\epsilon$ and $\beta$. We started the trend at each vertex and observed saturation when $\epsilon$ and $\beta$ ranged from $0.1$ to $0.25$ and $0.8$ to $0.95$ respectively. We found only an effect from the degree of the initial vertex, and modeled the spread as:
\begin{equation}
\pi_\mathcal{S} \sim 1 - \exp(\alpha_0 - \alpha d)
\end{equation}
where $d$ is the degree of the seed vertex. Simple regression shows $\alpha = 0.17$ and $r^2 = 0.48$; i.e., the degree of the seed vertex explains 48\% of the variation in trend saturation.

The impact of $\epsilon$ and $\beta$ on the saturation can be better understood by evaluating a larger range of values. To do so, we used a sample of 5 randomly generated scale-free graphs each with 100 vertices and allowed $\epsilon$ and $\beta$ to vary from $0.05$ to $0.95$. A density and surface plot of the mean saturation vs. $\epsilon$ and $\beta$ is shown in Figure \ref{fig:SatEB}.
\begin{figure}[htbp]
\centering
\subfigure[Density Plot]{\includegraphics[scale=0.4]{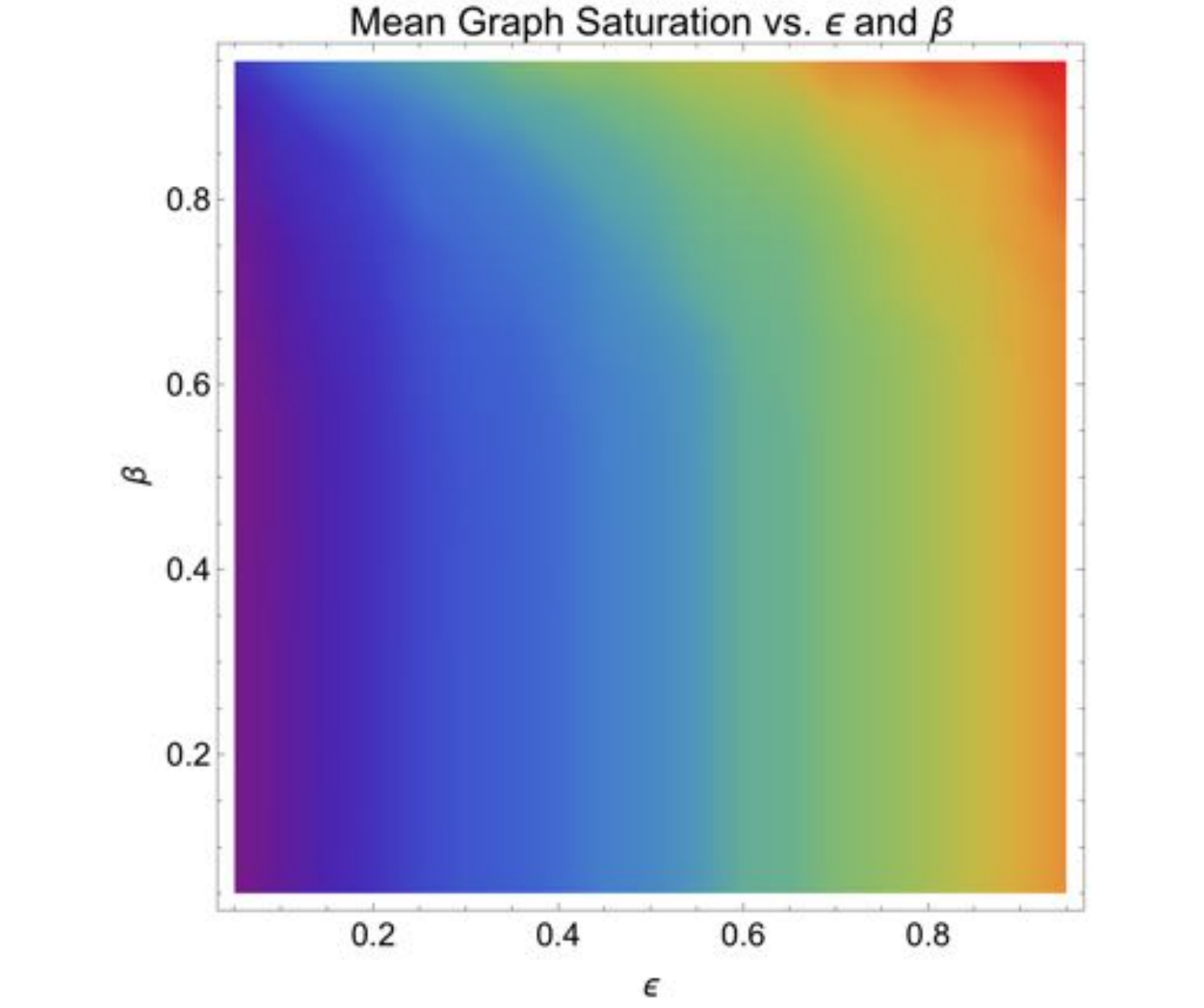}}
\subfigure[Surface Plot]{\includegraphics[scale=0.45]{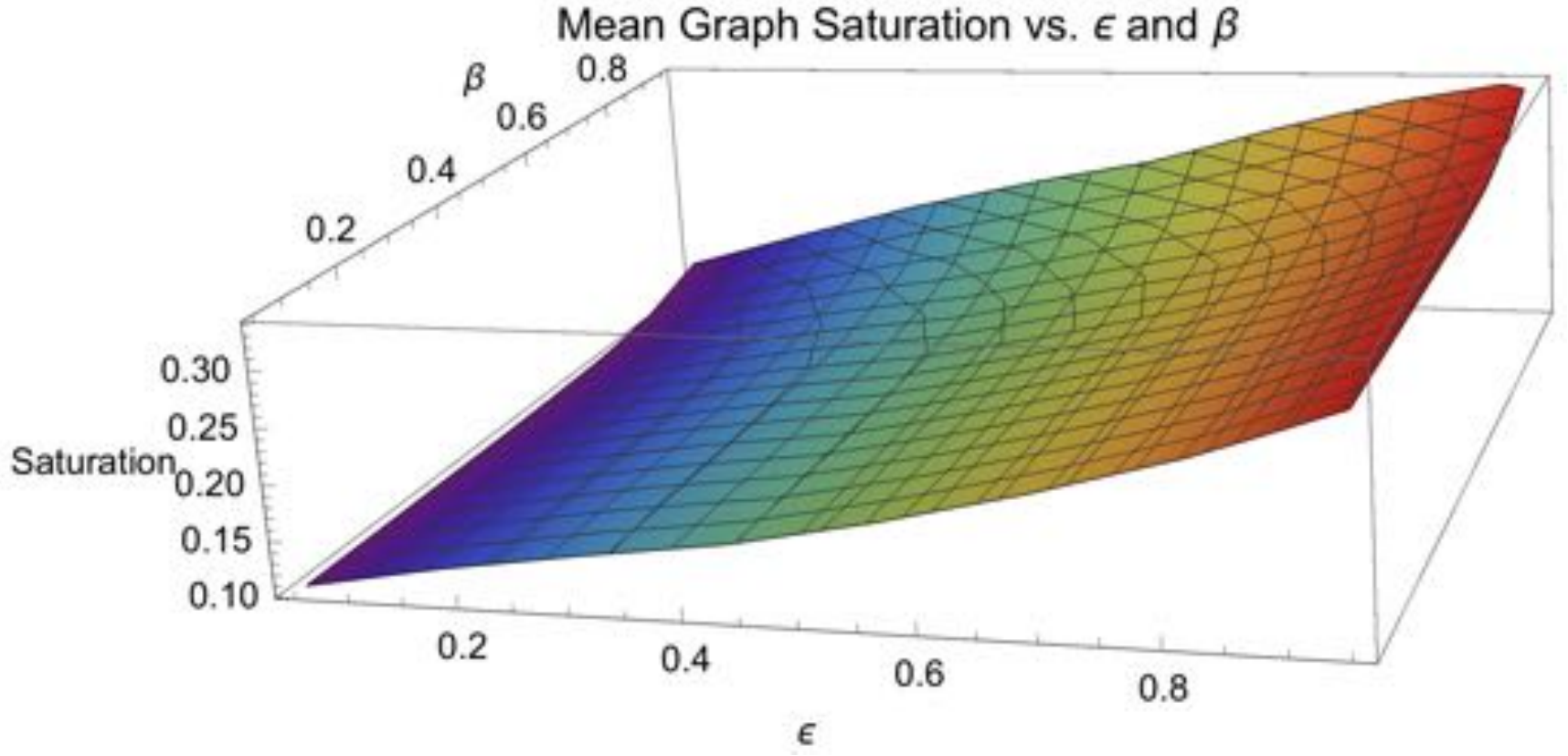}}
\caption{A density and surface plot of the mean saturation (over all graphs and seed vertices) as a function of $\epsilon$ and $\beta$. The mean increases (approximately) linearly in both parameters.}
\label{fig:SatEB}
\end{figure}
A linear fit of this data shows that:
\begin{displaymath}
\bar{\pi}_\mathcal{S} \sim 0.069 + 0.041\beta + 0.21\epsilon,
\end{displaymath}
which explains 94\% of the variation (i.e., $r^2 = 0.94$). While this explains the mean, the data show a large amount of variation because of the effect of the degree of the seed vertex, as illustrated in the four histograms in Figure \ref{fig:Histograms}. For comparison, we include the histogram of degree distributions over all five graphs used in this experiment.
\begin{figure}[htbp]
\centering
\subfigure[]{\includegraphics[scale=0.33]{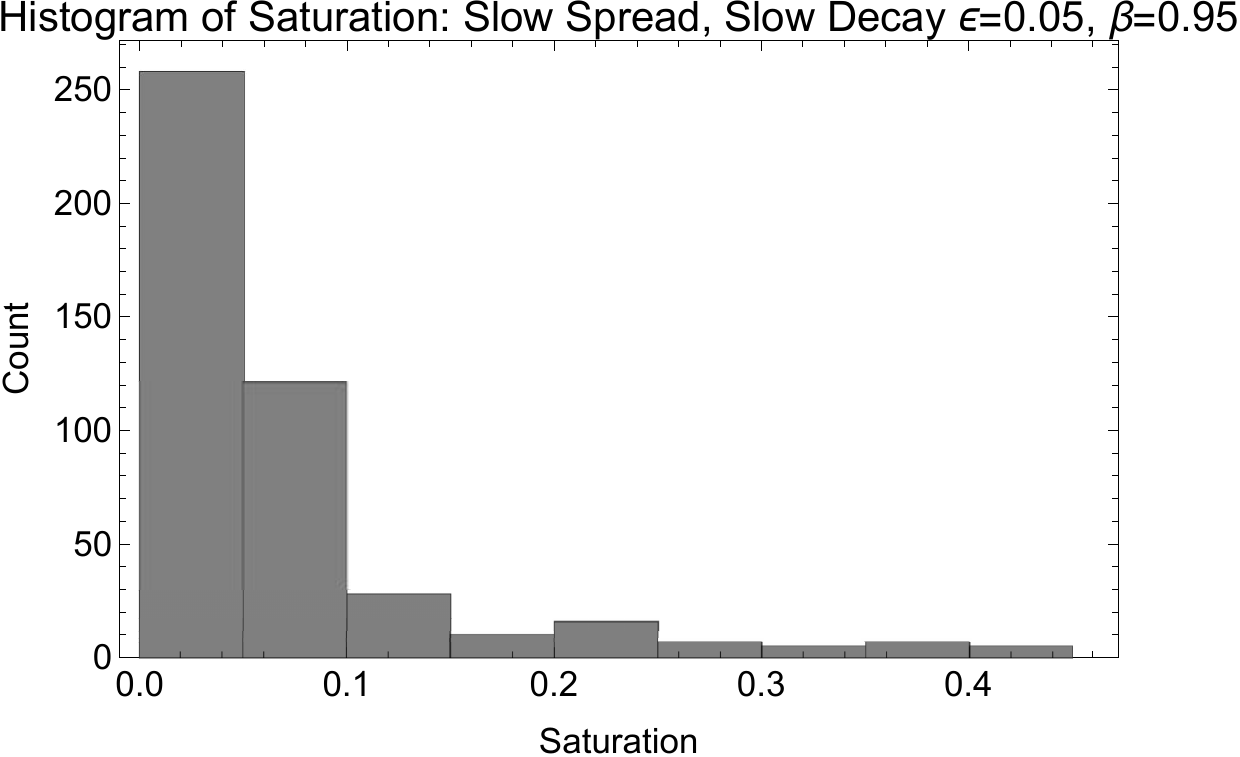}}
\subfigure[]{\includegraphics[scale=0.32]{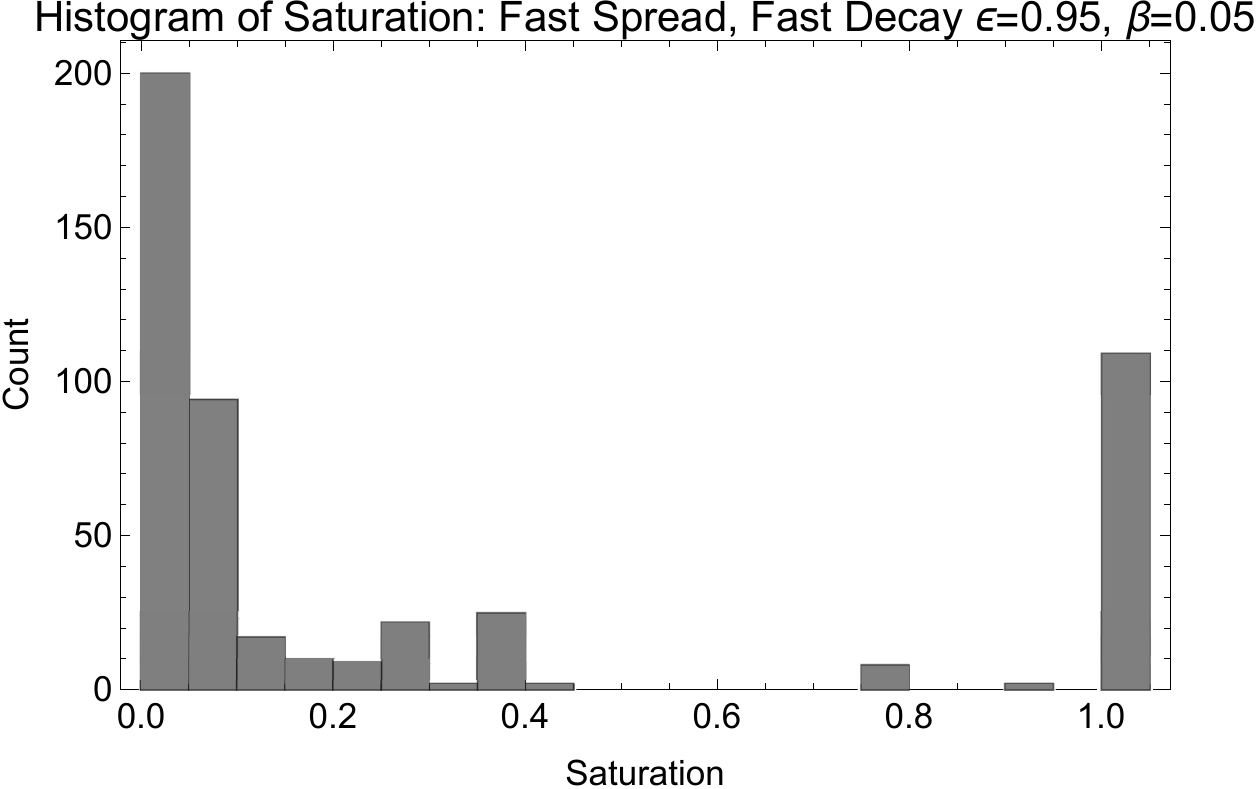}}
\subfigure[]{\includegraphics[scale=0.34]{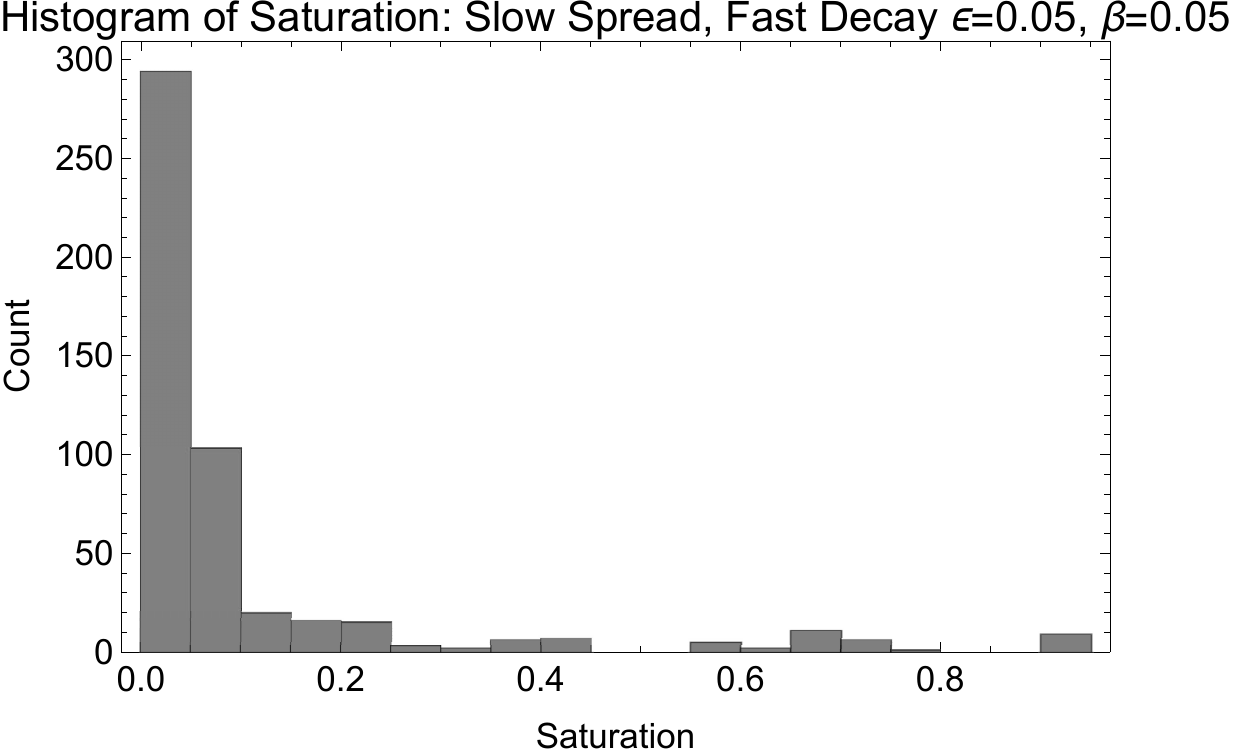}}
\subfigure[]{\includegraphics[scale=0.33]{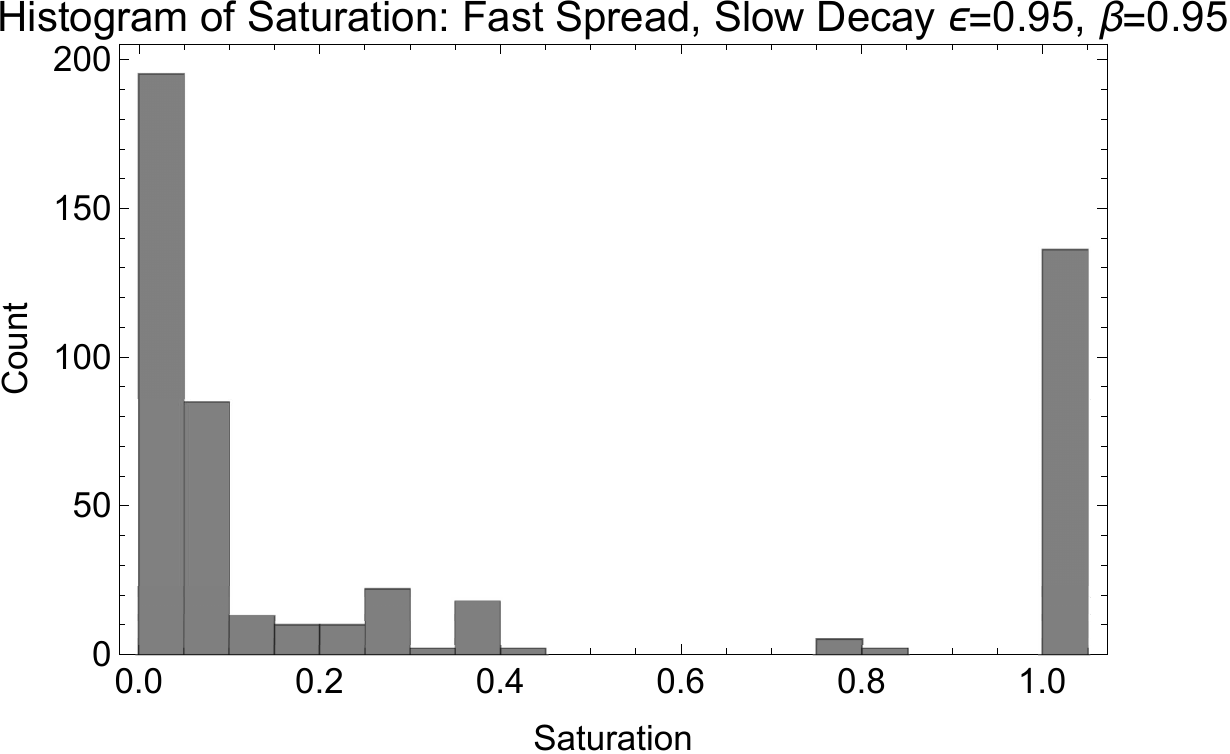}}
\subfigure[]{\includegraphics[scale=0.345]{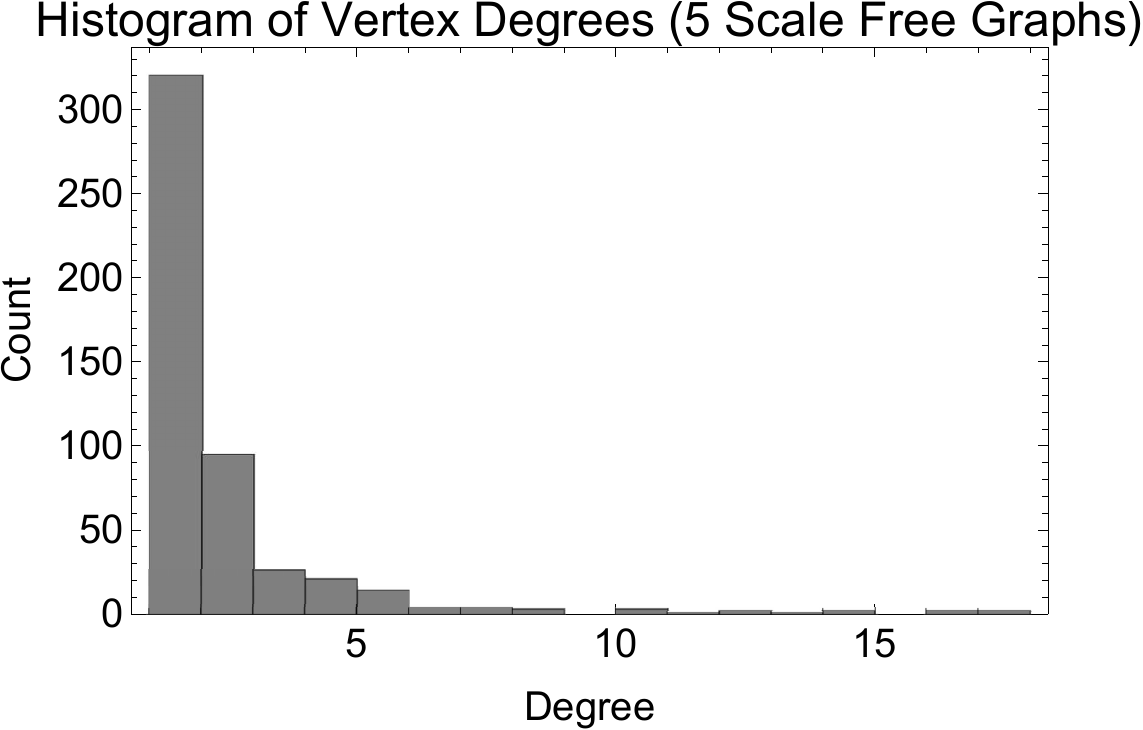}}
\caption{Histograms of trend saturation and the degree distribution in the experiment. Clearly degree distribution is related to trend saturation.}
\label{fig:Histograms}
\end{figure}
Figure \ref{fig:Histograms} also effectively illustrates the greater impact the rate of spread ($\epsilon$) has on trend saturation.

In the more general case, the stability of a starting fixed point in the presence of the trend is determined by the trend parameters and the partition $\sim_\mathcal{I}$. That is, if at trend initialization $\mathbf{x}(0) = \mathbf{x}_0$, then $\lim_{t\rightarrow \infty}\mathbf{x}(t) = \mathbf{x}_0$ if the trend does not create a trajectory that leaves the basin of attraction of $\mathbf{x}(t)$, which is partially determined by $\sim_\mathcal{I}$. Further theoretical study is clearly warranted for more general trend modeling, and is considered below.

\section{Conclusion and Future Directions}\label{sec:Conclusion}
In this paper, we studied strategic consensus in a network of agents who use a discrete time game-theoretic imitation rule. Agents imitate neighbors who are more successful than they are, without regard to the actual impact on their payoff function. Imitation is proportional to the relative success over the imitating agent. We showed that the resulting dynamics come to consensus when a certain fact about the imitation graph is true. We proposed two conjectures on general convergence of these dynamics and numerically illustrated general convergence in specific games: bistable stag-hunt game and the cyclic rock-paper-scissors game. In rock-paper-scissors, we found empirical evidence that consensus is common, but leave further analysis to future work. We also studied the network prisoner's dilemma game when the topology was allowed to change. In this case, we showed that the population divides into disconnected islands each with different cooperation levels. We then studied trends that can occur and spread on these graphs, providing empirical models of the impact of imitation rate ($\alpha$) and trend power ($\beta$) on the extensiveness of the trend spread.

There are several salient future directions that extend from this work. Completing this program of study by proving or finding counter-examples to our conjectures on convergence is essential. Proving that unbiased cyclic games almost always converge to consensus in this model is of substantial interest as it would represent a fundamental result on a large class of games that normally induce oscillations in evolutionary game theory \cite{EGB16}.

Analysis of the continuous time dynamics may yield novel behaviors not observed in the discrete time dynamics and may make a convergence proof or counter-example easier to find. In particular, we would be interested in counter examples of strategy oscillation as a result of the discontinuity in the right-hand-side of the dynamics, since we have not identified any numerical examples where the system failed to converge. Additionally, as a by-product of the imitation model, we observed a potentially new community detection algorithm (see Figure \ref{fig:IGraph}(b)) that provided consistent results with the maximum modularity community detection approach. Further evaluation is required to verify and study this. We provided only elementary theoretical results for the trend model with some empirical analysis. Further theoretical and empirical analysis is required. Additionally, it would be useful to determine whether this model is valid for real trends by attempting to fit a large-scale data set. Finally, analysis in the continuous time case following (e.g., \cite{CS07}) may lead to additional dynamics not observed in the discrete time case.

\section*{Acknowledgements}
The proof of Lemma \ref{lem:ConsensusLemma} that appears in this work was suggested by Eitan Tadmor. We gratefully acknowledge his helpful comments on this manuscript. We also acknowledge and thank the reviewers and editors for their time in helping us revise this paper.

\bibliography{BibDeviantBehavior}

\begin{thebibliography}{100}
\providecommand{\url}[1]{#1}
\csname url@samestyle\endcsname
\providecommand{\newblock}{\relax}
\providecommand{\bibinfo}[2]{#2}
\providecommand{\BIBentrySTDinterwordspacing}{\spaceskip=0pt\relax}
\providecommand{\BIBentryALTinterwordstretchfactor}{4}
\providecommand{\BIBentryALTinterwordspacing}{\spaceskip=\fontdimen2\font plus
\BIBentryALTinterwordstretchfactor\fontdimen3\font minus
  \fontdimen4\font\relax}
\providecommand{\BIBforeignlanguage}[2]{{%
\expandafter\ifx\csname l@#1\endcsname\relax
\typeout{** WARNING: IEEEtran.bst: No hyphenation pattern has been}%
\typeout{** loaded for the language `#1'. Using the pattern for}%
\typeout{** the default language instead.}%
\else
\language=\csname l@#1\endcsname
\fi
#2}}
\providecommand{\BIBdecl}{\relax}
\BIBdecl

\bibitem{MT14}
S.~Motsch and E.~Tadmor, ``{Heterophilious Dynamics Enhances Consensus},''
  \emph{SIAM Review}, vol.~56, pp. 577--621, 2014.

\bibitem{DeGroot74}
M.~H. DeGroot, ``Reaching a consensus,'' \emph{J. American Stat. Association},
  vol.~69, pp. 118--121, 1974.

\bibitem{Krause00}
U.~Krause, ``A discrete nonlinear and non-autonomous model of consensus
  formation,'' in \emph{In Communications in Difference Equations}, Gordon and
  Breach, Eds., 2000, pp. 227-- 236.

\bibitem{Centola15}
D.~Centola and A.~Baronchelli, ``Flocks, herds, and schools: A quantitative
  theory of flocking,'' \emph{Proceedings of the National Academy of Sciences},
  vol. 112, pp. 1989--1994, 2015.

\bibitem{TT98}
J.~Toner and Y.~Tu, ``Flocks, herds, and schools: A quantitative theory of
  flocking,'' \emph{Physical Review E}, vol.~58, p. 4828, 1998.

\bibitem{CS07}
F.~Cucker and S.~Smale, ``Emergent behavior in flocks,'' \emph{IEEE
  Transactions on Automatic Control}, vol.~52, no.~5, pp. 852--862, 2007.

\bibitem{EK01}
L.~Edelstein-Keshet, ``Mathematical models of swarming and social
  aggregation,'' in \emph{Proc. 2001 International Symposium on Nonlinear
  Theory and Its Applications (NOLTA 2001)}, Miyagi, Japan, Oct. 28 - Nov. 1
  2001.

\bibitem{L08}
W.~Li, ``Stability analysis of swarms with general topology,'' \emph{IEEE
  Trans. Systems, Man and Cybernetics, Part B}, vol.~38, pp. 1084--1097, 2008.

\bibitem{LX10}
X.~Li and J.~Xiao, ``Swarming in homogeneous environments: A social interaction
  based framework,'' \emph{J. Theoret. Biol.}, vol. 264, pp. 747--759, 2010.

\bibitem{DM11}
P.~Degond and S.~Motsch, ``A macroscopic model for a system of swarming agents
  using curvature control,'' \emph{J. Stat. Phys.}, vol. 143, no. 685-714,
  2011.

\bibitem{VZ12}
T.~Vicsek and A.~Zefeiris, ``Collective motion,'' \emph{Phys. Preprints},
  vol.~75, pp. 71--140, 2012.

\bibitem{OFM07}
R.~Olfati-Saber, J.~A. Fax, and R.~M. Murray, ``Consensus and cooperation in
  networked multi-agent systems,'' \emph{Proceedings of the IEEE}, vol.~95,
  no.~1, pp. 215--233, Jan 2007.

\bibitem{HK02}
R.~Hegselmann and U.~Krause, ``Opinion dynamics and bounded confidence: Models,
  analysis and simulation,'' \emph{J. Artificial Soc. Social Simul.}, vol.~5,
  2002.

\bibitem{BN05}
E.~Ben-Naim, ``Opinion dynamics: Rise and fall of political parties,''
  \emph{Europhys. Lett.}, vol.~69, pp. 671--677, 2005.

\bibitem{WDA05}
G.~Weisbuch, G.~Deffuant, and F.~Amblard, ``Persuasion dynamics,''
  \emph{Physica A}, vol. 353, 2005.

\bibitem{Toscani06}
G.~Toscani, ``Kinetic models of opinion formation,'' \emph{Commun. Math. Sci.},
  vol.~4, pp. 481--496, 2006.

\bibitem{Weisb06}
G.~Weisbuch, ``Social opinion dynamics,'' in \emph{Econophysics and
  Sociophysics: Trends and Perspectives}, B.~K. Chakrabarti, A.~Chakrabarti,
  and A.~Chatterjee, Eds.\hskip 1em plus 0.5em minus 0.4em\relax Wiley, 2006,
  pp. 67--94.

\bibitem{Lorenz07}
J.~Lorenz, ``Continuous opinion dynamics of multidimensional allocation
  problems under bounded confidence. a survey,'' \emph{Internat. J. Modern
  Phys. C}, vol.~18, pp. 1819--1838, 2007.

\bibitem{BHT09}
V.~D. Blondel, J.~M. Hendrickx, and J.~N. Tsitsiklis, ``On krause's multi-agent
  consensus model with state-dependent connectivity,'' \emph{IEEE Transactions
  on Automatic Control}, vol.~54, no.~11, pp. 2586--2597, Nov 2009.

\bibitem{CFL09}
C.~Castellano, S.~Fortunato, and V.~Loreto, ``Statistical physics of social
  dynamics,'' \emph{Rev. Modern Phys.}, vol.~81, pp. 591--646, 2009.

\bibitem{KR11}
S.~Kurz and J.~Rambau, ``On the hegselmann-krause conjecture in opinion
  dynamics,'' \emph{J. Difference Equ. Appl.}, vol.~17, pp. 859--876, 2011.

\bibitem{DMPW12}
B.~Duering, P.~Markowich, J.~F. Pietschmann, and M.~T. Wolfram, ``Boltzmann and
  {F}okker-{P}lanck equations modelling opinion formation in the presence of
  strong leaders,'' \emph{Proc. R. Soc. Lond. Ser. A}, vol. 465, no. 3678-3708,
  2012.

\bibitem{CFT12}
C.~Canuto, F.~Fagnani, and P.~Tilli, ``An eulerian approach to the analysis of
  krause's consensus models,'' \emph{SIAM J. Contr. and Opt.}, pp. 243--265,
  2012.

\bibitem{JM14}
\BIBentryALTinterwordspacing
P.-E. Jabin and S.~Motsch, ``Clustering and asymptotic behavior in opinion
  formation,'' \emph{Journal of Differential Equations}, vol. 257, no.~11, pp.
  4165 -- 4187, 2014. [Online]. Available:
  \url{http://www.sciencedirect.com/science/article/pii/S002203961400326X}
\BIBentrySTDinterwordspacing

\bibitem{HT08}
S.~Y. Ha and E.~Tadmor, ``From particle to kinetic and hydrodynamic
  descriptions of flocking,'' \emph{Kinetic Related Models}, vol.~1, no.
  415-435, 2008.

\bibitem{HL09}
S.~Y. Ha and J.~G. Liu, ``A simple proof of the cucker-smale flocking dynamics
  and mean-field limit,'' \emph{Commun. Math. Sci.}, vol.~7, no. 297-325, 2009.

\bibitem{CFRT10}
J.~A. Carrillo, M.~Fornasier, J.~Rosado, and G.~Toscani, ``Asymptotic flocking
  dynamics for the kinetic cucker--smale model,'' \emph{SIAM J. Math. Anal.},
  vol.~42, pp. 218--236, 2010.

\bibitem{MOA10}
N.~Mecholsky, E.~Ott, and T.~M. Antonsen, ``Obstacle and predator avoidance in
  a model for flocking,'' \emph{Phys. D}, vol. 239, pp. 988--996, 2010.

\bibitem{Hask13}
J.~Haskovec, ``Flocking dynamics and mean field limit of the
  {C}ucker-{S}male-type model with topological interactions,'' \emph{Phys. D},
  vol. 261, pp. 42--51, 2013.

\bibitem{EHS16}
\BIBentryALTinterwordspacing
R.~Erban, J.~Ha{\v s}kovec, and Y.~Sun, ``A cucker--smale model with noise and
  delay,'' \emph{SIAM Journal on Applied Mathematics}, vol.~76, no.~4, pp.
  1535--1557, 2016. [Online]. Available:
  \url{http://dx.doi.org/10.1137/15M1030467}
\BIBentrySTDinterwordspacing

\bibitem{DY00}
A.~A. Dragulescu and V.~M. Yakovenko, ``Statistical mechanics of money,''
  \emph{European Phys. J. B}, vol.~17, pp. 723--729, 2000.

\bibitem{MSC01}
M.~McPherson, L.~Smith-Lovin, and J.~M. Cook, ``Birds of a feather: Homophily
  in social networks,'' \emph{Ann. Rev. Sociology}, vol.~27, pp. 415--444,
  2001.

\bibitem{OM04}
R.~Olfati-Saber and R.~M. Murray, ``Consensus problems in network of agents
  with switching topology and time delays,'' \emph{IEEE Trans. Automatic
  Control}, vol.~49, no.~9, pp. 1520--1533, 2004.

\bibitem{J08}
M.~Jackson, \emph{{Social and Economic Networks}}.\hskip 1em plus 0.5em minus
  0.4em\relax Princeton University Press, 2008.

\bibitem{H10}
D.~Helbing, \emph{Quantitative Sociodynamics: Stochastic Methods and Models of
  Social Interaction Processes}.\hskip 1em plus 0.5em minus 0.4em\relax
  Springer-Verlag, New York, 2010.

\bibitem{DdGL10}
P.~DeLellis, M.~diBernardo, F.~Garofalo, and D.~Liuzza, ``Analysis and
  stability of consensus in networked control systems,'' \emph{Appl. Math.
  Comput.}, vol. 217, pp. 988--1000, 2010.

\bibitem{BHT13}
N.~Bellomo, M.~Herrero, and A.~Tosin, ``On the dynamics of social conflicts:
  Looking for the black swan,'' \emph{Kinetic Related Models}, vol.~6, pp.
  459--479, 2013.

\bibitem{BCCC08}
M.~Ballerini, N.~Cabibbo, R.~Candelier, A.~Cavagna, E.~Cisbani, I.~Giardina,
  V.~Lecomte, A.~Orlandi, G.~Parisi, A.~Procaccini, M.~Viale, and
  V.~Zdravkovic, ``Interaction ruling animal collective behavior depends on
  topological rather than metric distance,'' \emph{Proc. Natl. Acad. Sci.},
  vol. 105, pp. 1232--1237, 2008.

\bibitem{HH08}
C.~K. Hemelrijk and H.~Hildenbrandt, ``Self-organized shape and frontal density
  of fish schools,'' \emph{Ethology}, vol. 114, pp. 245--254, 2008.

\bibitem{CCGP10}
A.~Cavagna, A.~Cimarelli, I.~Giardina, G.~Parisi, R.~Santagati, F.~Stefanini,
  and M.~Viale, ``Scale-free correlations in starling flocks,'' \emph{Proc.
  Natl. Acad. Sci.}, vol. 107, pp. 11\,865--11\,870, 2010.

\bibitem{BHOT05}
V.~D. Blondel, J.~M. Hendrickx, A.~Olshevsky, and J.~N. Tsitsiklis,
  ``Convergence in multiagent coordination, consensus, and flocking,'' in
  \emph{Proceedings of the 44th IEEE Conference on Decision and Control}, Dec
  2005, pp. 2996--3000.

\bibitem{CFSZ08}
R.~Carli, F.~Fagnani, A.~Speranzon, and S.~Zampieri, ``Communication
  constraints in the average consensus problem,'' \emph{Automatica}, vol.~44,
  pp. 671--684, 2008.

\bibitem{CS11}
G.~de~Campos and A.~Seuret, ``Improved consensus algorithms using memory
  effects,'' in \emph{Proc. 50th IEEE Conference on Decision and Control},
  2011.

\bibitem{DdG09}
P.~DeLellis, M.~diBernardo, and F.~Garofalo, ``Novel decentralized adaptive
  strategies for the synchronization of complex networks,'' \emph{Automatica},
  vol.~45, pp. 1312--1318, 2009.

\bibitem{FZ08}
F.~Fagnani and S.~Zampieri, ``Randomized consensus algorithms over large scale
  networks,'' \emph{IEEE J. Selected Areas Commun.}, vol.~26, pp. 634--649,
  2008.

\bibitem{RSGK16}
\BIBentryALTinterwordspacing
S.~Rajtmajer, A.~Squicciarini, C.~Griffin, S.~Karumanchi, and A.~Tyagi,
  ``Constrained social-energy minimization for multi-party sharing in online
  social networks,'' in \emph{Proceedings of the 2016 International Conference
  on Autonomous Agents \&\#38; Multiagent Systems}, 2016, pp. 680--688.
  [Online]. Available: \url{http://dl.acm.org/citation.cfm?id=2936924.2937025}
\BIBentrySTDinterwordspacing

\bibitem{PMC16}
A.~V. Proskurnikov, A.~S. Matveev, and M.~Cao, ``Opinion dynamics in social
  networks with hostile camps: Consensus vs. polarization,'' \emph{IEEE
  Transactions on Automatic Control}, vol.~61, no.~6, pp. 1524--1536, June
  2016.

\bibitem{WTCZ16}
X.~Wu, Y.~Tang, J.~Cao, and W.~Zhang, ``Distributed consensus of stochastic
  delayed multi-agent systems under asynchronous switching,'' \emph{IEEE
  Transactions on Cybernetics}, vol.~46, no.~8, pp. 1817--1827, Aug 2016.

\bibitem{MRC16}
J.~Mei, W.~Ren, and J.~Chen, ``Distributed consensus of second-order
  multi-agent systems with heterogeneous unknown inertias and control gains
  under a directed graph,'' \emph{IEEE Transactions on Automatic Control},
  vol.~61, no.~8, pp. 2019--2034, Aug 2016.

\bibitem{CD19}
\BIBentryALTinterwordspacing
F.~Cucker and J.~Dong, ``On flocks under switching directed interaction
  topologies,'' \emph{SIAM Journal on Applied Mathematics}, vol.~79, no.~1, pp.
  95--110, 2019. [Online]. Available: \url{https://doi.org/10.1137/18M116976X}
\BIBentrySTDinterwordspacing

\bibitem{Hofbauer98}
J.~Hofbauer and K.~Sigmund, \emph{{Evolutionary Games and Population
  Dynamics}}.\hskip 1em plus 0.5em minus 0.4em\relax Cambridge University
  Press, 1998.

\bibitem{ON08}
\BIBentryALTinterwordspacing
H.~Ohtsuki and M.~A. Nowak, ``Evolutionary stability on graphs,'' \emph{Journal
  of Theoretical Biology}, vol. 251, no.~4, pp. 698 -- 707, 2008. [Online].
  Available:
  \url{http://www.sciencedirect.com/science/article/pii/S0022519308000192}
\BIBentrySTDinterwordspacing

\bibitem{ON06}
\BIBentryALTinterwordspacing
------, ``The replicator equation on graphs,'' \emph{Journal of Theoretical
  Biology}, vol. 243, no.~1, pp. 86 -- 97, 2006. [Online]. Available:
  \url{http://www.sciencedirect.com/science/article/pii/S0022519306002426}
\BIBentrySTDinterwordspacing

\bibitem{Wei95}
J.~W. Weibull, \emph{Evolutionary Game Theory}.\hskip 1em plus 0.5em minus
  0.4em\relax MIT Press, 1997.

\bibitem{EGB16}
G.~B. Ermentrout, C.~Griffin, and A.~Belmonte, ``Transition matrix model for
  evolutionary game dynamics,'' \emph{Physical Review E}, vol.~93, p. 032138,
  2016.

\bibitem{H09}
I.~I. Hussein, ``An individual-based evolutionary dynamics model for networked
  social behaviors,'' in \emph{2009 American Control Conference}, June 2009,
  pp. 5789--5796.

\bibitem{PQ12}
A.~Pantoja and N.~Quijano, ``Distributed optimization using population dynamics
  with a local replicator equation,'' in \emph{2012 IEEE 51st IEEE Conference
  on Decision and Control (CDC)}, Dec 2012, pp. 3790--3795.

\bibitem{MM15}
D.~Madeo and C.~Mocenni, ``Game interactions and dynamics on networked
  populations,'' \emph{IEEE Transactions on Automatic Control}, vol.~60, no.~7,
  pp. 1801--1810, July 2015.

\bibitem{GTBS16}
B.~Gharesifard, B.~Touri, T.~Ba{\c s}ar, and J.~Shamma, ``On the convergence of
  piecewise linear strategic interaction dynamics on networks,'' \emph{IEEE
  Transactions on Automatic Control}, vol.~61, no.~6, pp. 1682--1687, June
  2016.

\bibitem{JLM03}
A.~Jadbabaie, J.~Lin, and A.~S. Morse, ``Coordination of groups of mobile
  autonomous agents using nearest neighbor rules,'' \emph{IEEE Transactions on
  Automatic Control}, vol.~48, no.~6, pp. 988--1001, June 2003.

\bibitem{JK10}
E.~W. Justh and P.~S. Krishnaprasad, ``Extremal collective behavior,'' in
  \emph{In Proc. 49th Conference on Decision and Control}, 2010, pp.
  5432--5437.

\bibitem{RGMS15}
\BIBentryALTinterwordspacing
S.~Rajtmajer, C.~Griffin, D.~Mikesell, and A.~Squicciarini, ``An evolutionary
  game model for the spread of non-cooperative behavior in online social
  networks,'' in \emph{Proceedings of the 30th Annual ACM Symposium on Applied
  Computing}, ser. SAC '15.\hskip 1em plus 0.5em minus 0.4em\relax New York,
  NY, USA: ACM, 2015, pp. 1154--1159. [Online]. Available:
  \url{http://doi.acm.org/10.1145/2695664.2695867}
\BIBentrySTDinterwordspacing

\bibitem{BGM13}
K.~Bhawalkar, S.~Gollapudi, and K.~Munagala, ``{Coevolutionary Opinion
  Formation Games},'' in \emph{Proceedings of the 45th ACM Symposium on the
  Theory of Computing (STOC'13)}, 2013, pp. 41 -- 50.

\bibitem{BHW92}
S.~Bikhchandani, D.~Hirshleifer, and I.~Welch, ``A theory of fads, fashion,
  custom, and cultural change as informational cascades,'' \emph{Journal of
  Political Economy}, vol. 100, pp. 992--1026, 1992.

\bibitem{H00}
H.~W. Hethcote, ``The mathematics of infectious diseases,'' \emph{SIAM Review},
  vol.~42, pp. 599--653, 2000.

\bibitem{KGH14}
\BIBentryALTinterwordspacing
A.~D.~I. Kramer, J.~E. Guillory, and J.~T. Hancock, ``Experimental evidence of
  massive-scale emotional contagion through social networks,''
  \emph{Proceedings of the National Academy of Sciences}, vol. 111, no.~24, pp.
  8788--8790, 2014. [Online]. Available:
  \url{http://www.pnas.org/content/111/24/8788.abstract}
\BIBentrySTDinterwordspacing

\bibitem{CSKM14}
\BIBentryALTinterwordspacing
L.~Coviello, Y.~Sohn, A.~D.~I. Kramer, C.~Marlow, M.~Franceschetti, N.~A.
  Christakis, and J.~H. Fowler, ``Detecting emotional contagion in massive
  social networks,'' \emph{PLoS ONE}, vol.~9, no.~3, p. e90315, 03 2014.
  [Online]. Available: \url{http://dx.doi.org/10.1371%2Fjournal.pone.0090315}
\BIBentrySTDinterwordspacing

\bibitem{CF13}
\BIBentryALTinterwordspacing
N.~A. Christakis and J.~H. Fowler, ``Social contagion theory: examining dynamic
  social networks and human behavior,'' \emph{Statistics in Medicine}, vol.~32,
  no.~4, pp. 556--577, 2013. [Online]. Available:
  \url{http://dx.doi.org/10.1002/sim.5408}
\BIBentrySTDinterwordspacing

\bibitem{CF07}
\BIBentryALTinterwordspacing
------, ``The spread of obesity in a large social network over 32 years,''
  \emph{New England Journal of Medicine}, vol. 357, no.~4, pp. 370--379, 2007,
  pMID: 17652652. [Online]. Available:
  \url{http://www.nejm.org/doi/full/10.1056/NEJMsa066082}
\BIBentrySTDinterwordspacing

\bibitem{CF08}
\BIBentryALTinterwordspacing
E.~Cohen-Cole and J.~M. Fletcher, ``Is obesity contagious? social networks vs.
  environmental factors in the obesity epidemic,'' \emph{Journal of Health
  Economics}, vol.~27, no.~5, pp. 1382 -- 1387, 2008. [Online]. Available:
  \url{http://www.sciencedirect.com/science/article/pii/S0167629608000362}
\BIBentrySTDinterwordspacing

\bibitem{JV05}
M.~Jackson and A.~Van Den~Nouweland, ``{Strongly stable networks},''
  \emph{Journal of Games and Economic Behavior}, vol.~51, no.~2, pp. 420--444,
  2005.

\bibitem{J03}
M.~Jackson, ``{A survey of models of network formation: Stability and
  efficiency},'' \emph{Game Theory and Information}, 2003.

\bibitem{JW96}
M.~Jackson and A.~Wolinsky, ``{A strategic model of social and economic
  networks},'' \emph{Journal of economic theory}, vol.~71, no.~1, pp. 44--74,
  1996.

\bibitem{BJG09}
J.~Bang-Jensen and G.~Gutin, \emph{Digraphs: Theory, Algorithms and
  Applications}, 2nd~ed.\hskip 1em plus 0.5em minus 0.4em\relax Springer, 2009.

\bibitem{B22}
S.~Banach, ``Sur les op\'{e}rations dans les ensembles abstraits et leur
  application aux \'{e}quations int\'{e}grales,'' \emph{Fund. Math.}, vol.~3,
  pp. 133--181, 1922.

\bibitem{W71}
\BIBentryALTinterwordspacing
R.~F. Williams, ``Composition of contractions,'' \emph{Boletim da Sociedade
  Brasileira de Matem{\'a}tica}, vol.~2, no.~2, pp. 55--59, Sep 1971. [Online].
  Available: \url{https://doi.org/10.1007/BF02584684}
\BIBentrySTDinterwordspacing

\bibitem{S67}
S.~Smale, ``Differentiable dynamical systems,'' \emph{Bulletin of the American
  mathematical Society}, vol.~73, no.~6, pp. 747--817, 1967.

\bibitem{N68}
S.~Nadler, ``Sequences of contractions and fixed points,'' \emph{Pacific
  Journal of Mathematics}, vol.~27, no.~3, pp. 579--585, 1968.

\bibitem{FN69}
R.~Fraser and S.~Nadler, ``Sequences of contractive maps and fixed points,''
  \emph{Pacific Journal of Mathematics}, vol.~31, no.~3, pp. 659--667, 1969.

\bibitem{B65}
F.~E. Browder, ``Fixed-point theorems for noncompact mappings in hilbert
  space,'' \emph{Proceedings of the National Academy of Sciences}, vol.~53,
  no.~6, pp. 1272--1276, 1965.

\bibitem{B67}
------, ``Convergence theorems for sequences of nonlinear operators in banach
  spaces,'' \emph{Mathematische Zeitschrift}, vol. 100, no.~3, pp. 201--225,
  1967.

\bibitem{H67}
B.~Halpern, ``Fixed points of nonexpanding maps,'' \emph{Bulletin of the
  American Mathematical Society}, vol.~73, no.~6, pp. 957--961, 1967.

\bibitem{B96}
\BIBentryALTinterwordspacing
H.~H. Bauschke, ``The approximation of fixed points of compositions of
  nonexpansive mappings in hilbert space,'' \emph{Journal of Mathematical
  Analysis and Applications}, vol. 202, no.~1, pp. 150 -- 159, 1996. [Online].
  Available:
  \url{http://www.sciencedirect.com/science/article/pii/S0022247X9690308X}
\BIBentrySTDinterwordspacing

\bibitem{J05}
\BIBentryALTinterwordspacing
J.~S. Jung, ``Iterative approaches to common fixed points of nonexpansive
  mappings in banach spaces,'' \emph{Journal of Mathematical Analysis and
  Applications}, vol. 302, no.~2, pp. 509 -- 520, 2005. [Online]. Available:
  \url{http://www.sciencedirect.com/science/article/pii/S0022247X04006717}
\BIBentrySTDinterwordspacing

\bibitem{Gill91}
J.~Gill, ``The use of the sequence $f_n(z) = f_n\circ\cdots\circ f_1(z)$ in
  computing fixed points of continuous fractions, products and series,''
  \emph{Applied Numerical Mathematics}, vol.~8, pp. 469--476, 1991.

\bibitem{Lorentzen90}
L.~Lorentzen, ``Compositions of contractions,'' \emph{J. Computational and
  Applied Mathematics}, vol.~32, pp. 169--178, 1990.

\bibitem{LBF04}
Z.~Lin, M.~Broucke, and B.~Francis, ``Local control strategies for groups of
  mobile autonomous agents,'' \emph{IEEE Transactions on automatic control},
  vol.~49, no.~4, pp. 622--629, 2004.

\bibitem{QWH05}
Z.~Qu, J.~Wang, and R.~A. Hull, ``Products of row stochastic matrices and their
  applications to cooperative control for autonomous mobile robots,'' in
  \emph{Proc. American Control Conference}, Portland, OR, USA, June 8-10 2005,
  pp. 1066--1071.

\bibitem{TJ10}
A.~Tahbaz-Salehi and A.~Jadbabaie, ``Consensus over ergodic stationary graph
  processes,'' \emph{IEEE Transactions on automatic Control}, vol.~55, no.~1,
  pp. 225--230, 2010.

\bibitem{CXL16}
Y.~Chen, W.~Xiong, and F.~Li, ``Convergence of infinite products of stochastic
  matrices: a graphical decomposition criterion,'' \emph{IEEE Transactions on
  Automatic Control}, vol.~61, no.~11, pp. 3599--3605, 2016.

\bibitem{S81}
E.~Seneta, \emph{{Non-negative Matrices and Markov Chains}}, 2nd~ed.\hskip 1em
  plus 0.5em minus 0.4em\relax Springer, 1981.

\bibitem{CS77}
S.~Chatterjee and E.~Seneta, ``Towards consensus: Some convergence theorems on
  repeated averaging,'' \emph{Journal of Applied Probability}, vol.~14, no.~1,
  pp. 89--97, 1977.

\bibitem{H74}
D.~Hartfiel, ``On infinite products of nonnegative matrices,'' \emph{SIAM
  Journal on Applied Mathematics}, vol.~26, no.~2, pp. 297--301, 1974.

\bibitem{P66}
N.~Pullman, ``Infinite products of substochastic matrices,'' \emph{Pacific
  Journal of Mathematics}, vol.~16, no.~3, pp. 537--544, 1966.

\bibitem{L92}
A.~Leizarowitz, ``On infinite products of stochastic matrices,'' \emph{Linear
  algebra and its applications}, vol. 168, pp. 189--219, 1992.

\bibitem{W08}
C.~W. Wu, ``On some properties of contracting matrices,'' \emph{Linear Algebra
  and its Applications}, vol. 428, no. 11-12, pp. 2509--2523, 2008.

\bibitem{DL92}
\BIBentryALTinterwordspacing
I.~Daubechies and J.~C. Lagarias, ``Sets of matrices all infinite products of
  which converge,'' \emph{Linear Algebra and its Applications}, vol. 161, no.
  Supplement C, pp. 227 -- 263, 1992. [Online]. Available:
  \url{http://www.sciencedirect.com/science/article/pii/002437959290012Y}
\BIBentrySTDinterwordspacing

\bibitem{DL01}
\BIBentryALTinterwordspacing
------, ``Corrigendum/addendum to: Sets of matrices all infinite products of
  which converge,'' \emph{Linear Algebra and its Applications}, vol. 327,
  no.~1, pp. 69 -- 83, 2001. [Online]. Available:
  \url{http://www.sciencedirect.com/science/article/pii/S0024379500003141}
\BIBentrySTDinterwordspacing

\bibitem{Bert99}
D.~P. Bertsekas, \emph{{Nonlinear Programming}}, 2nd~ed.\hskip 1em plus 0.5em
  minus 0.4em\relax Athena Scientific, 1999.

\bibitem{Grif11}
C.~Griffin, ``Game theory: Penn state math 486 lecture notes (v 1.1.1),''
  http://www.personal.psu.edu/cxg286/Math486.pdf, 2014.

\bibitem{Z77}
W.~W. Zachary, ``An information flow model for conflict and fission in small
  groups,'' \emph{Journal of Anthropological Research}, vol.~33, no.~4, pp.
  452--473, 1977.

\bibitem{BDGG08}
U.~Brandes, D.~Delling, M.~Gaertler, R.~Gorke, M.~Hoefer, Z.~Nikoloski, and
  D.~Wagner, ``On modularity clustering,'' \emph{IEEE Transactions on Knowledge
  and Data Engineering}, vol.~20, no.~2, pp. 172--188, Feb 2008.

\bibitem{W16}
Wikipedia, ``Ice-bucket challenge,''
  \url{https://en.wikipedia.org/wiki/Ice_Bucket_Challenge}, Last Access: August
  8 2016.

\bibitem{BA99a}
A.~Barab\'{a}si, R.~Albert, and H.~Jeong, ``Mean-field theory for scale-free
  random graphs,'' \emph{Physica A}, vol. 272, no. 1/2, pp. 173--187, 1999.

\end{thebibliography}
\bibliographystyle{IEEEtran}

\end{document}